\documentclass[10pt]{article}

\usepackage{amsmath}
\usepackage{amssymb}
\usepackage{amsthm}
\usepackage{latexsym}
\usepackage{color}

\usepackage{dsfont}
\usepackage{graphicx}
\usepackage{todonotes}
\usepackage{appendix}
\usepackage{enumerate}
\usepackage[T1]{fontenc}
\usepackage[utf8]{inputenc}
\usepackage[english]{babel}
\usepackage{comment}
\usepackage{hyperref}
\DeclareUnicodeCharacter{014D}{\=o}
\usepackage{geometry}
\definecolor{red}{rgb}{0.7,0.15,0.15}
\definecolor{green}{rgb}{0,0.5,0}
\definecolor{blue}{rgb}{0,0,0.7}
\hypersetup{colorlinks, linkcolor={red},citecolor={green}, urlcolor={blue}}

\newtheorem{theorem}{Theorem}[section]
\newtheorem{definition}[theorem]{Definition}
\newtheorem{proposition}[theorem]{Proposition}
\newtheorem{assumption}[theorem]{Assumption}
\newtheorem{lemma}[theorem]{Lemma}
\newtheorem{remark}[theorem]{Remark}
\newtheorem{example}[theorem]{Example}

\newcommand{\la}{\langle}
\newcommand{\ra}{\rangle}
\def\ch{\textsc{h}}
\def\dis{\displaystyle}
\def\a{\alpha}
\def\b{\beta}
\def\g{\gamma}

\def\e{\varepsilon}

\def\l{\lambda}

\def\si{\sigma}
\def\t{\tau}
\def\f{\varphi}
\def\th{\theta}
\def\o{\omega}

\def\G{\Gamma}
\def\D{\Delta}
\def\Th{\Theta}

\def\F{\Phi}
\def\O{\Omega}

\def\no{\noindent}

\def\ms{\medskip}

\def\q{\quad}

\def\pa{\partial}
\def\cd{\cdot}
\def\cds{\cdots}

\def\cA{{\cal A}}

\def\cF{{\cal F}}

\def\cJ{{\cal J}}

\def\cP{{\cal P}}

\def\cS{{\cal S}}

\def\dbE{\mathbb{E}}
\def\dbF{\mathbb{F}}

\def\dbL{\mathbb{L}}

\def\dbN{\mathbb{N}}
\def\dbP{\mathbb{P}}
\def\dbR{\mathbb{R}}
\def\dbS{\mathbb{S}}

\def \D{{\bf{D}}}
\def \E{\mathbb{E}}
\def \F{\mathbb{F}}
\def \H{\mathbb{H}}

\def \P{\mathbb{P}}

\def \R{\mathbb{R}}

\def \G{\mathbb{G}}

\newcommand{\ol}{\overline}
\newcommand{\ul}{\underline}

\def\Ac{{\cal A}}

\def\Fc{{\cal F}}
\def\Gc{{\cal G}}

\def\Mc{{\cal M}}
\def\Nc{{\cal N}}

\def\Pc{{\cal P}}

\def\no{\noindent}

\def\={\;=\;}
\def\.{\;.}

\def\eps{\varepsilon}

\def\reff#1{{\rm(\ref{#1})}}

\def\1{{\bf 1}}

\def\eps{\epsilon}

\def\b*{\begin{eqnarray*}}
\def\e*{\end{eqnarray*}}

\title{\bf Zero--sum path--dependent stochastic differential games in weak formulation}

\author{Dylan {\sc Possama\"i}\footnote{Columbia University, IEOR, 500W 120th St., 10027 New York, NY, dp2917@columbia.edu. Research supported by the ANR project PACMAN ANR-16-CE05-0027.}   \and
Nizar {\sc Touzi}\footnote{CMAP, Ecole Polytechnique Paris, nizar.touzi@polytechnique.edu. Research supported by ANR the Chair {\it Financial Risks} of the {\it Risk Foundation} sponsored by Soci\'et\'e G\'en\'erale, and the Chair {\it Finance and Sustainable Development} sponsored by EDF and Calyon. }
       \and Jianfeng {\sc Zhang}\footnote{University of Southern California, Department of Mathematics, jianfenz@usc.edu. Research supported in part by NSF grant DMS 1413717.}}

\begin{document}

\maketitle

\begin{abstract}
We consider zero sum stochastic differential games with possibly path--dependent controlled state. Unlike the previous literature, we allow for weak solutions of the state equation so that the players' controls are automatically of feedback type. Under some restrictions, needed for the {\it a priori} regularity of the upper and lower value functions of the game, we show that the game value exists when both the appropriate path--dependent Isaacs condition, and the uniqueness of viscosity solutions of the corresponding {path--dependent} Isaacs--HJB equation hold. We also provide a general verification argument and a characterisation of saddle--points by means of an appropriate notion of second--order backward SDE.

\vspace{0.8em}

\vspace{5mm}


\end{abstract}

\section{Introduction}
\label{sect:introduction}
\setcounter{equation}{0}

Stochastic differential games have attracted important attention during the last three decades. Due to the crucial role of the information structure, the corresponding literature is technically and conceptually more involved than standard stochastic control. It had been recognised as early as in the 60s, in a series of papers by Varaiya \cite{varaiya1967existence}, Roxin \cite{roxin1969axiomatic} and Elliot and Kalton \cite{elliot1972existence}, in the context of deterministic differential games, that having both players play a classical control generally led to ill--posed problems, and that the appropriate notion was rather that of a strategy, that is to say that a given player uses a non--anticipative map from the other player's set of controls to his own set of controls. Earlier definitions of value functions for games require appropriate approximations procedures, by discrete--time games in Fleming's definition \cite{fleming1957note,fleming1961convergence,fleming1964convergence}, or by discretising the players's actions in Friedman's definition \cite{friedman1970definition,friedman1971differential} (see also Varaiya and Lin \cite{varaiya1969existence} for an earlier related notion). This makes the whole approach technically cumbersome. 

\vspace{0.5em}

The connexion between the value function of the game and the corresponding Hamilton--Jacobi--Isaacs partial differential equation was formally established by Isaacs \cite{isaacs1954differential,isaacs1965differential} in the 50s. The Elliot--Kalton definition induces an easy argument to prove rigorously this connexion by using the notion of viscosity solutions, see Evans and Souganidis \cite{evans1984differential2}, as well as the generalisation by Evans and Ishii \cite{evans1984differential}\footnote{Notice that value functions in Friedman's and Fleming's definitions were also proved to be related to the HJI PDE by Souganidis \cite{souganidis1983approximation} and Barron, Evans and Jensen \cite{barron1984viscosity}.}. The Elliot--Kalton strategies have been successfully generalised to the context of stochastic differential games by Fleming and Souganidis \cite{fleming1989existence}, where the above mentioned strategy map is restricted to adapted controls. Despite the asymmetry between the two players of the induced game problem formulation, this approach has been followed by an important strand of the literature in continuous--time stochastic differential games, see notably the revisits of Buckdahn and Li \cite{buckdahn2008stochastic} or Fleming and Hern\'andez-Hern\'andez \cite{fleming2011value}. 

\vspace{0.5em}

{\color{black}
This approach has some important drawbacks, however.  Besides some stemming from practical considerations, see Remark \ref{rem-fleming1989existence1}, the asymmetry between the players 
makes the problem of existence of saddle--points much harder in general. This justified the recent emergence of several alternative formulations of the game. While the corresponding results may look similar at first sight, these reformulations have very subtle differences. We shall devote Section \ref{sect:formulation} completely to
 an incremental presentation of the different formulations which appeared in the literature, with appropriate examples highlighting the main differences. In particular, when the diffusion coefficient is not controlled by any of the players, the problem is completely addressed in Hamad\`ene and Lepeltier \cite{hamadene1995zero}, see Subsection \ref{sect:hamadene1995zero}. However, when the diffusion coefficient is also controlled, all the results in the literature require the control/strategy to be simple in some sense. This constraint makes it essentially impossible to obtain the existence of saddle points under those formulations. 
}

\vspace{0.5em}

The main contribution of this paper is to show that considering stochastic differential games in weak formulation allows to bypass major difficulties pointed out in the previous literature. In our setting, introduced in Section \ref{sect:setting}, the controlled state process is a weak solution of the possibly path--dependent stochastic differential equation
 \[
 \mathrm{d}X^{\alpha}_t
 =
 b_t\big(X^{\alpha},\alpha^0_t,\alpha^1_t\big) \mathrm{d}t
 + \sigma_t\big(X^{\alpha},\alpha^0_t,\alpha^1_t\big) \mathrm{d}W_t,
 \]
where $W$ is a Brownian motion with appropriate dimension, $b$ and $\sigma$ are non--anticipating functions of the path, and $\alpha=(\alpha^0,\alpha^1)$ is the pair of controls of Players $0$ and $1$, respectively. We consider the largest set of controls $\alpha^i:[0,T]\times C^0([0,T])\longmapsto A^i$, $i=0,1$, by only assuming the natural non--anticipativity and measurability properties. In particular, we do not impose that they are simple in some sense so as to guarantee existence of a strong solution for the above state equation. Again, our approach is to consider weak solutions, without requiring uniqueness of such a solution.

\vspace{0.5em}
Our first main result, reported in Theorem \ref{thm-gamevalue}, states that, under the path--dependent Isaacs condition, uniqueness of viscosity solutions implies existence of the game value. Section \ref{sect-Vt} contains the technical arguments to prove this result, following the dynamic programming arguments as in Pham and Zhang \cite{pham2014two}. Our proof relies on the notion of path--dependent viscosity solutions, introduced by Ekren, Touzi and Zhang \cite{ekren2016viscosity, ekren2012viscosity}. Observe that this result covers the Markovian setting under the uniqueness condition of viscosity solutions in the standard sense of Crandall and Lions \cite{crandall1983viscosity}, as our set of test functions includes theirs. As our technique requires some {\it a priori} regularity for the game upper and lower values, Theorem \ref{thm-gamevalue} is established under restricting conditions on the coefficients $b$ and $\sigma$ which are essentially summarised in Assumption \ref{assum-bsi}, see also Section \ref{sect-extension} for a slight weakening of these conditions. Notice that the remarkable work of S\^irbu \cite{sirbu2014stochastic,sirbu2015asymptotic} does not need any such restrictions, as the Perron--like method this author uses allows to bypass the task of deriving directly the dynamic programming principle. However, the method is restricted to the Markovian setting, and the players controls are simple and thus much less general than the ones we consider here.

\vspace{0.5em}
As a second main result reported in Theorem \ref{thm-verification}, we provide a verification argument, still under the path-dependent Isaacs condition, including a characterisation of saddle--points. We emphasise that when the volatility of the diffusion is degenerate, this result is new, even in the special Markovian setting, as the value function of the game may fail to lie in the standard Sobolev spaces, due to possible non--existence of a density of the corresponding state equation. By further considering a convenient relaxation, we also provide in Theorem \ref{thm:2bsde} a characterisation by means of an appropriate notion of second--order backward SDE, which plays the same role as the Sobolev--type solution for the corresponding Hamilton--Jacobi-Bellman--Isaacs (HJBI for short) partial differential equation.

\vspace{0.5em}
\noindent {\bf Notations:} Throughout the paper, for $i=0, 1$, we assume that the control of Player $i$ takes values in $A_i\subset \dbR^{d_i}$, for some arbitrary integer $d_i$. We define $A := A_1 \times A_2$, and denote typically the elements of $A$ as $a = (a_0, a_1)$.  Throughout this paper, for every $p-$dimensional vector $b$ with $p\in \mathbb{N}$, we denote by $b^{1},\ldots,b^{p}$ its entries, for $1\leq i\leq p$. For $\alpha,\beta \in \R^p$ we denote by $\alpha\cdot \beta$ the usual inner product, with associated norm $|\cdot|$. For any $(\ell,c)\in\mathbb N \times\mathbb N $, $\mathcal M_{\ell,c}(\mathbb R)$ denotes the space of $\ell\times c$ matrices with real entries. The elements of  matrix $M\in\mathcal M_{\ell,c}$ are denoted  $(M^{i,j})_{1\leq i\leq \ell,\ 1\leq j\leq c}$, and the transpose of $M$ is denoted by $M^\top$. We identify $\mathcal M_{\ell,1}$ with $\R^\ell$. When $\ell=c$, we let $\mathcal M_{\ell}(\mathbb R):=\mathcal M_{\ell,\ell}(\mathbb R)$. We also denote by $\dbS^\ell$ (resp. $\dbS^\ell_+$) 
the set of symmetric (resp. symmetric semi--definite positive) matrices in $\Mc_{\ell}(\R)$. The trace of a matrix $M\in\mathcal M_\ell(\R)$ will be denoted by ${\rm Tr}[M]$. {\color{black}For further reference, we list here all the filtrations that will be used throughout the paper. For any filtration $\mathbb G:=(\Gc_t)_{0\leq t\leq T}$, and for any probability measure $\P$ on our space $(\O,\Fc)$, we denote by $\G^\P:=(\Gc^\P_t)_{0\leq t\leq T}$ the usual $\P-$augmentation\footnote{\color{black}The $\P-$augmentation is defined for any $t\in[0,T]$ by $\Gc^\P_t:=\sigma\big(\Gc_t\cup \mathcal N^\P\big)$, where $\mathcal N^\P:=\big\{A\subset\Omega,\; A\subset B,\; \text{with}\; B\in\Fc_T,\; \P[B]=0\big\}.$} of $\G$, and by $\G^+:=(\Gc_t^+)_{0\leq t\leq T}$ the right--limit of $\G$. Similarly, the right limit of $\G^\P$ will be denoted by $\G^{\P+}:=(\Gc^{\P+}_t)_{0\leq t\leq T}$. For technical reasons, we also need to introduce the universal filtration $\G^{U} := \big(\Gc^{U}_t \big)_{0 \le t \le T}$ defined by $\mathcal G^{U}_t:=\cap_{\P \in \mbox{\tiny Prob}(\Omega)}\Gc_t^{\P}$, $t\in[0,T]$, where ${\rm Prob}(\O)$ is the set of all probability measures on $(\O,\Fc)$, }and we denote by {\color{black}$\G^{U+}$}, the corresponding right--continuous limit. Moreover, for a subset $\Pc\subset \mbox{Prob}(\Omega)$, we introduce the set of $\Pc-$polar sets $\mathcal N^{\Pc}:=\big\{N\subset\Omega: N\subset A$ for some $A\in\mathcal F_T$ with $\sup_{\P\in\Pc}\P(A)=0\big\}$, and we introduce the $\Pc-$completion of $\G$, $\G^{\Pc}:=\left(\Gc^{\Pc}_t\right)_{t\in[0,T]},$ with $\Gc^{\Pc}_t:=\Gc^{U}_t\vee\sigma\left(\Nc^{\Pc}\right),$ $t\in[0,T],$ together with the corresponding right--continuous limit $\G^{\Pc+}$.

\section{Stochastic differential game formulations and examples}
\label{sect:formulation}
\setcounter{equation}{0}

In this section we introduce the main formulations of zero--sum stochastic differential games from the existing literature, and explain through several examples why we have chosen to concentrate our attention on the "weak formulation with control against control".  {\color{black}The section is somewhat lengthy. However, due to the subtleties involved in the  formulations, we think such a detailed introduction will prove helpful for our readers.}

\subsection{Strong formulation with control against control}
\label{sect:strong}
Fix some time horizon $T>0$. In the strong formulation paradigm, a filtered probability space $(\O, \mathcal F,\dbF:=(\mathcal F_t)_{0\leq t\leq T}, \dbP_0)$, on which is defined a $d-$dimensional Brownian motion $W$, is fixed. We denote by $\mathbb F^W$ the natural filtration of $W$, augmented under $\mathbb P_0$, and for $i=0,1$, we let $\cA^i_{\rm S}$ denote the set of $\dbF^W-$progressively measurable $A_i-$valued processes, and  $\cA_{\rm S}:= \cA^0_{\rm S}\times \cA^1_{\rm S}$.  Throughout the paper, we take the notational convention that we write $i$ as subscript for deterministic objects and as superscript for random objects. Consider then, for $i=0,1$ the following $n-$dimensional controlled state processes with controls $\a := (\a^0,\a^1)\in \cA_{\rm S}$
\begin{equation}
\label{Strong-X}
X^{\a}_t := \int_0^t b(s, X^\a_s, \a_s)  \mathrm{d}s+ \int_0^t \si(s, X^\a_s, \a_s)  \mathrm{d}W_s,\; t\in[0,T],\; \dbP_0-\mbox{a.s.}
\end{equation}
where $b:[0,T]\times\mathbb R^n\times A\longrightarrow\mathbb R^n$ and $\sigma: [0,T]\times\mathbb R^n\times A\longrightarrow\mathbb R^{n\times d}$ are appropriate Borel measurable functions so that the above SDE has a unique strong solution for any $\a \in \cA_{\rm S}$.
We introduce the so--called upper and lower values of the game
\begin{eqnarray}
\label{Strong-V0}
\overline{V}^{\rm S}_0 := \inf_{\a^0\in \Ac_{\rm S}^0} \sup_{\a^1\in \cA_{\rm S}^1} J_{\rm S}(\a^0,\a^1),
&\mbox{and}&
\underline{V}^{\rm S}_0 := \sup_{\a^1\in \cA_{\rm S}^1} \inf_{\a^0\in \cA_{\rm S}^0}  J_{\rm S}(\a^0,\a^1),
\end{eqnarray}
where the criterion of the players $J_{\rm S}$ is defined, for some appropriate functions $f: [0,T]\times\mathbb R^n\times A\longrightarrow\mathbb R$ and $g:\mathbb R^n\longrightarrow\mathbb R$, by
\begin{equation}
\label{Strong-J}
J_{\rm S}(\a) := \dbE^{\dbP_0}\bigg[g(X^\a_T) + \int_0^T f(t, X^\a_t, \a_t) \mathrm{d}t\bigg].
\end{equation}
It is clear by definition that $\ul V^{\rm S}_0 \le \ol V^{\rm S}_0$. There are two central problems for the game defined above:
\begin{itemize}
\item[$(i)$] Does the game value exists, namely $\ol V^{\rm S}_0 = \ul V^{\rm S}_0$?

\item[$(ii)$] Is there a saddle--point (also called equilibrium) for the game? That is to say, can we find some $\widehat\a:= (\widehat \a^0, \widehat \a^1) \in  \cA_{\rm S}$ such that
\begin{equation}
\label{Strong-equilibrium}
 J_{\rm S}(\widehat \a^0, \a^1) \le   J_{\rm S}(\widehat \a^0, \widehat \a^1) \le J_{\rm S}(\a^0, \widehat \a^1),~\mbox{for any}~ \a^0\in \cA_{\rm S}^0,\; \a^1 \in \cA_{\rm S}^1.
\end{equation}
\end{itemize}
Notice immediately that the existence of a saddle--point $\widehat \a$ implies automatically that the game value exists, and is equal to $J_S(\widehat \a)$.
 
\vspace{0.5em}

Despite the fact that the above formulation is very close to the usual framework of stochastic control, it has never been considered in the literature, since even in seemingly benign situations, the game value may fail to exist.

\begin{example}
\label{eg-strong}{\rm This is a simplified version of an example borrowed from R. Buckdahn, see \cite[Appendix E]{pham2014two}.

\vspace{0.5em}
\noindent Let $A_0=A_1 = [-1,1]$, $d=n=2$, and $c\in\dbR$, $\rho \in [-1, 1]$  be two constants. Consider the following specification
\[f:=0,\; g(x):=\big|x_1-x_2\big|^2,\; b(t,x,a):=\begin{pmatrix}a_0\\ a_1\end{pmatrix},\; \sigma(t,x,a):=\begin{pmatrix}c & 0\\ c\rho & c\sqrt{1-\rho^2}\end{pmatrix}.\]
In this case, we have
\begin{equation*}
 X^{1,\a}_t := \int_0^t \a^0_s  \mathrm{d}s + c W^1_t,\; X^{2,\a}_t := \int_0^t \a^1_s  \mathrm{d}s + c \big[\rho W^1_t + \sqrt{1-\rho^2} W^2_t\big],\; J_{\rm S}(\a) := \dbE^{\dbP_0}\Big[|X^{1,\a}_T-X^{2,\a}_T|^2\Big].
\end{equation*}
Then, we claim that 
 \begin{eqnarray*}
 \ul V^{\rm S}_0 \;\le\; 2(1-\rho)c^2 T 
 &\mbox{and}& 
 T^2 \;\le\; \ol V^{\rm S}_0.
 \end{eqnarray*}
so that $\ul V^{\rm S}_0 < \ol V^{\rm S}_0$ whenever $2(1-\rho)c^2  < T$, and the game does not have a value in this formulation. To see this, notice that for any $\a^1 \in \cA_S^1$, if Player $0$ also plays the control $\alpha^1$, we have
\[
J_{\rm S}(\a^1, \a^1) = \dbE^{\dbP_0}\Big[\big|c(1-\rho) W^1_T - c\sqrt{1-\rho^2} W^2_T\big|^2\Big]  = 2(1-\rho)c^2T. 
\]
Thus $\inf_{\a^0\in \cA_{\rm S}^0} J_{\rm S}(\a^0, \a^1) \le 2(1-\rho)c^2T$, so that by arbitrariness of $\a^1\in \mathcal A_{\rm S}^1$, we have $\ul V^S_0\le 2(1-\rho)c^2T$. On the other hand, for any $\a^0\in \cA_{\rm S}^0$, set 
\[x_0 :=  \dbE^{\dbP_0}\bigg[\int_0^T \a^0_sds\bigg], \; {\rm sgn}(x_0) :=  \1_{\{x_0\ge 0\}} - \1_{\{x_0< 0\}}\in A_1,\; \a^1_t := - {\rm sgn}(x_0), \; t\in [0,T].\]
 Then by Jensen's inequality
\begin{align*}
 J_{\rm S}(\a^0, \a^1)  
 \ge  \Big|\dbE^{\dbP_0}\big[X^{1,\a}_T-X^{2,\a}_T\big]\Big|^2&= 
 \bigg|\dbE^{\dbP_0}\bigg[\int_0^T \a^0_s  \mathrm{d}s -\int_0^T\a^1_sds\bigg]\bigg|^2=
 \big|x_0 + T{\rm sgn}(x_0)\big|^2 \ge |T{\rm sgn}(x_0)|^2 = T^2.
\end{align*}
This implies that $\sup_{\a^1\in\cA_{\rm S}^1} J_{\rm S}(\a^0, \a^1)\ge T^2$ for any $\a^0\in\cA_{\rm S}^0$, and thus  $\ol V^{\rm S}_0\ge T^2$.\qed
 }
\end{example}

\vspace{0.5em}

{\color{black}
We recall that  the zero--sum game \reff{Strong-X}--\reff{Strong-J} is closely related to the following HJBI PDEs
\begin{equation}
\label{PDE1}
- \pa_t \ol v - \ol H(t, x,  D \ol v, D^2\ol v) =0,\;  - \pa_t \ul v - \ul H(t, x,D \ul v, D^2 \ul v) =0,
\end{equation}
where the Hamiltonians  $\ol H, \ul H$ are defined as:
 \begin{equation}
\label{Hamiltonian1}
\left.\begin{array}{c} 
\dis h(t,x,z,\g, a) :=  {\frac12} {\rm Tr}\big[(\si\si^\top)(t,x,a) \g\big] + b(t,x,a) \cd z+ f(t,x, a), \; (t,x,z,\gamma,a)\in[0,T]\times\R^d\times\mathbb R^d\times\mathbb S^d,\\[0.5em]
\dis \ol H(t,x,z,\g) := \inf_{a_0\in A_0} \sup_{a_1\in A_1} h(t,x,z,\g, a_0, a_1), \;
\ul H(t,x,z,\g) := \sup_{a_1\in A_1}\inf_{a_0\in A_0}  h(t,x,z,\g, a_0, a_1).
\end{array}\right.
\end{equation}
Moreover, the following Isaacs condition is crucial for the existence of the game value:
 \begin{equation}
\label{Isaacs1}
\ol H = \ul H =: H.
\end{equation}
Under the above condition, we say $(\hat a_0, \hat a_1)\in A$ is a saddle point of the Hamiltonian $H$ at $(t,x,z,\g)$ if
 \begin{equation}
\label{saddle1}
 h(t,x,z,\g, \hat a_0, a_1) \le H(t,x,z,\g) \le h(t,x,z,\g,  a_0,  \hat a_1)  \q\mbox{for all}\q a_0\in A_0, a_1\in A_1.
\end{equation}
}

\begin{remark}
\label{rem-novalue}
Direct calculation reveals that the Isaacs condition \reff{Isaacs1} holds in the  context of Example \ref{eg-strong}, with
\[
H(z,\g)  ={\frac{c^2}2} \big(\g_{11} + \g_{22} + 2\rho \g_{12}\big)+ |z_2|- |z_1|,\; (z,\g)\in \dbR^2\times \dbS^2.
\]
So the game value does not exist, despite the fact that Isaacs's condition holds. Notice as well that when $c=0$, this is a deterministic game, and when $c>0$ and $|\rho|<1$, $\si$ is non--degenerate. Thus potential degeneracy of the diffusion coefficient is not the reason for the non--existence of the game value.
\end{remark}

\subsection{Strong formulation with strategy against control}
\label{sect:fleming1989existence}

As we have seen above, naively considering games in a control against control formulation usually leads to non--existence of the game value. One way to properly formalise the fact that in continuous--time differential games the players also observe each other continuously consists in introducing the notion of non--anticipative strategies. Roughly speaking, in such a framework a strategy for one player is simply a non--anticipative map from the set of controls of the other player to the set of controls of this player. Though strategies were introduced in deterministic games by Varaiya \cite{varaiya1967existence}, Roxin \cite{roxin1969axiomatic} and Elliot and Kalton \cite{elliot1972existence}, the first work to extend this notion in a stochastic setting is due to Fleming and Souganidis \cite{fleming1989existence}. Let us now give a proper definition.
\begin{definition}
\label{fleming1989existence-strategy}
Let $\cS^0$ denote the set of mappings  ${\rm a}^0: \cA_S^1\longrightarrow \cA_S^0$ such that, for any $t\in [0,T]$, and any $(\a^1, \widetilde \a^1)\in \cA_S^1\times \Ac^1_S$ satisfying $\a^1 = \widetilde \a^1$, $ds\times d\dbP_0-${\rm a.s.} on $[0, t]\times \O$, we have ${\rm a}^0(\a^1) ={\rm a}^0(\widetilde \a^1)$, $ds\times d\dbP_0-${\rm a.s.} on $[0, t]\times \O$. Similarly we define $\cS^1$ as the set of appropriate mappings ${\rm a}^1: \cA_S^0\longrightarrow \cA_S^1$. 
\end{definition}

The upper and lower values in this formulation are then defined as
\begin{equation}
\label{fleming1989existence-V}
\ol V^{\rm FS}_0 := \sup_{{\rm a}^1 \in \cS^1} \inf_{\a^0 \in \cA_S^0}  J_{\rm S}(\a^0, {\rm a}^1(\a^0)),
\;
\ul V^{\rm FS}_0 := \inf_{{\rm a}^0 \in \cS^0} \sup_{\a^1 \in \cA_S^1}  J_{\rm S}({\rm a}^0(\a^1), \a^1).
\end{equation}
We emphasise that in this framework, the upper value is defined as a $\sup \inf$, rather than an $\inf \sup$. Besides, since the setting is by nature asymmetric, it is not {\it a priori} clear that 
\[ \ul V^{\rm FS}_0 \le  \ol V^{\rm FS}_0,\; \text{or}\;  \ol V^{\rm FS}_0 \le  \ul V^{\rm FS}_0.\]
Nevertheless, this formulation has been very successful in the existing literature because the game value is well understood, and characterised by the so--called Hamilton--Jacobi--Bellman--Isaacs PDE. The main result of Fleming and Souganidis \cite[Theorem 2.6]{fleming1989existence} is the following.
\begin{theorem}
\label{thm-fleming1989existence} 
Under appropriate technical conditions on the coefficients $b, \si, f, g$, we have  $\ol V^{\rm FS}_0 = \ol v(0,0)$, $\ul V^{\rm FS}_0 = \ul v(0,0)$, where $\ol v$, $\ul v$ are viscosity solutions of the corresponding {\rm HJBI} equations \reff{PDE1}. 
In particular, if Isaacs condition \reff{Isaacs1} holds, and the viscosity solution to the above {\rm PDEs} is unique, then $\ol V^{\rm FS}_0=\ul V^{\rm FS}_0$, and the game value exists.
\end{theorem}

Notice that the approach of Fleming and Souganidis \cite{fleming1989existence} has been substantially improved and simplified by Buckdahn and Li \cite{buckdahn2008stochastic} (see also the works of Bouchard, Moreau and Nutz \cite{bouchard2014stochastic} and Bouchard and Nutz \cite{bouchard2015stochastic} for a similar approach in stochastic target games), who considered a similar framework, allowing for controls depending on the full past of the trajectories of $W$ (implying in particular that their cost functionals become random variables), and also for more general running cost functionals in the form of backward SDEs. Though their framework remains Markovian, a recent extension to non--Markovian dynamics has been proposed by Zhang \cite{zhang2017existence,zhang2017existence2}, relying on top of the BSDE method of Buckdahn and Li, on an approximation of the non--Markovian game by sequences of standard Markovian games.

\vspace{0.5em}

While the above results are beautiful mathematically, it has two major drawbacks, as illustrated in the following two remarks.
\begin{remark}
\label{rem-fleming1989existence1}
The strategies are typically difficult to implement in practice. 

\vspace{0.5em}
{\rm (i)} In the problem $\ol V^{\rm FS}_0$, Player $1$ needs to observe the control $\a^0$ of Player $0$. But since this is a zero--sum game, the players typically would not tell their competitors their controls, due to the so--called moral hazard. 

\vspace{0.5em}
{\rm (ii)}  Notice further that  the strategy ${\rm a}^1$ is a function of the whole process $\a^1$, rather than the paths of $\a^1$. This imposes further difficulty for the practical implementation of non--anticipative strategies. {\color{black}Even in the full information case $($without moral hazard$)$}, the players do not actually observe their opponent's adapted control, but just a realisation of this control in the actual state of the world.
\end{remark}

\begin{remark}
\label{rem-fleming1989existence2}
The study of the existence of saddle--points in this setting also proves very difficult. Among some of the reasons, we would like to highlight the following.

\vspace{0.5em}
{\rm (i)} The information is asymmetric in this setting. As a consequence, it is not possible to define saddle--points as conveniently as in the spirit of \eqref{Strong-equilibrium}.

\vspace{0.5em}
{\rm (ii)}  The problem $\ol V^{\rm FS}_0$ can be viewed as a zero--sum Stackelberg game, which requires to solve sequential optimisation problems. Given ${\rm a}^1$, it will in general be difficult to solve $\inf_{\a^0\in \cA_S^0} J_S(\a^0, {\rm a}^1(\a^0))$, since as a general strategy there are not many properties we can impose on ${\rm a}^1$. The optimisation over ${\rm a}^1$ can then become even harder. 

\vspace{0.5em}
{\rm (iii)}  Moreover, we emphasise that this formulation is still in a strong setting, namely all involved processes are required to be $\dbF^W-$progressively measurable. In this case, the set $\cA_S$ of admissible controls is typically not compact, meaning that saddle--points are even less likely to exist.
\end{remark}

We illustrate the above points by considering two examples where saddle--points cannot exist, no matter how one defines them. For this purpose, we borrow a function $\zeta$ from Barlow \cite{barlow1982one} which satisfies the following properties
\begin{itemize}
\item[$(1)$] $\zeta: \dbR\longrightarrow [1, 2]$ and is uniformly H\"older continuous.
\item[$(2)$] The following SDE admits a unique (in law) weak solution but no strong solution
\begin{equation}\label{BarlowSDE}
X_t = \int_0^t \zeta(X_s)  \mathrm{d}W_s,\; \dbP_0-\mbox{a.s.}
\end{equation}
\end{itemize}
\begin{example}
\label{eg-fleming1989existence}
Set $A_0 := [1, 2], A_1 := \{0\}$ and $d=1$. Consider the following specification
\[b(t,x,a):=0,\; \sigma(t,x,a):=|a_0|,\; g(x):=|x|^2,\; f(t,x,a):=\big|\zeta(x)\big|^2-2a_0\zeta(x).\]
We then have
\begin{equation*}
X^{\a}_t = \int_0^t |\a^0_s|  \mathrm{d}W_s,\; \dbP_0-\mbox{a.s.},\; J_{\rm S}(\a) := \dbE^{\dbP_0}\bigg[ |X^\a_T|^2 - \int_0^T \big[ 2\a^0_t \zeta(X^\a_t) -  |\zeta(X^\a_t)|^2\big]  \mathrm{d}t\bigg].
\end{equation*}
Then $\ol V^{\rm FS}_0=\ul V^{\rm FS}_0$, but there is no saddle--point in any appropriate sense. To see this, observe that $\cA^1_S$ consists of only the constant process $0$, and thus $\cS^1$ also consists only of the trivial mapping ${\rm a}^1 =0$. Then it is clear that
\begin{equation}
\ol V^{\rm FS}_0 = \ul V^{\rm FS}_0 =\inf_{\a^0\in \cA^0_S} J_{\rm S}(\a^0,0).
\end{equation}
This is a standard optimal control problem, and we know its value is $v(0,0)$, where $v$ is the unique viscosity solution to the following {\rm HJB} equation
\begin{equation}
-\pa_t v - \inf_{a_0\in A_0} \Big\{{\frac12} |a_0|^2 \pa^2_{xx} v  - 2a_0 \zeta(x) + |\zeta(x)|^2\Big\} =0,\; v(T,x) = x^2.
\end{equation}
One can check straightforwardly that $v(t,x) = x^2$ is the classical solution to the above {\rm PDE}. In particular, uniqueness for the last {\rm HJB} equation follows from the standard verification argument, and this implies that
\begin{equation}
\ol V^{\rm FS}_0 = \ul V^{\rm FS}_0 =\sup_{\a^0\in \cA^0_{\rm S}} J_{\rm S}( \a^0, 0)  = v(0,0) = 0.
\end{equation}

Now assume the game has a saddle--point in some appropriate sense, which will be associated to a certain $\widehat a^0 \in \cA_S^0$ and $\widehat \a^1=0$. Then, denoting $\widehat X := X^{\widehat\a^0, 0}$
\begin{equation}
0 = J_S(\widehat\a^0,0) = \dbE^{\dbP_0} \bigg[\int_0^T \big[|\widehat\a^0_t|^2  -2\widehat\a^0_t \zeta(\widehat X_t) +|\zeta(\widehat X_t)|^2\big] dt\bigg] =  \dbE^{\dbP_0} \bigg[\int_0^T\big |\widehat\a^0_t -\zeta(\widehat X_t)\big|^2  \mathrm{d}t\bigg] .
\end{equation}
This implies that necessarily $\widehat \a^0 = \zeta(\hat X)$, $\dbP_0-$a.s. In other words, $\widehat X$ must satisfy
\begin{equation*}
\widehat X_t = \int_0^t |\widehat \a_s| dW_s = \int_0^t \zeta\big(\widehat X_s\big) \mathrm{d}W_s,\; t\in[0,T],\; \dbP_0-\mbox{a.s.}
\end{equation*}
By Barlow {\rm \cite{barlow1982one}}, the above {\rm SDE} has no strong solution, which contradicts with our assumption that $\widehat X = X^{ \widehat\a^0,0}$ is $\dbF^W-$measurable.
\qed
\end{example}
 
The last example may seem very special, since the game problem is actually reduced to a stochastic control problem. The following example shows that similar concerns appear in a genuine game problem.

\begin{example}
\label{eg-fleming1989existence2}
Set $A_0:=[1,2]$, $A_1:=[0,1]$. Consider the following specification
\[b(t,x,a):=0,\; \sigma(t,x,a):=|a_0|,\; g(x):=|x|^2,\; f(t,x,a):=\big|\bar\zeta(x)\big|^2-2a_0\bar\zeta(x),\; \bar\zeta := \sqrt{|\zeta|^2-1}.\]
Then Isaacs condition \reff{Isaacs1} holds and 
\begin{equation}
J_{\rm S}(\a) 
= 
\dbE^{\dbP_0}\bigg[ |X^\a_T|^2 - \int_0^T \big[ 2\a^0_t\bar\zeta(X^\a_t)  -  |\bar\zeta(X^\a_t)|^2\big]  \mathrm{d}t\bigg].
\end{equation}
 In this case, we still have $\ol V^{\rm FS}_0=\ul V^{\rm FS}_0 = v(0,0)$, where $v(t,x) = x^2+T-t$ is the unique classical solution to the following {\rm HJBI} equation
\begin{equation}
\label{egfleming1989existence2-Isaacs}
-\pa_t v 
- \inf_{a_0\in A_0}  \Big\{{\frac12} |a_0|^2 \pa^2_{xx} v - 2a_0\bar\zeta(x)+ |\bar\zeta(x)|^2 \Big\}
- \sup_{a_1\in A_1} \Big\{{\frac12} |a_1|^2 \pa^2_{xx} v  \Big\}  =0,
 \;
 v(T,x) = x^2.
\end{equation}
 Moreover, the $($unique$)$ saddle--point of the Hamiltonian   in the sense of \eqref{saddle1} is
\begin{equation}
\label{eg3-saddle}
\widehat a_0 = \bar \zeta(x),\; \widehat a_1 = 1,\; \mbox{and thus}\; |\widehat a|= \sqrt{|\widehat a_0|^2 + |\widehat a_1|^2} = \zeta.
\end{equation}
Consequently, any natural saddle--point for the game should correspond to these feedback controls. Unfortunately, similar to the previous example, no strong solution exists under these feedback controls. Since this formulation is in strong setting, all involved processes should be $\dbF^W-$measurable, so it is very unlikely that a saddle--point under any reasonable definition will exist for this example.
 \end{example}

\begin{remark}
\label{rem-eg3}
{ We emphasise that the feedback controls \eqref{eg3-saddle} can be obtained naturally from the Hamiltonian of the {\rm PDE} $($provided the {\rm PDE} has a classical solution$)$. In weak formulation of the game, which will be introduced soon and will be the main focus of this paper, the saddle--points of the  Hamiltonian indeed lead to the saddle points of the game. However, in strategy against control formulation, saddle--points of the Hamiltonian provide no clue on the possible saddle--points for the game. This is one of the main drawbacks of this formulation.
}
\end{remark}

\subsection{Strong formulation with symmetric delayed pathwise strategies}
\label{sect:CR}

Recall that the setting in the previous subsection is not symmetric. Cardaliaguet and Rainer \cite{cardaliaguet2013pathwise} have reformulated the game problem by using what they call non--anticipative strategies with delay, thus, formalising the fact that the players only observe their opponent’s action in the actual state, as well as the path of the resulting solution of the stochastic differential equation. Their framework could thus be coined as "strategy against strategy".

\begin{definition}
\label{fleming1989existence-delayed}
Let $\cS^0_{\rm CR}$ be the collection of progressively measurable $($deterministic$)$ maps ${\rm a}^0: C^0([0, T], \dbR^d)\times \dbL^0([0, T], A_1) \longrightarrow \dbL^0([0, T], A_0)$ satisfying the following delay condition
\begin{equation}
[{\rm a}^0(\o, \a^1)](t) = \big[{\rm a}^0( \o_{(t-\delta)^+ \wedge \cd}, \a^1_{(t-\delta)^+\wedge \cd})\big](t),\; 0\le t\le T,
\end{equation}
 for some $\delta>0$ independent of $(\o, \a^1)$, and where for a generic set $E$, $\dbL^0([0, T], E)$ is the set of Borel measurable $E-$valued maps. We define similarly $\cS^1_{\rm CR}$. 
\end{definition}
For delayed strategies, the following simple but crucial result holds.
\begin{lemma}
\label{lem-delay}
For any $({\rm a}^0, {\rm a}^1)\in \cS^0_{\rm CR}\times \cS^1_{\rm CR}$, there exists unique $(\a^0, \a^1)\in \cA^0_{\rm S}\times \cA^1_{\rm S}$ such that
\begin{equation}
\label{tha}
\a^0(\o) = {\rm a}^0(\o, \a^1(\o)),\; \a^1(\o) = {\rm a}^1(\o, \a^0(\o)),\; \mbox{for all}~\o\in\O
\end{equation}
\end{lemma}
Then Cardaliaguet and Rainer \cite{cardaliaguet2013pathwise} define upper and lower values of the game as
\begin{equation}
\label{CR-V0}
\ol V^{\rm CR}_0 := \inf_{{\rm a}^0\in \cS^0_{\rm CR}} \sup_{{\rm a}^1\in \cS^1_{\rm CR}} J_{\rm S}(\a^0, \a^1),
\; 
\ul V^{\rm CR}_0 :=\sup_{{\rm a}^1\in \cS^1_{\rm CR}}  \inf_{{\rm a}^0\in \cS^0_{\rm CR}} J_{\rm S}(\a^0, \a^1),
\end{equation}
where $(\a^0, \a^1)$ are determined by Lemma \ref{lem-delay}. We emphasise that the mapping from $({\rm a}^0, {\rm a}^1)$ to $(\a^0,\a^1)$ is in pairs, and it does not necessarily induce a mapping from ${\rm a}^0$ to $\a^0$ (or from ${\rm a}^1$ to $\a^1$). Consequently, the game values $(\ol V^{\rm CR}_0, \ul V^{\rm CR}_0)$ are different from the values $(\ol V^{\rm S}_0, \ul V^{\rm S}_0)$ in \eqref{Strong-V0}. We also notice that, unlike in \eqref{fleming1989existence-V} the upper value is defined an $\inf \sup$, since the setting is symmetric again. The main result of \cite{cardaliaguet2013pathwise} is the following.

\begin{theorem}
\label{thm-cardaliaguet2013pathwise}
Under appropriate conditions, including the Isaacs condition, we have $\ol V^{\rm CR}_0 = \ul V^{\rm CR}_0 =  v(0,0)$, where $v$ is the unique viscosity solution to the corresponding {\rm HJBI} equation.
\end{theorem}
In particular, under the above conditions the game values in \cite{fleming1989existence} and \cite{cardaliaguet2013pathwise} are equal. Notice as well that without Isaacs condition, \cite{cardaliaguet2013pathwise} establishes a partial comparison principle, implying that $\ol v$ and $\ul v$ are only viscosity semi--solutions, not necessarily viscosity solutions of the associated HJBI PDE. We remark in addition that this setting is symmetric and one can naturally define saddle--points $(\widehat {\rm a}^0, \widehat {\rm a}^1)$. However, the comments in Remark \ref{rem-fleming1989existence1} (i) and Remark \ref{rem-fleming1989existence2}  (ii), (iii) remain valid, and no existence result of saddle--points is available in general because of the delay restriction on the strategies. Notice also that Buckdahn, Cardaliaguet and Quincampoix \cite{buckdahn2011some} have adapted the BSDE method of Buckdahn and Li \cite{buckdahn2008stochastic} to the framework of non--anticipative strategies with delay.

\vspace{0.5em}
Another way of symmetrising the game problem has been proposed by Fleming and Hern\'andez-Hern\'andez \cite{fleming2011value}. Building upon the fact that Elliott--Kalton strategies are such that, for instance for the lower value the minimising player has an advantage in the information available to him at each time, they propose to restrict to so--called strictly progressively measurable strategies which make this advantage disappear. They then define a notion of approximate $\eps-$saddle--points, but cannot obtain the existence of a saddle--point in the sense we have considered so far.

\subsection{Strong formulation with symmetric feedback controls}
\label{sect:CRpham2014two}

The works closest to our present paper are Cardaliaguet and Rainer \cite{cardaliaguet2009stochastic}, Pham and Zhang \cite{pham2014two}, and S\^irbu \cite{sirbu2014stochastic, sirbu2015asymptotic}. Consider the following SDE with feedback controls  $\alpha:[0,T]\times C^0([0,T],\dbR^d)\longrightarrow A$
\begin{equation}\label{SDE-feedbacksimple}
X_t = X_0+\int_0^t \si(s, X_s, \a_s(X_\cd))  \mathrm{d}W_s + \int_0^t b(s, X_s, \a_s(X_\cd))  \mathrm{d}s,\; \mathbb P_0-\mbox{a.s.}
\end{equation}
Let $\cA_{\rm sp}$ denote the set of certain simple feedback controls (meaning controls which are constant or deterministic in between points of a partition of $[0,T]$ and/or $\Omega$), in particular so that the above SDE has a unique strong solution for every admissible control $\a$. We remark that the sets $\cA_{\rm sp}$ in \cite{cardaliaguet2009stochastic}, \cite{pham2014two}, \cite{sirbu2014stochastic, sirbu2015asymptotic} are not the same, differing mainly on whether some mixing is allowed in the strategies or not, and on whether stopping times are allowed in the time--discretisation. The upper and lower game values are then defined as
\begin{equation}
\ol V^{\rm sp}_0 
:= 
\inf_{\a^0\in \cA^0_{\rm sp}} \sup_{\a^1\in \cA^1_{\rm sp}} J_{\rm S}\big(\a^0(X), \a^1(X)\big),
\;
\ul V^{\rm sp}_0 :=\sup_{\a^1\in \cA^1_{\rm sp}}  \inf_{\a^0\in \cA^0_{\rm sp}} J_{\rm S}\big(\a^0(X_\cdot), \a^1(X_\cdot)\big),
\end{equation}
where, for $\a\in\cA_{\rm sp}$, $X$ is determined by \eqref{SDE-feedbacksimple}. Then we still have
\begin{theorem}
\label{thm-pham2014two}
Under appropriate conditions, including the Isaacs condition, we have $\ol V^{\rm sp}_0 = \ul V^{\rm sp}_0 =  v(0,0)$, where $v$ is the unique viscosity solution to the HJBI equation \eqref{PDE1}.
\end{theorem}

We observe that, in contrast with \cite{cardaliaguet2009stochastic} and \cite{sirbu2014stochastic, sirbu2015asymptotic}, the framework of \cite{pham2014two} allows for path--dependent dynamics, and in this case the Isaacs equation becomes a so--called path--dependent PDE, which will be the main tool in this paper. However, \cite{pham2014two} does not allow for $x-$dependence in the coefficients $b$ and $\sigma$, mainly for the purpose of proving the regularity of the value functions, an issue which becomes very subtle in the present feedback formulation. In the Markovian setting, this difficulty is remarkably by--passed in \cite{sirbu2014stochastic, sirbu2015asymptotic} by using the notion of stochastic viscosity solutions of Bayraktar and S\^irbu \cite{bayraktar2012stochastic, bayraktar2013stochastic}. 

\vspace{0.5em}
We also remark that for feedback controls, it is a lot more convenient to use the so--called weak formulation, under which the state process $X$ is fixed and the players control its distribution. We finally note that the restriction to simple feedback controls is a serious obstacle for obtaining a saddle--point. This is the main drawback that addressed in this paper.

\subsection{A complete result in the uncontrolled diffusion setting}
\label{sect:hamadene1995zero}

The case where the diffusion coefficient $\sigma$ is not controlled by any of the players has received a lot of attention since the inception of the study of stochastic differential games, as it allowed for a much simpler treatment. Hence, using PDE methods in a Markovian setting with feedback controls (though it is not clear whether strong or weak formulation is considered, see \cite[Another remark, p. 85]{friedman1972stochastic}), Friedman \cite{friedman1971differential,friedman1972stochastic} proved existence of an equilibrium for an $N-$players game, as well as existence of a value and a saddle--point for two--players zero--sum games. Using the martingale approach of Davis and Varaiya \cite{davis1973dynamic} for stochastic control problems, a version of the problem in weak formulation was then considered by Elliott \cite{elliott1976existence,elliott1977existence}, and Elliott and Davis \cite{elliott1981optimal}, allowing in particular for non--Markovian dynamics. Their approach was subsequently streamlined and simplified by Hamad\`ene and Lepeltier \cite{hamadene1995zero,hamadene1995backward} and Hamad\`ene, Lepeltier and Peng \cite{hamadene1997bsdes} using BSDEs methods. 

\vspace{0.5em}
Since their approach is close in spirit to the one we wish to follow, we dedicate this section to describing it. Consider the following drift--less SDE
\begin{equation}
\label{SDE-driftless}
X_t = \int_0^t \si_s(X_\cd)  \mathrm{d}W_s,\ \dbP_0-\mbox{a.s.}
\end{equation}
where $\si:[0,T]\times C^0([0,T],\dbR^d)\longrightarrow\dbS^d$ is progressively measurable, bounded, non--degenerate, and uniformly Lipschitz in $x$. Hence, the above SDE has a unique strong solution, and $X$ and $W$ generate the same filtration. We next introduce a progressively measurable bounded map $\lambda:[0,T]\times C^0([0,T],\dbR^d)\times A\longrightarrow\dbR^d$, together with the equivalent probability measures
 \begin{equation}
\frac{ \mathrm{d} \dbP^\a}{ \mathrm{d}\dbP_0}:= \exp\bigg(\int_0^T \lambda_t\big(X_\cd,\alpha_t(X_\cd)\big)\cdot  \mathrm{d}W_t - {\frac12} \int_0^T |\lambda_t\big(X_\cd,\alpha_t(X_\cd)\big)|^2 \mathrm{d}t\bigg)
 \mbox{ for all }
 \alpha \in \cA_{\rm HL},
 \end{equation}
 where $\cA_{\rm HL}$ is the set of  admissible controls, for which the above probability measure is well--defined. 
By Girsanov's theorem, the process $W^\a:=W-\int_0^\cd \lambda_t(X_\cd,\alpha_t(X_\cd)) \mathrm{d}t$ is an $\dbP^\a-$Brownian motion, and $(\dbP^\a, X)$ is a weak solution of the drift--controlled SDE
 \begin{equation}\label{SDE-drift}
 X_t =  \int_0^t b_s(X_\cd, \a_s(X_\cd))  \mathrm{d}s+ \int_0^t \si_s(X_\cd)  \mathrm{d}W^\a_s,\; \dbP^\a-\mbox{a.s., where}~ b := \si \l.
\end{equation}
The upper and lower values of the game are then defined by
$$
\ol V^{\rm HL}_0 := \inf_{\a^0\in \cA_{\rm HL}^0} \sup_{\a^1\in \cA_{\rm HL}^1} J_{\rm W}(\a^0,\a^1),
 \; 
 \ul V^{\rm HL}_0 := \sup_{\a^1\in \cA_{\rm HL}^1} \inf_{\a^0\in \cA_{\rm HL}^0}  J_{\rm W}(\a^0,\a^1),
 ~\mbox{where}~
 J_{\rm W}(\a) := \dbE^{\dbP^\a}\Big[\xi + \int_0^T f_t(, \a_t) \mathrm{d}t\Big],
 $$
 for some appropriate functions $f:[0,T]\times C^0([0,T],\dbR^d)\times A\longrightarrow \R$ and $\xi:C^0([0,T],\dbR^d)\longrightarrow\R$. We finally introduce 
 \begin{equation*}
 \ol F_t(x,z) := \inf_{a_0\in A_0} \sup_{a_1\in A_1} \big\{b_t(x,a)\cdot z + f_t(x,a) \big\}
 ~\mbox{and}~
 \ul F_t(x,z) := \sup_{a_1\in A_1} \inf_{a_0\in A_0} \big\{b_t(x,a)\cdot z + f_t(x,a) \big\},
 \end{equation*}
where as usual $a=(a_0,a_1)$. Notice that, by extending the Hamiltonians \eqref{Hamiltonian1} to the path dependent case in the obvious way, we have the correspondence 
\[\ol H=\ol F + \frac12{\rm Tr}[\sigma^2\g], \; \text{and}\; \ul H=\ul F + \frac12{\rm Tr}[\sigma^2\g].\]
The main result of \cite{hamadene1995zero} is the following.
\begin{theorem} 
\label{thm-hamadene1995zero}
Under appropriate conditions, including Isaacs's condition $\ol F=\ul F=:F$, we have $\ol V^{\rm HL}_0 = \ul V^{\rm HL}_0 =  Y_0$, where $(Y,Z)$ is the unique solution to the backward {\rm SDE}
 \begin{equation}
 \label{hamadene1995zero-BSDE}
 dY_t
 =
  -F_t(X_\cd,Z_t) \mathrm{d}t + Z_t\cdot  \mathrm{d}X_t,
  ~\mbox{and}~Y_T=\xi(X_\cdot),
 ~\dbP_0-{\rm a.s.}
 \end{equation}
Moreover, any saddle--point of $F$ induces a saddle--point of the game.
\end{theorem}

The following example shows that the framework of \cite{hamadene1995zero} allows to obtain a saddle--point of a game in typical situations where all the previous frameworks of this section fail.

\begin{example}
\label{eg-weak1}
Consider the setting in Example \ref{eg-strong} with $c=1, \rho=0$, except that $X = (X^1, X^2)$ should be viewed as weak solution of the following {\rm SDE}
\begin{equation*}
X^{1}_t = \int_0^t \a^0_s(X^{1}_\cd, X^{2}_\cd)  \mathrm{d}s +  W^{1,\a}_t,\; 
X^{2}_t = \int_0^t \a^1_s(X^{1}_\cd, X^{2}_\cd)  \mathrm{d}s + W^{2,\a}_t.
\end{equation*}
Note that, unlike in Example \ref{eg-strong}, here $X^{1}$ and $X^{2}$ depend on both $\a^0$ and $\a^1$. 
In this case, we have 
\begin{equation*}
\ol F_t(z) = \ul F_t(z) = F_t(z) 
=
\inf_{|a_0|\le 1} \{a_0 z_1\} + \sup_{|a_1|\le 1} \{a_1 z_2\}
=
 - |z_1| +|z_2|,
\end{equation*}
and the unique saddle--point of $F$ is given by
\begin{equation*}
\widehat a_0(z) := -{\rm sgn}(z_1),\; \widehat a_1(z) := {\rm sgn}(z_2).
\end{equation*}
Notice that since $\sigma$ is the identity matrix here, we have $X=W$ and the {\rm BSDE} \eqref{hamadene1995zero-BSDE} becomes
\begin{equation*}
Y_t = \big|X^1_T-X^2_T\big|^2 + \int_t^T \big(|Z^2_s|-|Z^1_s|\big)  \mathrm{d}s - \int_t^T Z_s \cd  \mathrm{d}X_s,\;\dbP_0-{\rm a.s.}
\end{equation*}
In fact, one may verify straightforwardly that the solution to the above {\rm BSDE} is 
\[Y_t := \big|X^1_t-X^2_t\big|^2+2(T-t), \; Z^1_t := 2\big(X^1_t-X^2_t\big),\; Z^2_t = 2\big(X^2_t - X^1_t\big).\]
Consequently, the game value is $Y_0 = 2T$ and the saddle--point of the game is given by
\[\widehat \a^0_t := -{\rm sgn}(Z^1_t) = - {\rm sgn}\big(X^1_t - X^2_t\big), ~\widehat \a^1_t := {\rm sgn}(Z^2_t) = {\rm sgn}\big(X^2_t-X^1_t\big).\]
\end{example}

Our objective of this paper  is to  extend Theorem \ref{thm-hamadene1995zero} to a controlled (possibly degenerate) diffusion framework. Again we shall use weak formulation and our main tool will be path dependent PDEs, whose semi--linear counterpart is exactly the backward SDE. 
In particular, our results will provide a complete characterisations of Examples \ref{eg-fleming1989existence} and \ref{eg-fleming1989existence2}, once reformulated in the setting of weak solutions, 
as well as the degenerate situation of Example \ref{eg-weak1} (namely $|\rho| = 1$). As will become clearer later, the assumptions on the coefficients we will require are slightly stronger than in the recent work of Zhang \cite{zhang2017existence,zhang2017existence2}, but the latter considers the strong formulation with strategy against control and thus cannot obtain any positive results towards existence of saddle--points. The weak formulation allows to do so, but at the price of more stringent assumptions on the coefficients.

{\color{black}
\section{Main results}
\label{sect:main}
\setcounter{equation}{0}

\subsection{Path dependent game in weak formulation}
\label{sect:setting}
}
The canonical space $\O:= \{\o\in C([0, T]; \dbR^d):\o_0=0\}$ is endowed with the $\dbL^\infty-$norm. The corresponding canonical process $X$ induces the natural filtration $\dbF$. The time space set 
$\Th := [0, T]\times \O$ is equipped with the pseudo-distance $d_\infty((t,\o), (t',\o')) := |t-t'| + \|\o_{\wedge t} - \o'_{\wedge t'}\|_\infty$.

The set of control processes $\ol\cA :=\ol \cA^0\times\ol \cA^1$ consists of all $\dbF-$progressively measurable $A-$valued processes, for some subset $A:=A_0\times A_1$ of a finite dimensional space. For all $\a\in  \ol\cA$, we denote by $\cP(\a)$ the set of weak solutions of the following path--dependent  SDE
 \begin{equation}
 \label{SDE}
X_t
= \int_0^t b_s(X_\cd, \alpha_s(X_\cd)) \mathrm{d}s+
\int_0^t \sigma_s(X_\cd, \alpha_s(X_\cd)) \mathrm{d}W_s,
~t\in[0,T],
 \end{equation}
where $b: \Th\times A \longrightarrow \dbR^d$, $\si: \Th\times A \longrightarrow \dbS^d$ satisfies the conditions in Assumption \ref{assum-coef} below. Here, for simplicity, we assume that $X$ and $W$ have the same dimension $d$ and $\si$ is symmetric, but the extension to the more general situation is straightfroward. Equivalently, any $\dbP\in \cP(\a)$ is a probability measure on the canonical space $\O$ which solves the following martingale problem  
 \begin{eqnarray}\label{cPa}
 M^\a_t := X_t - \int_0^t b_s(\a_s) \mathrm{d}s
 &\mbox{is a $\dbP-$martingale, with}& 
 \la M^\a\ra_t  =\int_0^t \si^2_s(\a_s)  \mathrm{d}s,~0\le t\le T,~\dbP-\mbox{a.s.}
 \end{eqnarray}
Here we take the notational convention that, by putting the time variable $s$ as subscript,  we mean $b_s, \si_s, \a_s$ depend on $X$, but we often omit $X$ itself for notational simplicity. Notice that in general, the set $\cP(\a)$ for an arbitrary $\a\in\overline \Ac$ may be empty. We thus introduce the following subset $\cA := \cA^0\times \cA^1$, where\footnote{The idea here is that if there existed $\a^{0,\star}\in\overline{\Ac}^0$ such that $\Pc(\alpha^{0,\star},\alpha^1)=\emptyset$ for any $\alpha^1\in\overline{\Ac}^1$, then obviously in the upper value Player $0$ will play $\alpha^{0,\star}$, since then whatever the choice of Player $1$ afterwards will lead to a value of $-\infty$. Similarly, if there existed $\a^{1,\star}\in\overline{\Ac}^1$ such that $\Pc(\alpha^{0},\alpha^{1,\star})=\emptyset$ for any $\alpha^0\in\overline{\Ac}^0$, in the lower value Player $1$ will play $\alpha^{1,\star}$, since then whatever the choice of Player $0$ afterwards will lead to a value of $+\infty$. Our restriction is here merely to prevent this obvious degeneracy.}
\begin{eqnarray}
\label{cA}
\cA^i \;:=\; \big \{\a^i\in \ol\cA^i: \cP(\a^i)\neq \emptyset\big\},
&\mbox{where}&
\cP(\a^i) \;:=\; \bigcup_{\a^{1-i}\in \ol\cA^{1-i}}  \cP(\a^0,\a^1),
~~i=0,1.
\end{eqnarray}
For an $\cF_T-$measurable random variable $\xi$ and an $\dbF-$progressively measurable $f: \Th \times A \longrightarrow \dbR$, we now define 
 \begin{equation}
 \label{J}
 J_0(\a, \dbP) := \dbE^\dbP\bigg[\xi + \int_0^T f_s(\a_s) \mathrm{d}s\bigg],
 ~~
 \ol J_0(\a) := \sup_{\dbP\in  \cP(\a)} J_0(\a, \dbP),
 ~~\mbox{and}~
 \ul J_0(\a) := \inf_{\dbP\in  \cP(\a)} J_0(\a, \dbP),
 \end{equation}
with the convention that $\sup \emptyset := -\infty$ and $\inf\emptyset := \infty$. 
The upper and lower values of the game are then
 \begin{eqnarray}
 \label{ubar}
 \ol V_0 \;:=\; \inf_{\a^0\in \cA^0} \sup_{\a^1\in \cA^1} \ol J_0(\a),
 &\mbox{and}& 
 \ul V_0 \;:=\; \sup_{\a^1\in \cA^1}  \inf_{\a^0\in \cA^0} \ul J_0(\a).
 \end{eqnarray}
Notice that the inequality $\infty\ge \ol V_0 \ge \ul V_0\ge -\infty$ always holds. 

\vspace{0.5em}

\begin{assumption}\label{assum-coef}
{\rm (i)} $b$, $\si$, and $f$  are bounded, $\dbF-$progressively measurable in all variables, and uniformly continuous in $(t,\o)$ under $d_\infty$, uniformly in $a\in A$.

\vspace{0.5em}
{\rm (ii)} $\xi$ is bounded and uniformly continuous in $\o$.
\end{assumption}

We remark that we allow $\si$ to be degenerate. Under the assumptions on $b$ and $\si$, the sets $\cA^1$ and $\cA^2$ are not empty, as they contain constant and even piecewise constant controls. The remaining conditions on $f$ and $\xi$ guarantee that $J_0$ is finite. We emphasise that the boundedness assumption may be relaxed to linear growth. 

\begin{definition}
\label{defn-saddle}
The game value exists if $\ol V_0 = \ul V_0$. A control $\widehat \a = (\widehat \a^0, \widehat \a^1)\in \cA$ is a saddle--point of the game if
\begin{equation}
\label{saddle}
 \ol J_0(\widehat \a^0,\a^1) 
 \;\le\;  
 \ol J_0(\widehat \a) 
 \;=\; 
 \ol V_0
 \;=\;
 \ul V_0 
 \;=\; 
 \ul J_0(\widehat\a) 
 \;\le\; 
 \ul J_0(\a^0, \widehat\a^1), 
 ~\mbox{for all}~~
 (\a^0, \a^1)\in \cA.
\end{equation}
\end{definition}

We remark that, if $\widehat \a$ is a saddle--point, then $\cP(\widehat\a)\neq \emptyset$ and $J_0(\widehat\a, \dbP) = \ol V_0 = \ul V_0$ for all $\dbP\in \cP(\widehat\a)$. We conclude this subsection with a generic result concerning saddle--points. Denote
\begin{equation}
\label{J0a}
\ol J_0(\a^0) := \sup_{\a^1 \in \cA^1} \ol J_0(\a^0, \a^1),
~~\mbox{and}~~
\ul J_0(\a^1) := \inf_{\a^0 \in \cA^0} \ul J_0(\a^0, \a^1).
\end{equation}
Then the game problems in \eqref{ubar} can be rewritten
\begin{equation}
\label{ubar2}
\ol V_0 = \inf_{\a^0\in \cA^0}  \ol J_0(\a^0)
\mbox{ and }
 \ul V_0 =\; \sup_{\a^1\in \cA^1}   \ul J_0(\a^1).
\end{equation}

\begin{proposition}
\label{prop-saddle} Assume the game value exists. Then $\widehat \a\in \cA$ is a saddle--point of the game if and only if $\widehat \a^0$ and $\widehat \a^1$ are optimal controls of $\ol V_0$ and $\ul V_0$ in \eqref{ubar2}, respectively.
\end{proposition}

The proof is omitted as it follows directly from Definition \ref{defn-saddle}. We shall see that the existence of the game value will typically follow from the uniqueness of the viscosity solution to the corresponding {\rm HJBI} equation. By Proposition \ref{prop-saddle}, it seems that the existence of saddle--points is then reduced to two standard optimisation problems. However, we emphasise that, due to the weak formulation or more specifically our choice of feedback controls, the mappings $\a^0 \in \cA^0 \longmapsto \ol J_0(\a^0)$ and    $\a^1 \in \cA^1\longmapsto \ul J_0(\a^1)$ are typically not continuous. Even worse, it is not clear what is the appropriate topology for the sets $\cA^0$ and $\cA^1$, and whether they do have the appropriate topological requirements. Therefore, the optimisation problems in \eqref{ubar2} are actually real challenges.

\subsection{Path-dependent HJBI characterisation}
\label{sect:viscosity}
Similar to the Markovian case, we shall use path--dependent PDEs as a powerful tool to study the present path--dependent differential games. For $a=(a_0, a_1) \in A$, similar to \reff{Hamiltonian1} we define
 \begin{equation}
 \label{Hamiltonian}
\left.\begin{array}{c}
\dis h_t(\o,z,\g, a) :=  {\frac12} {\rm Tr}\big[\si^2_t(\o, a) \g\big] + b_t(\o,a) \cd z+ f_t(\o, a),
\\
 \dis\ol H_t := \inf_{a_0\in A_0} \sup_{a_1\in A_1} h_t(., a_0, a_1), 
 ~~ 
 \ul H_t := \sup_{a_1\in A_1}  \inf_{a_0\in A_0} h_t(., a_0, a_1).
 \end{array}\right.
 \end{equation}
Our main result requires the standard Isaacs condition 
\begin{equation}
\label{Isaacs}
\ol H = \ul H =: H.
\end{equation}

To prove existence of the game value, we shall use the viscosity theory of path--dependent PDEs (PPDE hereafter). Let $C^0(\Th, \dbR)$ denote the set of functions $u: \Th\longrightarrow \dbR$ continuous under $d_\infty$, $C^0_b(\Th, \dbR)$ the subset of bounded functions, and ${\rm UC}_b(\Th, \dbR)$ the subset of uniformly continuous functions. For any $L>0$, $\cP_L$ denotes the set of semimartingale measures on $\O$  such that the drift and diffusion of $X$ are bounded by $L$ and $\sqrt{2L}$, respectively, and $\cP_\infty := \cup_{L>0} \cP_L$. For any $t\in[0,T]$,  let $\mathcal T_t$ denote the set of $\mathbb F-$stopping times smaller than $t$. For a generic measurable set $E$, we also denote by $\dbL^0(\Th, E)$ the set of $\dbF-$progressively measurable $E-$valued functions. For any subset $\Pc$ of $\Pc_\infty$, we say that a property holds $\Pc-$q.s. if it holds $\P-$a.s. for any $\P\in\Pc$.

\vspace{0.5em}
For $\th=(t,\o)\in \Th$ and $\o'\in \O$, define
\begin{equation*}
(\o\otimes_t \o')_s := \o_s \1_{[0, t]} + (\o_t + \o'_{s-t} - \o'_0) \1_{(t, T]} (s),
~\zeta^{\th}(\o') := \zeta(\o\otimes_t \o'),
~\mbox{and}~
 \eta_s^{\th}(\o') := \eta_{t+s}(\o\otimes_t \o'),
 ~
  s\in[0,T-t],
 \end{equation*}
for an $\cF_T-$measurable random variable $\zeta$, and an $\dbF-$measurable process $\{\eta_s,s\in[0,T]\}$. We observe in particular that $\zeta^{\th}$ is $\cF_{T-t}-$measurable, and the process $\eta^{\th}$ is $\dbF-$adapted. Finally, for $\eps>0$, we introduce the stopping time $\ch_\eps(\o) := \inf\big\{t: d_\infty((t,\o), (0,0)) \ge \eps\big\}\wedge T$.

\vspace{0.5em}
For $u\in C^0(\Th, \dbR)$ and $\th:= (t,\o) \in\Th$, the super and sub--jets are defined by the subsets of $\dbR\times \dbR^d\times \dbS^d$: 
\begin{equation}
\label{cJ}
\left.\begin{array}{c}
\dis \ol\cJ^L u(\th) 
:= \Big\{ (\kappa, z, \g):
              -u(\th) = \sup_{\dbP\in \cP_L} \sup_{\t\in\mathcal T_{\ch_\eps}} \dbE^\dbP\big[\kappa \t + q^{z,\g}(X_\t) - u^\th_\t\big],
              ~\mbox{for some}~ \eps>0\Big\},
\\
\dis \ul\cJ^L u(\th) 
:= 
\Big\{ (\kappa, z, \g):
          -u(\th) = \inf_{\dbP\in \cP_L} \inf_{\t\in\mathcal T_{\ch_\eps}} \dbE^\dbP\big[\kappa \t + q^{z,\g}(X_\t) - u^\th_\t\big],
          ~\mbox{for some}~ \eps>0\Big\}.
\end{array}\right.
\end{equation}
where we used the notation $q^{z,\g}({\rm x}):= z\cd {\rm x} +  {\frac12}{\rm Tr}\big[\g  {\rm x} {\rm x}^\top\big],$ for all $({\rm x},z, \g)\in\dbR^d\times\R^d\times\dbS^d$. Now consider the following PPDE with generator $G: \Th \times \dbR\times \dbR^d\times \dbS^d \longrightarrow \dbR$
\begin{equation}
\label{PPDEG}
 - \pa_t u_t(\o) - G_t(\o, u, \pa_\o u, \pa^2_{\o\o} u)=0.
\end{equation}
The appropriate notion of path--derivatives in the above equation has been proposed first by Dupire {\rm \cite{dupire2009functional}}, and consists essentially in a directional derivative with respect to perturbations of the last value taken by the path. Since we are only interested in viscosity solutions, we do not need to detail this any further.

\begin{definition}
\label{defn-viscosity}
Let $u\in C^0(\Th, \dbR)$.
\\
{\rm (i)} For any $L>0$, we say $u$ is a $\cP_L-$viscosity super--solution $($resp. sub--solution$)$ of {\rm PPDE} \eqref{PPDEG} if, for any $\th\in\Th$ and any $(\kappa, z, \g) \in \ol\cJ^L u(\th)$ $($resp. in $ \ul\cJ^L u(\th))$, it holds that
\begin{eqnarray*}
- \kappa - G_t(\o, u_t(\o), z, \g) 
&\ge~(\mbox{\rm resp.}~ \le)& 
0.
\end{eqnarray*}
{\rm (ii)} $u$ is a $\cP_L-$viscosity solution of {\rm PPDE} \eqref{PPDEG} if it is both a $\cP_L-$viscosity super and sub--solution of \eqref{PPDEG}. Moreover, $u$ is a $\cP_\infty-$viscosity solution of {\rm PPDE} \eqref{PPDEG} if it is an $\cP_L-$viscosity solution for some $L>0$.
\end{definition}

We remark that in our weak formulation setting (or more precisely due to the feedback type of controls), the regularity of the game value is a very subtle question. As will be explained later on in the paper, our main need for {\it a priori} regularity is linked to our proof of the dynamic programming principle. This is in stark contrast with the usual control theory, for which proving dynamic programming in weak formulation requires merely to assume proper measurability of the coefficients, see for instance the recent papers by El Karoui and Tan \cite{karoui2013capacities,karoui2013capacities2}. Proving regularity in a strong formulation framework, either in the strategy against control setting or the delayed strategies one of Section \ref{sect:formulation}, is usually simply a matter of obtaining appropriate estimates on the moments of the controlled diffusion, which does not pose any specific problem in general. However in weak formulation, the fact that the players are controlling the distribution of $X$ makes matters much worse. Once again, this is the price to pay if one wishes to obtain general existence results for saddle--points, as shown in our previous examples.

\vspace{0.5em}

In order to bypass this difficulty and for the sake of clarity, we consider in this subsection a special case which is easier to deal with. A more general case will be studied in Section \ref{sect-extension} below.

\begin{assumption}
\label{assum-bsi}
 $b= b(t, a)$ and $\si = \si(t, a)$ are independent of $\o$.
%
\end{assumption}

Our first main result is the following. The proof is postponed to the next section.
 
\begin{theorem}
\label{thm-gamevalue} 
Let Assumptions \ref{assum-coef}, \ref{assum-bsi}, and Isaacs's condition \eqref{Isaacs} hold. Then
\\
{\rm (i)} The following path dependent Isaacs equation has a  viscosity solution $u\in \mbox{\rm UC}_b(\Theta,\dbR)$
 \begin{equation}
 \label{PPDE}
 -\pa_t u - H_t(\o,  \pa_\o u, \pa^2_{\o\o} u) =0,\;  t<T,\; u_T = \xi.
 \end{equation}
{\rm (ii)} Assume further that this {\rm PPDE} has a unique viscosity solution in $ \mbox{\rm UC}_b(\Theta,\dbR)$. Then $\ol V_0=\ul V_0=u_0(0)$.
 \end{theorem}

\begin{remark}
\label{rem-gamevalue}
{\rm (i)} The uniqueness of viscosity solution is a highly nontrivial issue. In Ekren, Touzi and Zhang {\rm \cite{ekren2012viscosity}} and Pham and Zhang {\rm \cite{pham2014two}}, uniqueness was proved only in the case where $\si$ is uniformly non--degenerate and the dimension $d$ is either $1$ or $2$.

\vspace{0.5em}
{\rm (ii)} If one can prove the existence of viscosity solution in a smaller class $\mathfrak{C} \subset C( \Theta,\dbR)$ $($or more precisely show that the dynamic upper and lower value processes of the game are in $\mathfrak{C})$, then it is actually enough to prove the uniqueness of viscosity solution in the class $\mathfrak{C}$. Hence, Ren,Touzi and Zhang {\rm \cite{ren2015comparison}} proved the existence and uniqueness in a subclass of $\mbox{\rm UC}_b(\Theta,\dbR)$, under slightly stronger regularity assumptions on $\xi$ and $f$ in terms of $\o$. Using their result provides us with the existence of the game value, even when $\si$ is degenerate or $d\ge 3$.

\vspace{0.5em}
{\rm (iii)} Notice however that {\rm \cite{ren2015comparison}} requires that $b$ and $\si$ are independent of $\o$. On the other hand, Ekren and Zhang {\rm \cite{ekren2016pseudo}} have introduced a different type of viscosity solution called pseudo--Markovian viscosity solutions, and proved the corresponding uniqueness under weaker conditions. This will be interesting for us when we extend our game problem to path dependent $(b, \si)$ in Section \ref{sect-extension} below.

\vspace{0.5em}
{\rm (iv)} Finally, in the Markovian case, we can replace our uniqueness assumption with the uniqueness in the standard Crandall--Lions notion of viscosity solutions. However, notice that our uniqueness assumption is always weaker.
\end{remark}

\subsection{A verification result}
\label{sect-verification}
In this subsection we establish the existence of a saddle--point under the stronger condition that the Isaacs equation \eqref{PPDE} has additional regularity inspired from the $W^{1,2}-$Sobolev solutions in the PDE literature. We emphasise that in this subsection we do not require Assumption \ref{assum-bsi} to hold. This is to be expected as already in standard stochastic control problems, verification--type arguments do not require to prove beforehand the dynamic programming principle, which, once more, is the only result which requires regularity of the value function. Note that
\begin{equation}
\label{hatsi}
\mathrm{d}\la X\ra_t \;=:\; \widehat \si^2_t \mathrm{d}t, ~\mbox{holds}~\cP_\infty-\mbox{q.s.},
\end{equation}  
and can be defined without reference to a specific measure in $\Pc_\infty$ by classical results of Bichteler \cite{bichteler1981stochastic}.

\begin{definition}
 \label{defn-W12loc}
 Let $u \in \dbL^0(\Th, \dbR)$ and $\cP \subset \cP_\infty$. We say that $u\in W^{1,2}_{\rm loc}(\cP)$ if 
 \\
{\rm (i)}  $u$ is a uniformly integrable semimartinagle under any $\dbP\in \cP$;
 \\
{\rm (ii)} for some measurable processes $\pa_t u,\pa_\o u,\pa^2_{\o\o} u$ valued in $\mathbb{R},\mathbb{R}^d$, and $\mathbb{S}^d $, $u$ has the decomposition
 \begin{equation}
 \label{Ito}
  \mathrm{d} u_t 
  = \pa_t u_t \,dt + \pa_\o u_t \cd  \mathrm{d}X_t + {\frac12} {\rm Tr}\big[\pa^2_{\o\o} u_t \,d\la X\ra_t\big],
  ~ \cP-\mbox{\rm q.s.} ,
  \end{equation} 
with $\int_0^T \Big(\Big|\pa_t u_t + {\frac12} {\rm Tr}\big[\pa^2_{\o\o} u_t \widehat \si^2_t\big]\Big| +  |\pa_\o u_t|^2 \Big)  \mathrm{d}t  <\infty$, $\cP-$q.s.
 \end{definition}
 
 \begin{remark}
 \label{rem-W12}
{\rm (i)} By definition, any $u\in W^{1,2}_{\rm loc}(\cP)$ is continuous in $t$, $\cP-$q.s. However, unlike in Dupire {\rm \cite{dupire2009functional}} or in Ekren, Touzi and Zhang {\rm \cite{ekren2016viscosity, ekren2012viscosity}}, we do not require pointwise continuity in $\o$. In particular, this relaxation will allow us to cover {\rm BSDEs} and {\rm 2BSDEs} with measurable coefficients. 
 
 \vspace{0.5em}
{\rm (ii)} Clearly, \eqref{Ito} is closely related to the functional It\^o formula in {\rm \cite{dupire2009functional, ekren2016viscosity, ekren2012viscosity}}. However, our requirements here are weaker. In particular, we do not require the uniqueness of $\pa_t u$ and $\pa^2_{\o\o} u$.
 
 \vspace{0.5em}
{\rm (iii)} One can easily see that $\pa_\o u$ is unique $ \mathrm{d}\la X\ra_t \times  \mathrm{d}\dbP-${\rm a.s.},  and  $\pa_t u_t + {\frac12} {\rm Tr}[\widehat \si^2_t  \pa^2_{\o\o} u_t]$ is unique $ \mathrm{d}t\times  \mathrm{d}\dbP-${\rm a.s.}, for all $\dbP\in \cP$. If we set $\cP = \cP_L$ for some $L>0$, and require $\pa_t u_t$ and $\pa^2_{\o\o} u_t$ to be $\cP_L-${\rm q.s.} continuous, then it follows from Song {\rm \cite{song2012uniqueness}} that $\pa_t u$ and $\pa^2_{\o\o} u$ are unique in the $\cP_L-${\rm q.s.} sense. This additional quasi--sure continuity requirement leads to the $G-$Sobolev space $W^{1,2}_G$ introduced by Peng and Song {\rm\cite{peng2015g}}. 
 \end{remark}
  
 \begin{theorem}
\label{thm-verification} 
Let Assumption \ref{assum-coef} and Isaacs's condition \eqref{Isaacs} hold. Assume further that

\vspace{0.5em}
{\rm (i)} The {\rm PPDE} \eqref{PPDE} has a $W^{1,2}_{\rm loc}(\cP)-$solution $u$, where $\cP := \cup_{\a\in\cA} \cP(\a)$. In other words, there exist, not necessarily unique, $(\pa_t u, \pa_\o u, \pa^2_{\o\o} u)$ satisfying \eqref{Ito} and 
\begin{equation}
\label{PPDEqs}
-\pa_t u - H_t(\pa_\o u, \pa_{\o\o}^2 u)=0,\; \cP-{\rm q.s.}
\end{equation}
{\rm (ii)} The Hamiltonian $H_t(\o, \pa_\o u, \pa_{\o\o}^2 u)$ has a measurable saddle--point $\widehat \a_t(\o)$, such that $\cP-{\rm q.s.}$
\begin{align}\label{Isaacs-H-saddle}
h_t(\pa_\o u, \pa_{\o\o}^2u,  \widehat \a^0_t, a_1) 
&\le  h_t(\pa_\o u, \pa_{\o\o}^2u, \widehat\a^0_t, \widehat \a^1_t)  = H_t(\pa_\o u, \pa_{\o\o}^2 u)\le h_t(\pa_\o u, \pa_{\o\o}^2u, a_0, \widehat \a^1_t),\;\mbox{for all}\; a\in A.
\end{align}
{\rm (iii)} The set $\cP(\widehat \a)$ is not empty. 

\vspace{0.5em}
Then $\ol V_0 = u_0 = \ul V_0$ and  $\widehat \a$ is a saddle--point of the game.
 \end{theorem}

\proof For any $\a^0\in \cA^0$, and $\dbP\in \cP(\a^0, \widehat \a^1)$, it follows from \eqref{Ito} and \eqref{PPDEqs} that, $\dbP-\mbox{a.s.}$
 $$
 d u_t 
 = 
 \Big(b_t(\a^0_t, \widehat\a^1_t) \cd \pa_\o u 
        + {\frac12}{\rm Tr}\big[ \si^2_t(\a^0_t, \widehat\a^1_t) \pa_{\o\o}^2u\big]
        - H_t(\pa_\o u, \pa_{\o\o}^2 u)\Big) \mathrm{d}t 
        + \pa_\o u \cdot  \mathrm{d} M^{\a^0, \widehat \a^1}_t
 \ge 
 -f_t(\a^0_t, \hat\a^1_t)  \mathrm{d}t  +\pa_\o u \cdot  \mathrm{d} M^{\a^0, \widehat \a^1}_t.
 $$
Recall the integrability condition in Definition \ref{defn-W12loc} (ii), and denote $\t_n := T\wedge\inf\big\{t:  \int_0^t (|\pa_t u_s + {\frac12}{\rm Tr}[ \pa^2_{\o\o} u_s \hat \si^2_s]| +  |\pa_\o u_s|^2 ) ds \ge n\big\}$, so that  $\lim_{n\to\infty}\dbP[\t_n = T]=1$. As $\widehat \si^2_t =  \si^2_t(\a^0_t, \widehat\a^1_t)$, $\dbP-$a.s. we see by integrating and taking expectations above that $u_0 \le \dbE^\dbP\big[u_{\t_n} + \int_0^{\t_n} f_t(\a^0_t, \hat\a^1_t) \mathrm{d}t\big]$. Send now $n\longrightarrow \infty$, since $u$ is continuous in $t$ and uniformly integrable under $\dbP$, we obtain from the dominated convergence theorem under $\dbP$ that
 \begin{equation*}
 J_0(\a^0, \widehat\a^1, \dbP) 
 \;=\; \dbE^\dbP\bigg[\xi + \int_0^T f_t(\a^0_t, \widehat\a^1_t) \mathrm{d}t\bigg] 
 \;\ge\; 
 u_0.
 \end{equation*}
Following similar arguments, for any $\a^0\in \cA^0$, $\a^1\in \cA^1$, $\dbP^0\in \cP(\a^0, \widehat \a^1)$, $\dbP^1\in \cP(\widehat\a^0, \a^1)$, $\widehat \dbP\in \cP(\widehat\a)$, we obtain
 \begin{equation*}
 J_0(\widehat\a^0, \a^1, \dbP^1)  
 \;\le\; u_0 
 \;=\; 
 J_0(\widehat\a, \widehat\dbP) 
 \;\le\;  
 J_0(\a^0, \widehat\a^1, \dbP^0).
 \end{equation*}
This implies immediately that $\ol V_0 = u_0 = \ul V_0$ and  $\widehat \a$ is a saddle--point of the game.
 \qed
 
 \begin{remark}
 \label{rem-unique}
Notice that in Theorem \ref{thm-verification} (iii), we do not require $\cP(\widehat \a)$ to be a singleton. When it consists of several measures $\dbP$, the value $J_0(\widehat \a, \dbP)$ is actually independent of the choice of $\P$.
 \end{remark}
 
 \begin{remark}
 \label{rem-vW12}
In the Markovian case, it is not enough to assume that the {\rm PDE} has a Sobolev solution $u$ in the standard $W^{1,2}-$space. Indeed, consider the first inequality of \eqref{Isaacs-H-saddle} in that setting
\begin{equation}
\label{Isaacs-Markov}
h(t,x, D u, D^2u,  \widehat \a^0(t,x), a_1) \le   H(t,x, D u, D^2 u),\;  (t, x) \notin N,
\end{equation}
where $N$ is an appropriate set with Lebesgue measure $0$ on $[0,T]\times\mathbb R^d$. When $\sigma$ is uniformly elliptic, the generalised It\^o formula of Krylov {\rm \cite[Theorem 2.10.1]{krylov1980controlled}} allows to conclude as usual. However, in the degenerate case, equality fails to hold in general, see {\rm \cite[Theorem 2.10.2]{krylov1980controlled}}, which means that \eqref{Isaacs-Markov} does not imply that the corresponding inequality in \reff{Isaacs-H-saddle} holds $\dbP-$a.s. so that the arguments in Theorem \ref{thm-verification} would fail. In other words, even in the Markovian setting, with degenerate volatility, our quasi--sure characterisation of the Sobolev solution and the saddle--points of the Hamiltonian seems necessary for the verification arguments.
 \end{remark}
 
A saddle--point exists under our formulation in the following examples. We emphasise that none of the formulations in the existing literature allows to handle these examples, see Section \ref{sect:formulation}.

\begin{example}
\label{eg-fleming1989existence3}
Compare to Example \ref{eg-fleming1989existence2}.
Set $A_0:=[1,2]$, $A_1:=[0,1]$, $d:=1$, $b:=0$, $\si_t(a) := |a_0|$, $f_t(a) := \ol \zeta(\o_t) - 2 a_0 \ol \zeta(\o_t)$, $\xi(\o) := |\o_T|^2$.  That is, $\cP(\a)$ consists of weak solutions of 
\begin{eqnarray}
\label{egfleming1989existence3-X}
X_t = \int_0^t |\a^0_s(X_\cd)|  \mathrm{d}W_s,
&\mbox{and}&
J_0(\a, \dbP) := \dbE^{\dbP}\Big[ |X_T|^2 - \int_0^T \big[2\a^0_t\bar\zeta(X_t)  -  |\bar\zeta(X_t)|^2\big]  \mathrm{d}t\Big].
\end{eqnarray}
We remark that $f$ and $\xi$ here violate the boundedness assumption in Assumption \ref{assum-coef}. However, this condition is mainly for the convenience of the general path dependent PDEs, and in this particular example, it is not needed.  Nevertheless, in this case, we have $\ol V_0=\ul V_0 = v(0,0)$, where $v(t,x) = x^2+T-t$ is  the classical solution to the {\rm HJBI} equation \eqref{PPDE}. Again, the $($unique$)$ saddle--point of the Hamiltonian in the sense of \eqref{Isaacs-H-saddle} is provided by \eqref{eg3-saddle}. Define $\widehat \a_t(\o) := \widehat a_t(\o_t)$. Then $|\widehat \a_t(\o)| = \zeta(\o_t)$ and thus \eqref{egfleming1989existence3-X} becomes \eqref{BarlowSDE}. By Barlow {\rm \cite{barlow1982one}}, $\cP(\hat\a)$ is a singleton. This implies easily that $\widehat \a$ is indeed a saddle--point of the game.
\end{example}

The next example extends Example \ref{eg-weak1} to the degenerate case, namely  $c=0$ or $|\rho|=1$. We remark that the result of Hamad\`ene and Lepeltier \cite{hamadene1995zero}  does not apply in this case.

\begin{example}
\label{eg-weak2}
Consider the setting of Example \ref{eg-weak1}, but with $c=0$ or $|\rho|=1$.
\\
{\rm (i)} Let $\rho=1$ and $c>0$; the case $\rho=-1$ can be treated similarly. The {\rm HJBI} equation reduces to
\begin{equation*}
 -\pa_t v - {\frac12}c^2\big(\pa_{x_1x_1} v + \pa_{x_2x_2} v + 2 \pa_{x_1x_2} v
                                     \big) 
 +|\pa_{x_1} v| - |\pa_{x_2} v|=0,\; v(T,x_1,x_2) 
 =
 |x_1-x_2|^2,
\end{equation*}
which has a classical solution $v(t,x_1, x_2) = |x_1-x_2|^2$.  The saddle--points for the Hamiltonian are then given by
\begin{equation*}
\widehat a_0(t,x) := -{\rm sgn} (x_1-x_2) - \beta_1(t,x) \1_{\{x_1=x_2\}},~ \widehat a_1(t,x) := {\rm sgn} (x_2-x_1) +  \beta_2(t,x)\1_{\{x_1=x_2\}},
\end{equation*}
for any arbitrary measurable functions $\beta_1, \beta_2$ taking values in $[-1,1]$. Denote $\D X := X^2-X^1$. In order to prove that the two--dimensional {\rm SDE} $X^{i}_t = \int_0^t \big({\rm sgn} (\D X_s) - \beta_1(s, X_s) \1_{\{\D X_s = 0\}}\big) \mathrm{d}s  + c W_t$,~$i=0,1$, has a weak solution, we observe that the difference satisfies the {\rm ODE} $\D X_t =  \int_0^t (\beta_1+\beta_2)(s, X_s) \1_{\{\D X_s = 0\}}  \mathrm{d}s$, which has a solution if and only if $\beta_1 + \beta_2=0$. In this case, the unique solution is $\D X_t = 0$, i.e. $X^1=X^0$, and we are reduced to the ODE $X^0_t =  \int_0^t \beta_1(s, X^0_s,X^0_s)  \mathrm{d}s  + c W_t$, which also has a unique weak solution since $c>0$ and $\beta_1$ is bounded by definition. Therefore, the saddle--points of the game are
 \begin{eqnarray*}
 \widehat \a^0_t =\widehat \a^1_t = {\rm sgn}(X^1_t-X^0_t) + \beta(t, X^0_t) \1_{\{X^0_t = X^1_t\}},
 &\mbox{for an arbitrary measurable}&
 \beta: [0, T]\times \dbR\longrightarrow [-1, 1].
 \end{eqnarray*}
{\rm (ii)} When $c=0$, the last characterisation of the saddle--point still holds true provided that the {\rm ODE} 
$X^0_t = \int_0^t \beta(s, X^0_s)  \mathrm{d}s$
has a solution.
\end{example}
 
In the rest of this section, we investigate the question of existence of weak solutions for path--dependent SDEs, thus providing sufficient conditions for Theorem \ref{thm-verification} (iii). In the Markovian setting, we refer to Krylov \cite[Section 2.6]{krylov1980controlled} for existence of weak solutions of SDEs with measurable coefficients. We emphasise again that uniqueness is not required in our approach. 

\begin{definition}
\label{qscont}
An $\dbF-$mesurable process $\f$, with values in an Euclidean space, is called $\cP-${\rm q.s.} continuous if, for any $\eps>0$, there exist closed subsets $\{\O^\eps_t\}_{0\le t\le T} \subset \O$ such that  
\\
{\rm (i)} the process $\1_{\O^\eps_t}$ is $\dbF-$progressively measurable with $\sup_{\dbP\in \cP} \dbE^\dbP\Big[\int_0^T \1_{(\O^\eps_t)^c} \mathrm{d}t \Big] \le \eps$.
\\
{\rm (ii)} For each $t$,  $\f(t,\cd)$ is continuous in $\O^\eps_t$.  
\end{definition}

The  following existence result looks standard and its proof is postponed to the Appendix Section \ref{sect:Appendix}.

\begin{theorem}
\label{thm-qsSDE}
Assume $b:\Th \longrightarrow \dbR^d$ and $\si: \Th\longrightarrow \dbS^d$ are bounded by a certain constant $L$, are $\dbF-$measurable, and $\cP_L-${\rm q.s.} continuous. Then there is a weak solution to the {\rm SDE} $X_t = \int_0^t b(s, X_\cd) \mathrm{d}s + \int_0^t \si(s, X_\cd) \mathrm{d}W_s$.
\end{theorem}

We finally discuss the desired regularity for saddle--points of the Hamiltonian, for which we define $W^{1,2}(\cP)$ as the subset of $W^{1,2}_{\rm loc}(\cP)$ consisting of processes $u$ such that $\pa_t u$, $\pa_\o u$, and $\pa^2_{\o\o} u$ are $\Pc-$q.s. continuous.

 \begin{theorem}
\label{thm-verification2} 
Let Assumption \ref{assum-coef} holds, $b, \si$ continuous in $a$, and assume further the Isaacs condition \eqref{Isaacs} holds with saddle-point $\widehat \a_t(\o, z,\g)$
 \begin{equation*}
 h_t(\o, z,\g,   \widehat \a^0_t, a_1) 
 \le  h_t(\o, z,\g, \widehat\a^0_t, \widehat \a^1_t)  = H_t(\o, z,\g)\le h_t(\o, z,\g, a_0, \widehat \a^1_t),\; \mbox{for all}\; a\in A,\;  
 \cP- {\rm a.s.},
 \end{equation*}
such that $\widehat \a$ is continuous in $(z,\g)$, and uniformly $\cP-$continuous in $(t,\o)$, i.e. $\widehat \a(\cd, z,\g)$ is $\cP-$continuous for all $(z,\g)$, with a common $\{\O^\eps_t\}_{0\le t\le T}$ for all $(z,\g)$.

 \vspace{0.5em}
 
Let $u$ be a $W^{1,2}(\cP)-$solution of the {\rm PPDE} \eqref{PPDE} with $\cP := \cup_{\a\in\cA} \cP(\a)$. Then $\ol V_0 = u_0 = \ul V_0$ and  $\a^*_t(\o) := \widehat \a_t\big(\o, \pa_\o u(t,\o), \pa^2_{\o\o} u(t,\o)\big)$ is a saddle-point of the game.
 \end{theorem}
 
\proof  By our regularity assumptions on $u$ and $\widehat\alpha$, we see that  $\a^*_t(\o) := \widehat \a_t\big(\o, \pa_\o u(t,\o), \pa^2_{\o\o} u(t,\o)\big)$ is $\cP-$q.s. continuous. Together with the continuity of $b$ and $\sigma$ in $a$, this implies that $\widehat b(t,\o) := b_t(\o, \a_t^*(\o))$ and $\widehat \si(t,\o) := \si_t(\o, \a_t^*(\o))$ are also $\cP-$q.s. continuous. Now it follows from Theorem \ref{thm-qsSDE} that $\cP(\a^*) \neq \emptyset$.  Therefore, we can apply Theorem \ref{thm-verification} to conclude. 
\qed

\subsection{Second order backward SDE characterisation}
\label{sect-2BSDE}

For $\a^0\in \cA_0$ fixed, The stochastic control problem $\ol J_0(\a^0)$, defined in \eqref{J0a}, can be characterized by the corresponding second order BSDEs, as introduced by Cheridito, Soner, Touzi and Victoir \cite{cheridito2007second}, and further developed by Soner, Touzi and Zhang \cite{soner2012wellposedness} and Possama\"i, Tan and Zhou \cite{possamai2015stochastic}. 

\subsubsection{The general case}

We emphasise that in this subsection, we do not require Assumption \ref{assum-bsi} to hold. 

\begin{assumption}
\label{assump:2bsde} 
{\rm (i)} $b= \si \l$ for some $\l: \Th \times A\to \dbR^d$ progressively measurable and bounded.
\\
{\rm (ii)} For every $(t,\omega)\in\Th$ and $a=(a_0,a_1)\in A$, the following two sets are convex
$$
\big\{(b,\sigma,f)(t,\omega,a_0', a_1):~a_0' \in A_0\big\},\q \big\{(b,\sigma,f)(t,\omega,a_0, a_1'):~a_0' \in A_1\big\}.
$$
\end{assumption}

Assumption \ref{assump:2bsde} (i) is made in order to focus on the main arguments, and can be relaxed at the price of more technical developments highlight more our core ideas. Assumption \ref{assump:2bsde} $(ii)$ is mainly technical, and will be explained in the proof of Lemma \ref{lem-weak} below.

\vspace{0.5em}

Recall  $\ol \cA$, $\cP(\a)$, and \reff{cA} from Subsection \ref{sect:setting}. We introduce the corresponding terms when $b=0$
\begin{eqnarray}
\nonumber
&\dis \Pc_0(\alpha):= \Big\{\dbP:~X_.=\int_0^.\sigma_s\big(X_\cdot,\a(X_\cdot)\big) \mathrm{d}W_s,
                                                   ~\dbP-\mbox{a.s.}
                                   \Big\},
~~
\dis \Pc_0^0(\a^1):=\bigcup_{\alpha^0\in \ol\Ac^0}\Pc_0(\alpha),
~~
\Pc_0^1(\a^0):=\bigcup_{\alpha^1\in \ol\Ac^1}\Pc_0(\alpha),&\\
&\dis \Ac_0^0=\big\{\a^0\in\overline\Ac^0: \Pc^1_0(\alpha^0)\neq\emptyset\big\},\; \Ac_0^1:=\big\{\a^1\in\overline\Ac^1: \Pc^0_0(\a^1)\neq\emptyset\big\},\; \cA_0 := \cA^0_0\times \cA^1_0.&\label{driftlessSDEsquare2}
\end{eqnarray}
For every pair $(\alpha,\P)\in\Ac_0\times \Pc_0(\alpha)$, let $W^{\dbP, \a}$ be an $\P-$Brownian motion corresponding to the driftless SDE in the definition of $\Pc_0(\alpha)$. 
For $\lambda$ satisfying Assumption \ref{assump:2bsde} $(i)$, we denote $\lambda^\alpha:=\lambda\big(X_\cdot,\a(X_\cdot)\big)$, and we introduce an equivalent measure $\P^{\a}$, together with the corresponding Brownian motion (by Girsanov theorem):
\begin{eqnarray*}
\frac{ \mathrm{d}\P^{\a}}{ \mathrm{d}\P}:=\exp\Big(\int_0^T\lambda^\alpha_s\cdot  \mathrm{d}W_s^{\P,\alpha}-\frac12\int_0^T\big|\lambda^\alpha_s\big|^2 \mathrm{d}s\Big),
&\mbox{and}&
W^{\P^{\a}}:= W^{\P,\alpha}-\int_0^.\lambda^\alpha_s \mathrm{d}s,
\end{eqnarray*}
so that the dynamics of the canonical process under $\dbP^\alpha$ are given by:
 \begin{eqnarray*}
 \mathrm{d}X_s
 &=&
 b_s\big(X_\cdot,\a_s(X_\cdot)\big) \mathrm{d}s+\sigma_s\big(X_\cdot,\a(X_\cdot)\big) \mathrm{d}W^{\P^{\a}}_s,
 ~~\P^\a-\mbox{a.s.}
 \end{eqnarray*}
The upper and lower values of the game can then be rewritten as
\begin{eqnarray*}
&\dis \ol V_0^0 \;:=\; \inf_{\a^0\in \cA^0_0} \sup_{\a^1\in  \cA^1_0} \ol J_0(\alpha), 
  ~~\mbox{and}~~
  \ul V_0^0 := \sup_{\a^1\in  \cA^1_0}  \inf_{\a^0\in \cA^0_0}\ul J_0(\alpha),
 &\\
 &\dis \mbox{where}~~  
 \ol J_0(\alpha):=\sup_{\P\in {\color{black}\Pc_0(\a)}}\E^{\P^{\a}}\Big[\xi+\int_0^Tf_s(\a_s) \mathrm{d}s\Big],
 ~~
 \ul J_0(\alpha):=\inf_{\P\in {\color{black}\Pc_0(\a)}}\E^{\P^{\a}}\Big[\xi+\int_0^Tf_s(\a_s) \mathrm{d}s\Big].
 &
\end{eqnarray*}
We next introduce the nonlinear generators
 \begin{eqnarray}\label{2BSDEF}
 \ol F_t(\o,z,a_0)
 \;:=\;
 \sup_{a_1\in A_1(t,\omega,a_0)} F_t(\o, z, a_0, a_1),
 &\mbox{and}&
  \ul F_t(\o,z,a_1)
  \;:=\;
  \inf_{a_0\in A_0(t,\o,a_1)} F_t(\o, z, a_0, a_1),
\end{eqnarray}
where, for the  $\widehat\si$ defined in \reff{hatsi},
 \begin{eqnarray*}
 F_t(\o, z, a):= b_t(\o,a)\cdot z+f_t(\o,a),
 &\mbox{and}&
 A_i(t,\omega,a_{1-i})
 :=
 \big\{a_i\in A_i: \big(\sigma_t\sigma_t^\top\big)(\omega,a_0,a_1)=\widehat\si^2_t(\o)
 \big\},
 ~~i=0,1.
 \end{eqnarray*}
The second order backward SDEs (2BSDE, hereafter)  which will serve to represent the upper and lower values is defined by the following representation of the r.v. $\xi$ 
\footnote{We remark that 2BSDEs \reff{eq:2bsde4} and \reff{eq:2bsde5} do not include an orthogonal martingale term, even though the involved probabilities $\dbP$ may not satisfy the predictable martingale representation property.  Here we will use the so called optional decomposition rather than the Doob-Meyer one to obtain the processes $\ul K^{\a^1}$ and   $\ol K^{\a^0}$, see the notion of "saturated" solutions of 2BSDEs in  Possama\"i, Tan and Zhou \cite{possamai2015stochastic}.  }: 
 \begin{eqnarray}\label{eq:2bsde4}
 \ul Y^{\a^1}_t
 &=&
 \xi+\int_t^T\ul F_s\big ( \ul Z_s^{\a^1},\a^1_s\big) \mathrm{d}s
 -\int_t^T\ul Z_s^{\a^1}\cdot  \mathrm{d}X_s-\int_t^T \mathrm{d} \ul K^{\a^1}_s,
 ~\P-\mbox{a.s., for all}~\P\in\Pc^0_0(\a^1),\\
 \label{eq:2bsde5}
 \ol Y_t^{\a^0}
&=&
 \xi+\int_t^T\ol F_s\big(\ol Z_s^{\a^0},\a^0_s\big) \mathrm{d}s
 -\int_t^T\ol Z_s^{\a^0}\cdot  \mathrm{d}X_s+\int_t^T \mathrm{d} \ol K^{\a^0}_s,
 ~\P-\mbox{a.s., for all}~\P\in\Pc_0^1(\a^0).
 \end{eqnarray}

\begin{definition} \label{def:2BSDE}
We say that $(\ul Y^{\a^1},\ul Z^{\a^1})$  is a solution of the {\rm 2BSDE} \eqref{eq:2bsde4} if, for some $p>1$,

\vspace{0.3em}
\noindent {\rm (i)} $\ul Y^{\a^1}$  is a c\`adl\`ag and {$\F^{\Pc^0_0(\a^1)+}-$}optional process, with
$\|\ul Y^{\a^1}\|^p_{\mathbb D^0_p}:=\sup_{\P\in\Pc^0_0(\a^1)}\E^{\P}\left[\sup_{t\le T}|\ul Y^{\a^1}_t|^p\right]<\infty$.

\noindent 
{\rm (ii)} $\ul Z^{\a^1}$  is an $\F^{\Pc^0_0(\a^1)}-$predictable process, with
$\|\ul Z^{\a^1}\|^p_{\H^0_p}:=\sup_{\P\in{\Pc^0_0(\a^1)}}\E^{\P}\big[\big(\int_0^T(Z^{\a^1}_t)^\top\widehat\sigma_t^2Z^{\a^1}_t\mathrm{d}t\big)^{\frac p2}\big]<\infty$.

\noindent 
{\rm (iii)} For all $\dbP\in \Pc^0_0(\a^1)$,  the process 
\begin{equation}
\label{KaP}
\ul K^{\a^1, \dbP}_t :=  \ul Y^{\a^1}_t -  \ul Y^{\a^1}_0+\int_0^t \ul F_s\big ( \ul Z_s^{\a^1}, \a^1_s\big) \mathrm{d}s
 -\int_0^t \ul Z_s^{\a^1}\cdot  \mathrm{d}X_s,\; t\in[0,T],
\end{equation}
 is $\dbP-$optional, non--decreasing, and satisfies the minimality condition: 
 \begin{equation}
 \label{minimality2}
 \ul K^{\a^1,\P}_t
 =
 \underset{ \P' \in \Pc^0_0(\a^1; t,\P,\F^{+}) }{ {\rm essinf}^{\P} }
 \E^{\P'}\big[ \ul K^{\a^1,\P'}_T\big|\Fc_{t}^{\P+}\big], 
  ~0\leq t\leq T,\; \P-{\rm a.s.},
 \end{equation}
where  $\Pc^0_0(\a^1; t,\P,\F^{+}) := \left\{\P'\in\Pc^0_0(\a^1): \P[E]=\P'[E]~\mbox{for all}~E\in\Fc_t^{+}\right\},\;  t\in[0,T].$ 
\\
The solution $(\ol Y^{\a^0}, \ol Z^{\a^0})$ of the {\rm 2BSDE} \reff{eq:2bsde5} is defined similarly.
\end{definition}

We are now ready for the main result of this subsection, the proof is postponed to Section \ref{sect:2BSDE} below.

\begin{theorem}\label{thm:2bsde}
{\rm (i)} Under Assumptions \ref{assum-coef} and \ref{assump:2bsde}, the {\rm 2BSDEs} \eqref{eq:2bsde4} and \eqref{eq:2bsde5} have unique solutions, with  
 \begin{eqnarray*}
 \ol V_0 \;:=\; \inf_{\a^0\in \cA^0_0}\sup_{\P\in\Pc_0^1(\a^0)}\E^{\P^\a}\Big[\ol{Y}_0^{\a^0}\Big], 
 &\mbox{and}&
 \ul V_0 \;:=\; \sup_{\a^1\in  \cA^1_0}  \inf_{\P\in\Pc^0_0(\a^1)}\E^{\P^\a}\Big[\ul{Y}_0^{\a^1}\Big] ;
 \end{eqnarray*}
{\rm (ii)}  If in addition, for some $\widehat\a =(\widehat \a^0, \widehat \a^1)\in \Ac_0$, and $\widehat{\mathbb P} \in\cP_0(\widehat\a)$, 
\begin{equation}
\label{K0}
\ul F_s\big ( \ul Z_s^{\widehat\a^1}, \widehat\a^1_s\big)=\ol F_s\big(\ol Z_s^{\widehat \a^0},\widehat\a^0_s\big),\; \mathrm{d}s\otimes \mathrm{d}\widehat \P-{\rm a.e.}, \; \text{and}\; \ol K^{\widehat\a^0}_T = \ul K^{\widehat\a^1}_T=0,\; \widehat \P-\mbox{a.s.}
\end{equation}
then the game value exists, i.e. $ \ol V_0 = \ul V_0$, and $\widehat \a$ is a saddle--point of the game.
\end{theorem}

We notice that Theorem \ref{thm:2bsde} does not require Assumption \ref{assum-bsi}. Moreover, Theorem \ref{thm:2bsde} (ii) does not require \reff{Isaacs} directly. However, the conditions in \reff{K0} are rather strong, and essentially imply \eqref{Isaacs}, see Remark \ref{rem-2BSDE3} below for more details.  

\subsubsection{A simpler case: volatility controlled by one player only}

In this subsection we improve Theorem \ref{thm:2bsde} when only one player, say Player 0, controls the volatility, i.e. $\si_t(\o,a) = \si_t(\o, a_0)$.\footnote{See Mastrolia and Possama\"i {\rm \cite{mastrolia2015moral}}, Hern\'andez Santib\'a\~nez and Mastrolia {\rm \cite{hernandez2018moral}}, and Sung {\rm \cite{sung2015optimal}} for a similar setting in the context of moral hazard problems under uncertainty.}
Then, 
\reff{Hamiltonian} reduces to
 \[ \ol H_t(z,\g) = \inf_{\Sigma\in{\bf\Sigma}_t}\Big\{\frac12{\rm Tr}\big[\Sigma\g\big]+\ol G_t(z,\Sigma)\Big\},  
 ~~\ul H_t(z,\g) = \sup_{a_1\in A_1}  \inf_{\Sigma\in{\bf\Sigma}_t}\Big\{\frac12{\rm Tr}\big[\Sigma\g\big]+\ul g_t(z,\Sigma,a_1)\Big\},\q\mbox{where}
\]
 \begin{equation}\label{eq:G} 
  \Sigma_t(a_0):=\big(\sigma_t\sigma_t^\top\big)(a_0),\; {\bf \Sigma}_t:=\big\{\Sigma_t(a_0),\; a_0\in A_0\big\},
 ~~A_0(t, \Sigma):=\big\{a_0\in A_0: \big(\sigma_t\sigma_t^\top\big)(a_0)=\Sigma\big\},
 \end{equation}
\[
  \ol G_t(z,\Sigma):=\inf_{a_0\in A_0(t,\Sigma)}\sup_{a_1\in A_1} F_t(z, a_0, a_1),
  ~~ \ul g_t(z,\Sigma,a_1):=\inf_{a_0\in A_0(t,\Sigma)}F_t(z, a_0, a_1),\; \ul G_t(z,\Sigma):=\sup_{a_1\in A_1}\ul g_t(z,\Sigma,a_1).
  \]
Moreover, it is clear that $\cP_0(\a) = \cP^1_0(\a^0) =: \cP_0(\a^0)$ depends only on $\a^0$. Denote $\cP_0 := \cup_{\a^0\in \ol\cA^0} \cP_0(\a^0)$. 
We now introduce the following 2BSDEs with solutions defined similarly to Definition \ref{def:2BSDE}.
 \begin{align}\label{eq:2bsde2}
 \ul Y_t
& =
 \xi+\int_t^T\ul G_s\big ( \ul Z_s,\widehat{\sigma}^2_s\big)\mathrm{d}s
 -\int_t^T\ul Z_s\cdot \mathrm{d}X_s-\int_t^T\mathrm{d} \ul K_s,
 ~\P-\mbox{a.s., for all}~\P\in\Pc_0,\\
 \label{eq:2bsde3}
 \ol Y_t
& =
 \xi+\int_t^T\ol G_s\big(\ol Z_s,\widehat{\sigma}^2_s\big)\mathrm{d}s
 -\int_t^T\ol Z_s\cdot \mathrm{d}X_s-\int_t^T\mathrm{d} \ol K_s,
 ~\P-\mbox{a.s., for all}~\P\in\Pc_0.
 \end{align}

\begin{theorem}
\label{thm:2bsde2}
Under Assumptions \ref{assum-coef} and \ref{assump:2bsde}, the {\rm 2BSDEs} \eqref{eq:2bsde2} and \eqref{eq:2bsde3} have unique solutions with
 \begin{eqnarray*}
 \ol V_0=\inf_{\P\in\Pc_0}\E^{\P}\big[\ol Y_0\big],
 &\mbox{and}&
 \ul V_0=\inf_{\P\in\Pc_0}\E^{\P}\big[\ul Y_0\big].
 \end{eqnarray*}
If moreover the Isaac's--like condition $\ol G=\ul G$ holds, then the game has a value.
\end{theorem}

\begin{remark}
The conditions $\ol G=\ul G$ and $\ul H=\ol H$ are equivalent in the uncontrolled volatility setting of Section \ref{sect:hamadene1995zero}. However they are not comparable in  general. It can be readily checked that when a saddle--point exists, the condition $\ol G=\ul G$ actually implies that $\ul H=\ol H$ when evaluated at this saddle--point, which is exactly what is needed to conclude to the existence of a game value. Unlike the general case of Theorem \ref{thm:2bsde} $(ii)$, we can conclude here the existence of a game value, without having to assume existence of a saddle--point. 
\end{remark}

\section{Dynamic programming principle and viscosity property}
\label{sect-Vt}
\setcounter{equation}{0}
In this section we prove Theorem \ref{thm-gamevalue} under the additional {\color{black} Assumption \ref{assum-bsi}.  
As usual,} the main tool is the dynamic programming principle. We shall focus on the upper value, and the lower value can be analysed similarly. For any $t\in [0, T]$, denote
\begin{equation}
\label{sit}
b^t(s, a) := b((t+s)\wedge T, a),\; \si^t(s, a) := \si((t+s)\wedge T, a), \; (s,a)\in [0, T] \times A.
\end{equation}
We define $\cP(t,\a)$, $\cA_t$,..., in an obvious  way, by replacing $(b,\si)$ with $(b^t, \si^t)$. 
For each $\dbP\in \cP(t,\a)$, define
\begin{equation}\label{Vt}
J_t(\o,\a,\dbP) := \dbE^\dbP\bigg[\xi^{t,\omega}+\int_0^{T-t} \!\!f^{t,\omega}_s(\alpha_s)ds\bigg],
~\ol J_t :=  \sup_{\dbP\in \cP(t, \a)} J_t(.,\dbP),
~\mbox{and}~
\ol V_t(\o)
 :=  \inf_{\alpha^0\in\cA_t^{0}} \sup_{\a^1\in \cA^1_t} \ol J_t(\o, \a^0, \a^1).
\end{equation}
As we will explain in Remark \ref{rem-DPP} below, we are not able to establish the DPP for $\ol V$ directly. To get around of this difficulty, as an  intermediate step we shall modify the upper value slightly. Following the idea of Pham and Zhang \cite{pham2014two}, we restrict $\a^0$ to  a class of appropriately defined simple processes. We note that in \cite{pham2014two} both players are restricted to such simple controls, while here we only need this restriction for the one playing first.

\vspace{0.5em}
Let $\cA^{0,{\rm pc}}_t$ denote the subset of $\cA^0_t$ whose elements take the following form
 \begin{equation}
 \label{pc}
 \a^0_s(\o) 
 \;:=\; 
 \sum_{i=0}^{n-1} \sum_{j=1}^{m_i} a^0_{ij} \1_{E_{ij}}(\o) \1_{[t_i, t_{i+1})}(s), \; 0\le s\le T-t,\; \o\in \O,
 \end{equation}
 where  $0=t_0<\cds<t_n=T-t$, $\{E_{ij}\}_{1\le j\le m_i} \subset \cF_{t_i}$ is a partition of $\O$, and $a^0_{ij}\in A_0$ are constants. Define
 \begin{equation}
 \label{Vpc}
  \ol V^{{\rm pc}}_t(\o)
 \;:=\;  \inf_{\alpha^0\in\cA_t^{0,{\rm pc}}} \sup_{\a^1\in \cA^1_t} \ol J_t(\o, \a^0, \a^1).
 \end{equation}
 It is clear that 
 \begin{equation}
 \label{VpcV}
 \ol V_t 
 \;\le\; 
 \ol V^{{\rm pc}}_t.
\end{equation}
Moreover, the uniform continuity and the boundedness in Assumption \ref{assum-coef}, induce the following regularity immediately. 

\begin{lemma}
\label{lem-Vreg}
Under Assumptions \ref{assum-coef} and  \ref{assum-bsi}, 
$J$ is uniformly continuous in $\o$, uniformly in $(t,\a,\dbP)$. Consequently,  $\ol V$ and $\ol V^{{\rm pc}}$ are uniformly continuous in $\o$, uniformly in $t$.
\end{lemma}

We emphasise that 
the maps  $b$ and $\si$, and hence the probability measure $\dbP$, do not depend on $\o$. When $b$ or $\si$ depends on $\o$, the class $\cP(t,\a)$ will depend on $\o$, and thus one cannot fix an arbitrary $\dbP$ to discuss the regularly of $J$ in terms of $\o$. Indeed, in this general case, the regularity of $\ol V^{{\rm pc}}$ is the major difficulty in our approach, and we shall investigate it further in Section \ref{sect-extension} below. 

\vspace{0.5em}
The following result is crucial for the dynamic programming principle which will be established next, and is an immediate consequence of \cite[Remark 3.8 and Theorem 3.4]{karoui2013capacities2}. 

\begin{lemma}\label{lem-cP}
{\rm (i)} Let $\a= (\a^0, \a^1)\in \cA^{0,{\rm pc}}_0 \times \cA^1_0$, $t\in [0, T]$, and $\dbP\in \cP(\a)$.  Then $\a^{t,\o}\in  \cA^{0,{\rm pc}}_t \times \cA^1_t$ and the regular conditional probability distribution $\dbP^{t,\o} \in \cP(t, \a^{t,\o})$ for $\dbP$-a.e. $\o\in \O$. 
\\
{\rm (ii)} Let $t\in [0, T]$, $n\in\dbN$, $\a^0\in \cA^{0,{\rm pc}}_0$, $\a^0_i \in \cA^{0,{\rm pc}}_t$, $1\le i\le n,$ $\{E_i\}_{1\le i\le n}\subset \cF_t$ a finite partition of $\O$, and define the control
$\widetilde \a^0(\o\otimes_t \o') := \1_{[0, t)}\a^0(\o)  +\1_{[t, T]} \sum_{i= 1}^n \a^0_i(\o')\1_{E_i}(\o).$
 Then $\widetilde\a^0\in \cA^{0,{\rm pc}}_0$.
 \end{lemma}

We now prove the dynamic programming principle (DPP) for $\ol V^{{\rm pc}}$.
\begin{theorem}
\label{thm-DPP}
Under Assumptions \ref{assum-coef} and \ref{assum-bsi}, 
 $\ol V^{{\rm pc}}\in{\rm UC}_b(\Theta)$, and 
\begin{equation*}
\ol V^{{\rm pc}}_t(\o) = \inf_{\a^0\in \cA^{0,{\rm pc}}_t} \sup_{\a^1\in \cA^1_t}  \sup_{\dbP\in \cP(t, \a)} \dbE^\dbP\bigg[\big(\ol V^{{\rm pc}}\big)^\th_\t+\int_0^\t f^\th_s(\alpha_s)\mathrm{d}s\bigg],
~\mbox{for any}~\th \in \Th,~\mbox{and $\dbF-$stopping time}~\t \le T-t.
\end{equation*}
\end{theorem}
\proof We proceed in three steps.

\vspace{0.5em}
{\bf Step 1.} We first establish the DPP for deterministic $\t$. Without loss of generality we will just prove the following 
 \begin{equation}
\label{DPP0} 
\ol V^{{\rm pc}}_0 = \inf_{\a^0\in \cA^{0,{\rm pc}}_0} \sup_{\a^1\in \cA^1_0}  \sup_{\dbP\in \cP(\a)} \widetilde J_0(\a, \dbP),
~\mbox{where}~  \widetilde J_0(\a, \dbP) := \dbE^\dbP\bigg[ \ol V^{{\rm pc}}_{t}+\int_0^{t} f_s(\alpha_s)\mathrm{d}s\bigg]
~\mbox{for}~0<t\le T.
\end{equation}

\hspace{2em}{\it Step 1.1.} We first prove $\le$. By Lemma \ref{lem-Vreg}, $J$ is uniformly continuous in $\o$.  Since $\O$ is separable, for any $\eps>0$, there exist a partition $\{E_i\}_{i\ge 1} \subset \cF_{t}$ and $\o^i\in E_i$ such that $|J_{t}(\o, \a, \dbP) - J_{t}(\o^i,\a,\dbP)| \le \eps$ for all $\o\in E_i$, $i\ge 1$, and all $\a\in \cA$ and $\dbP\in \cP(t, \a)$. For each $i$, let $\a^0_i\in \cA^{0,{\rm pc}}_{t}$ be an $\eps-$optimiser for $\ol V^{{\rm pc}}_{t}(\o^i)$, namely 
\begin{equation}
\label{S}
S_{t}(\o^i,\a^0_i) \le \ol V^{{\rm pc}}_{t}(\o^i) + \eps,\;\mbox{where} \;  S_t(\o, \a^0_i):= \sup_{\a^1 \in \cA^1_{t}} \sup_{\dbP\in \cP(t, \a^0_i, \a^1)} J_t(\o, \a^0_i, \a^1, \dbP).
\end{equation}
Notice that $\a^0_i$ does not depend on $\o$. Then it is clear that
\begin{equation}
\label{Sreg}
|S_t(\o, \a^0_i)- S_t(\o^i, \a^0_i)|\le \eps, \;\mbox{and}\; |\ol V^{{\rm pc}}_{t}(\o) - \ol V^{{\rm pc}}_{t}(\o^i)|\le \eps,\;\mbox{for all}~ \o\in E_i,\; i\ge 1.
\end{equation}
 Now for any $\a=(\a^0,\a^1) \in \cA^{0,{\rm pc}}_0\times \cA^1_0$ and $\dbP\in \cP(\a)$, define $\widetilde \a^0\in \cA^0$ as in Lemma \ref{lem-cP} (ii). Note that $(\widetilde \a^0)^{t, \o} = \a^0_i$ for $\o\in E_i$. By Lemma \ref{lem-cP} (i), we obtain
\begin{align*}
\ol J_0(\widetilde \a^0, \a^1, \dbP) 
&=  \dbE^\dbP\Bigg[  \dbE^{\dbP^{t,\o}}\bigg[\xi^{t,\o} + \int_0^{T-t} f^{t,\o}_s\big(\big(\widetilde\a^0\big)^{t,\o}_s, \big(\a^1\big)^{t,\o}_s\big)\mathrm{d}s\bigg] + \int_0^{t} f_s(\a_s)\mathrm{d}s\Bigg]\\
&\le \dbE^\dbP\bigg[  \sum_{i\ge 1} S_{t}(\o, \a^0_i) \1_{E_i}(\o) + \int_0^{t} f_s(\a_s)\mathrm{d}s \bigg]\\
&\le  \dbE^\dbP\bigg[ \sum_{i\ge 1} S_{t}(\o^i, \a^0_i) \1_{E_i}(\o) + \int_0^{t} f_s(\a_s)\mathrm{d}s\bigg] +\eps\\
&\le   \dbE^\dbP\bigg[ \sum_{i\ge 1} \ol V^{{\rm pc}}_{t}(\o^i) \1_{E_i}(\o) + \int_0^{t} f_s(\a_s)\mathrm{d}s\bigg] +2\eps\\
&\le   \dbE^\dbP\bigg[  \sum_{i\ge 1}\ol V^{{\rm pc}}_{t}(\o) \1_{E_i}(\o) + \int_0^{t} f_s(\a_s)\mathrm{d}s \bigg] +3\eps
=  \dbE^\dbP\bigg[ \ol V^{{\rm pc}}_{t}(\o)  + \int_0^{t} f_s(\a_s)ds \bigg] +3\eps = \widetilde J_0(\a, \dbP) + 3\eps.
\end{align*}
By the arbitrariness of $\a^0, \a^1, \dbP$ and $\eps$, we prove  the "$\le$" part of  \eqref{DPP0}.

\vspace{0.5em}
\hspace{2em}{\it Step 1.2.} We next prove $\ge$. Fix $\a^0 = \sum_{i=0}^{n-1} \1_{[t_i, t_{i+1})} \sum_{j=1}^{m_i} a^0_{ij} \1_{E_{ij}} \in \cA^{0,{\rm pc}}_0$ as in \eqref{pc}. By otherwise adding the point $t$ in \eqref{DPP0} into the partition points in the definition of $\a^0$, we may assume without loss of generality that $t=t_i$ for a certain $i$. We claim that 
\begin{equation}
\label{DPP1}
 \sup_{\a^1\in \cA^1_0}  \sup_{\dbP\in \cP(\a^0,\a^1)}\dbE^\dbP\big[\ol V^{{\rm pc}}_{t_i}\big] 
 \;\le\;  
 \sup_{\a^1\in \cA^1_0}  \sup_{\dbP\in \cP(\a^0,\a^1)}\dbE^\dbP\big[\ol V^{{\rm pc}}_{t_{i+1}}\big],\;
 \mbox{for all} \;
 i=0,\cds, n-1,
 \end{equation}
so that, since $\ol V^{{\rm pc}}_{t_n} = \xi$
 \begin{equation*}
  \sup_{\a^1\in \cA^1_0}  \sup_{\dbP\in \cP(\a^0,\a^1)}\dbE^\dbP[\ol V^{{\rm pc}}_{t_i}] 
  \le
  \sup_{\a^1\in \cA^1_0}  \sup_{\dbP\in \cP(\a^0,\a^1)}\dbE^\dbP[\xi],\;
  \mbox{for all}\;
  i=0,\cds, n-1,
\end{equation*}
thus implying the "$\ge$" part of \eqref{DPP0} by the arbitrariness of $\a^0$.

\vspace{0.5em}
To see that \eqref{DPP1} holds, fix $i$,  $\a^1\in \cA^1_0$, and  $\dbP\in \cP(\a^0,\a^1)$. For any $\eps>0$, by Pham and Zhang \cite[Lemma 4.3]{pham2014two}, one may choose the partition $\{E_k\}_{k \ge 1}\subset \cF_{t_i}$ such that 
\[\sup_{\a\in \cA} \sup_{\dbP\in \cP(\a)} \dbP\bigg[\bigcup_{k> N}E_k\bigg] \le \eps,\; \text{for some $N$ large enough}.\]
Denote $\widetilde E_N := \cup_{k> N}E_k$, $E_{ijk} := E_{ij} \cap E_k$ for $j=1,\cds, m_i$, $k=1,\cds, N$, and $\widetilde E_{ij} := E_{ij} \cap \widetilde E_N$, for $j=1,\cds, m_i$.  Fix an $\o^{ijk}\in E_{ijk}$ for each $j, k$, whenever $E_{ijk} \neq \emptyset$. Recalling that $\{E_{ij}\}_{1\le j\le m_i}$ is a partition of $\O$, we have
\begin{align*}
\dbE^\dbP\big[\ol V^{{\rm pc}}_{t_i}\big] 
=
\dbE^\dbP\Bigg[\sum_{j=1}^{m_i} \ol V^{{\rm pc}}_{t_i}(\o) \1_{E_{ij}}(\o)\Bigg] &\le \dbE^\dbP\Bigg[\sum_{j=1}^{m_i}\sum_{k=1}^N \ol V^{{\rm pc}}_{t_i}(\o) \1_{E_{ijk}}(\o)\Bigg] + C\eps
\\
&\le
 \dbE^\dbP\Bigg[\sum_{j=1}^{m_i}\sum_{k=1}^N \ol V^{{\rm pc}}_{t_i}(\o^{ijk}) \1_{E_{ijk}}(\o)\Bigg] +C\eps
\\
&\le
 \dbE^\dbP\Bigg[\sum_{j=1}^{m_i} \sum_{k=1}^N\sup_{\widetilde \a^1 \in \cA^1_{t_i} }\sup_{\widetilde \dbP \in \cP(t_i, a^0_{ij}, \widetilde \a^1)} \dbE^{\widetilde \dbP}\big[(\ol V^{{\rm pc}}_{t_{i+1}})^{t_i, \o^{ijk}}\big]   \1_{E_{ijk}}(\o)\Bigg] +C\eps, 
\end{align*}
where the first inequality is due to the boundedness of $\ol V^{\rm pc}$ and  the last inequality used {\it Step 1.1}. For each $j, k$, there exist  $\widetilde \a^1_{jk} \in \cA^1_{t_i}$ and $\widetilde \dbP_{jk} \in \cP(t_i, a^0_{ij}, \widetilde \a^1_{jk})$ such that
 \begin{equation*}
  \sup_{\widetilde \a^1 \in \cA^1_{t_i} }
  \sup_{\widetilde \dbP \in \cP(t_i, a^0_{ij}, \widetilde \a^1)} \dbE^{\widetilde \dbP}\Big[(\ol V^{{\rm pc}}_{t_{i+1}})^{t_i, \o^{ijk}}\Big]   
  \le
  \dbE^{\widetilde \dbP_j}\Big[(\ol V^{{\rm pc}}_{t_{i+1}})^{t_i, \o^{ijk}}\Big] + \eps.
  \end{equation*}
Fix an arbitrary $a_1\in A_1$ and $\widetilde \dbP^{a^0_{ij}, a_1} \in  \cP(t_i, a^0_{ij},a_1)$. Define 
\begin{align*}
\widetilde \a^1 &:= \1_{[0, t_i) } \a^1+ \1_{[t_i, T]} \Bigg(\sum_{j=1}^{m_i} \sum_{k=1}^N \widetilde \a^1_{jk} \1_{E_{ijk}} + a_1 \1_{\widetilde E_N}\Bigg),\; \widetilde \dbP := \dbP\otimes_{t_i} \sum_{j=1}^{m_i} \Bigg(\sum_{k=1}^N \widetilde \dbP_{jk} \1_{E_{ijk}} + \widetilde \dbP^{a^0_{ij}, a_1}\1_{\widetilde E_{ij}}\Bigg) \in \cP(\a^0, \hat \a^1).
\end{align*}
 Then
 \begin{align*}
 \dbE^\dbP\Big[\ol V^{{\rm pc}}_{t_i}\Big] 
& \le 
 \dbE^\dbP\Bigg[\sum_{j=1}^{m_i}\sum_{k=1}^N  \dbE^{\widetilde \dbP_{jk}}\Big[(\ol V^{{\rm pc}}_{t_{i+1}})^{t_i, \o^{ijk}}\Big] \1_{E_{ijk}}(\o)
                 \Bigg] +C\eps \\
                 &\le
  \dbE^\dbP\Bigg[\sum_{j=1}^{m_i} \sum_{k=1}^N  \dbE^{\widetilde \dbP_{jk}}\Big[(\ol V^{{\rm pc}}_{t_{i+1}})^{t_i, \o}\Big] \1_{E_{ijk}}(\o)
                  \Bigg] 
 +C\eps
 =
 \dbE^\dbP\Bigg[\sum_{j=1}^{m_i}  \sum_{k=1}^N \dbE^{ \widetilde \dbP^{t_i, \o} }\Big[(\ol V^{{\rm pc}}_{t_{i+1}})^{t_i, \o}\Big] \1_{E_{ijk}}(\o)
                 \Bigg] 
 +C\eps
 \\
 &=
 \dbE^\dbP\Bigg[\dbE^{ \widetilde \dbP^{t_i, \o} }\Big[(\ol V^{{\rm pc}}_{t_{i+1}})^{t_i, \o}\Big]  \1_{\widetilde E_N^c}\Bigg] +C\eps 
 \le  \dbE^\dbP\Bigg[\dbE^{ \widetilde \dbP^{t_i, \o} }\Big[(\ol V^{{\rm pc}}_{t_{i+1}})^{t_i, \o}\Big]  \Bigg] +C\eps  =
 \dbE^{\widetilde \dbP}\Big[\ol V^{{\rm pc}}_{t_{i+1}}  \Big] 
 +C\eps.
\end{align*}
This leads to \eqref{DPP1} immediately.

\vspace{0.5em}
{\bf Step 2.}  We next show that $\ol V^{{\rm pc}}$ is uniformly continuous in $(t,\o)$.  Let $\rho$ denote the modulus of continuity function of $\ol V^{{\rm pc}}$ with respect to $\o$. For $t < t'$ and $\o\in \O$, denote $\delta := d_\infty((t,\o), (t',\o))$. By {\bf Step 1}, we have
\begin{align*}
\big|\ol V^{{\rm pc}}_{t} - \ol V^{{\rm pc}}_{t'} \big|(\o) &\le \sup_{\a^0\in \cA^{0,{\rm pc}}_t} \sup_{\a^1\in \cA^1_t}  \sup_{\dbP\in \cP(t, \a^0,\a^1)} \bigg| \dbE^{\dbP}\bigg[ \ol V^{{\rm pc}}_{t'}(\o\otimes_t X)+\int_0^{t'-t} f^{t,\o}_s(\a^0_s, \a^1_s) \mathrm{d}s\bigg] - \ol V^{{\rm pc}}_{t'}(\o)\bigg| \\
 &\le \sup_{\a^0\in \cA^{0,{\rm pc}}_t} \sup_{\a^1\in \cA^1_t}  \sup_{\dbP\in \cP(t, \a^0,\a^1)} \dbE^{\dbP}\bigg[ \rho\big(\|(\o\otimes_t X)_{\cd\wedge t'} - \o_{\cd\wedge t'}\|_\infty\big)+ \int_0^\delta  |f^{t,\o}_s(\a^0_s, \a^1_s)| \mathrm{d}s \bigg]  \\
  &\le \sup_{\a^0\in \cA^{0,{\rm pc}}_t} \sup_{\a^1\in \cA^1_t}  \sup_{\dbP\in \cP(t, \a^0,\a^1)} \dbE^{\dbP}\bigg[ \rho\big( \delta + \|X_{\cd\wedge\delta}\|_\infty\big)+  \int_0^\delta  |f^{t,\o}_s(\a^0_s, \a^1_s)| \mathrm{d}s \bigg]  \\
                           &\le \sup_{\dbP\in \cP_L} \dbE^\dbP\Big[\rho\big( \delta + \|X_{\cd\wedge\delta}\|_\infty\big)\Big] + C\delta,
\end{align*}
for some $L$ large enough. Under our conditions, this clearly implies the uniform continuity of $\ol V^{{\rm pc}}$ in $(t,\o)$.

\vspace{0.5em}
{\bf Step 3.} Following standard approximation arguments, we may extend Theorem \ref{thm-DPP} to stopping times.
\qed

\begin{remark}
\label{rem-DPP}
 $(i)$ Following similar arguments as in {\it Step 1.1} above, one may prove the partial {\rm DPP} for $\ol V$:
\begin{equation*}
\ol V_0 \le \inf_{\a^0\in \cA^0_0} \sup_{\a^1\in \cA^1_0}  \sup_{\dbP\in \cP(\a^0,\a^1)} \dbE^\dbP\bigg[ \ol V_{t}+\int_0^{t} f_s(\alpha_s)\mathrm{d}s \bigg].
\end{equation*}
However, to prove the opposite inequality, we encountered some serious difficulties that we would like to highlight. Let $f=0$ for simplicity of presentation. Then we want to prove, for fixed $\a^0 \in \cA^0_0$, that
 \begin{equation}\label{DPP2}
 \sup_{\a^1\in \cA^1_0}  \sup_{\dbP\in \cP(\a^0,\a^1)} \dbE^\dbP[ \xi] 
  \ge 
  \sup_{\a^1\in \cA^1_0}  
  \sup_{\dbP\in \cP(\a^0,\a^1)} \dbE^\dbP\Bigg[\sup_{\widetilde \a^1 \in \cA^1_t} \sup_{\widetilde \dbP\in \cP(t, (\a^0)^{t,\o}, \widetilde \a^1)} \dbE^{\widetilde \dbP}[\xi^{t,\o}] \Bigg],
\end{equation}
which is the partial {\rm DPP} for a control problem $($instead of a game problem$)$. However, we insist on the fact that we are using weak formulation, which implies in particular that the control $\a^1$ depends on $X$, and the fixed control $\a^0$ has typically absolutely no regularity with respect to $X$. The {\rm DPP} for this problem is not available in the literature, and there are indeed serious obstacles to overcome in order to establish it. Several authors managed to obtain such {\rm DPP} but either in strong formulation $($see Nisio {\rm\cite{nisio1988stochastic}}, Fleming and Souganidis {\rm \cite{fleming1989existence}}, \'Swi\k{e}ch {\rm\cite{swikech1996another}},  Buckdahn and Li {\rm \cite{buckdahn2008stochastic}}, Bouchard, Moreau and Nutz {\rm\cite{bouchard2014stochastic}}, Bouchard and Nutz {\rm\cite{bouchard2015stochastic}}, Krylov {\rm\cite{krylov2013dynamic,krylov2014dynamic}}$)$, or with the use of simple strategies for both players $($see Pham and Zhang {\rm\cite{pham2014two}} or S\^irbu {\rm\cite{sirbu2014stochastic,sirbu2015asymptotic}}$)$\footnote{A slight exception would be Kovats \cite{kovats2009value}, which considers games written somehow in a weak formulation, but with strategies against control, and relies on approximation techniques similar to \cite{fleming1989existence}.}.

\vspace{0.5em}
$(ii)$ As we saw in Example \ref{eg-strong}, in the strong formulation setting of Subsection \ref{sect:strong}, the game value typically does not exist.  The main reason is that in this setting the DPP fails for the dynamic upper $($and lower$)$ value of the game. In fact, in this case the dynamic version of $\ol V^S_0$ in \eqref{Strong-V0} is the following deterministic function $($assuming $b=f=0$ for simplicity$)$,
\begin{equation*}
\dis\ol u(t, x) := \inf_{\a^0\in \cA^0_S}\sup_{\a^1\in \cA^1_S} \dbE^{\dbP_0}\big[g(X^{t,x, \a}_{T-t}) \big],\; \mbox{where}\; X^{t,x,\a}_s = x +  \int_0^s \si^t(r, X^{t,x,\a}_r, \a_r) \mathrm{d}W_r,
~0\le s\le T-t, ~\dbP_0-\mbox{a.s.}
\end{equation*}
Under mild conditions, one can easily show that $\ol u(t, \cd)$ is uniformly continuous in $x$. Following similar arguments as in Step 1.2 $($actually much easier because we have the desired regularity under the strong formulation$)$, one can prove the following partial {\rm DPP} $($stated for $t_1=0, t_2 = t$ and $x=0$ for simplicity$)$, 
\begin{equation*}
\ol u(0,0) \ge  \inf_{\a^0\in \cA^0_S}\sup_{\a^1\in \cA^1_S} \dbE^{\dbP_0}\big[\ol u(t, X^\a_t) \big].
\end{equation*}
However, the opposite direction of the {\rm DPP} will not hold, since the game value does not exist. Let us explain why the arguments in Step 1.1 fail in this setting. 

\vspace{0.5em}
By the uniform regularity of $\ol u$ in $x$, there exists a partition $\{O_i\}_{i\ge 1}$ of $\dbR^d$ such that $|\ol u(t, x) - \ol u(t, x_i)|\le \eps$ for all $x\in O_i$, where $x_i \in O_i$ is fixed. Now for each $i$, let $\a^0_i \in \cA^0_S$ be an $\eps-$optimiser of $\ol u(t, x_i)$. Then we will have, for any $\a\in \cA_S$
\begin{equation*}
  \dbE^{\dbP_0}\big[\ol u(t, X^\a_t) \big] \ge  \dbE^{\dbP_0}\Bigg[ \sum_{i\ge 1} \sup_{\widetilde \a^1 \in \cA^1_S} \dbE^{\dbP_0} \Big[g(X^{t, x, (\a^0_i, \widetilde \a^1)}_{T-t}\Big]\Big|_{x=X^\a_t} \1_{O_i}(X^\a_t)\Bigg]  - \eps\ge\sup_{\widetilde \a^1 \in \cA^1_S}  \dbE^{\dbP_0} \Big[ g\big(\widetilde X_T\big)\Big] -\eps,
 \end{equation*}
 where, denoting $W^t_s := W_{t+s}-W_t$, 
 \begin{equation}
 \label{tildeX}
 \widetilde X_s = X^\a_t + \int_t^s \si\bigg(r, \widetilde X_r, \sum_{i\ge 1} \a^0_i(r-t, W^t_\cd) \1_{O_i}(X^\a_t), \widetilde \a^1(r-t, W^t)\bigg) \mathrm{d}r, \q\dbP_0-\mbox{a.s.}
 \end{equation}
 In other words, to prove the opposite direction of the {\rm DPP}, essentially we want to prove
 \begin{equation}
 \label{DPpeng2015g2}
  \inf_{\a^0\in \cA^0_S}\sup_{\a^1\in \cA^1_S} \dbE^{\dbP_0}\big[g(X^\a_T) \big] \le  \inf_{\a^0\in \cA^0_S}\sup_{\a^1\in \cA^1_S} \sup_{\widetilde \a^1 \in \cA^1_S}  \dbE^{\dbP_0} \Big[ g\big(\widetilde X_T\big)\Big]  .
 \end{equation}
 Fix an $\a^0\in \cA^0_S$. Note that the right side above involves only $\a^0\big|_{[0, t]}$.  The idea is to construct $\widetilde \a^0 \in \cA^0_S$ such that 
 \begin{equation*}
 \sup_{\a^1\in \cA^1_S} \dbE^{\dbP_0}\big[g(X^{\widetilde \a^0, \a^1}_T) \big] \le \sup_{\a^1\in \cA^1_S} \sup_{\widetilde \a^1 \in \cA^1_S}  \dbE^{\dbP_0} \big[ g(\widetilde X_T)\big].  
 \end{equation*}
 By \eqref{tildeX}, the most natural construction of $\widetilde \a^0$ is to set $\widetilde \a^0(s,W_\cd) := \1_{[0, t)}(s) \a^0(s,W_\cd) + \1_{[t, T]} \sum_{i\ge 1} \a^0_i(s-t, W^t_\cd) \1_{O_i}(X^\a_t)$.
This construction does work well in weak formulation, as we saw in Lemma \ref{lem-cP} $(ii)$. However, in the strong formulation considered here, the above construction relies on $X^\a_t$, thus in turn on $\a^1\big|_{[0, t]}$. In other words, the $\a^0$ in the left side of \eqref{DPpeng2015g2} will depend on $\a^1$. This is exactly the idea of the notion of strategy against control introduced in Subsection \ref{sect:fleming1989existence}, which is however not allowed in the current setting of strong formulation with control against control. 
\end{remark}
  
We now derive the viscosity property of $\ol V^{{\rm pc}}$. Once the DPP and the regularity of $\ol V^{{\rm pc}}$ have been established, it is a rather straightforward verification.

\begin{proposition}
\label{prop-viscosity}
Under Assumptions \ref{assum-coef} and \ref{assum-bsi}, 
$\ol V^{{\rm pc}}$ is a viscosity solution of the {\rm PPDE}
\begin{equation}
\label{PPDEol}
-\pa_t \ol V^{{\rm pc}} - \ol H_t\big(\pa_\o \ol V^{{\rm pc}}, \pa^2_{\o\o} \ol V^{{\rm pc}}\big) 
=
0.
\end{equation} 
\end{proposition}
\proof Without loss of generality, we shall only verify the $L-$viscosity property at $(0,0)$ for some $L$ large enough.  

\vspace{0.5em}
{\it Step 1.} We first verify the viscosity sub--solution property. Assume by contradiction that there exists $(\kappa, z, \g) \in  \ul\cJ^L \ol V^{{\rm pc}}(0,0)$ with corresponding hitting time $\ch_\eps$ such that $-c:= \kappa +\ol H_0(z, \g) < 0$.
By the definition of $\ol H$, there exists $a_0^*\in A_0$ such that $\kappa +\sup_{a_1\in A_1} h_0(z, \g, a^*_0, a_1) \le -{\frac c2} < 0$. By choosing $\eps>0$ small enough, it follows from the uniform regularity of $b, \si$ and $f$ that
\begin{equation}
\label{small1}
 \kappa +\sup_{a_1\in A_1} h_t(\o,  z, \g, a^*_0, a_1) \le -{\frac c3} < 0,\; 0\le t \le \ch_\eps(\o).
\end{equation}
Now fix the above $\eps>0$. For an arbitrary $\delta>0$, denote $\t := \ch_\eps \wedge \delta \le \ch_\eps$. By \eqref{cJ} we have
\begin{eqnarray}
\label{small2}
-\ol V^{{\rm pc}}_0 
&\le& 
\inf_{\dbP\in \cP_L} \dbE^\dbP\bigg[\kappa \t + z\cd X_\t + {\frac12} {\rm Tr}\big[\g X_\t X^\top_\t\big] - \ol V^{{\rm pc}}_\t\bigg].
\end{eqnarray}
On the other hand, by setting $\a^0$ as the constant process $a^*_0$ in the right side of DPP \eqref{DPP0}, we have
\begin{eqnarray*}
 \overline{V}^{{\rm pc}}_0 
 &\le&  
 \sup_{\alpha^1\in \Ac^1_0}  \sup_{\mathbb P\in \Pc(a^*_0,\alpha^1)}  \mathbb E^{\mathbb P}\bigg[ \overline{V}^{{\rm pc}}_\tau +\int_0^\tau f_s(a^*_0, \a^1_s) \mathrm{d}s \bigg].
\end{eqnarray*}
Choose $\a^1\in \cA^1_0$ and $\dbP\in \cP(a^*_0,\a^1)$ such that
\begin{equation}
\label{small3}
 \overline{V}^{{\rm pc}}_0 \leq  \mathbb E^{\mathbb P}\Big[ \overline{V}^{{\rm pc}}_\tau +\int_0^\tau f_s(a^*_0, \a^1_s) \mathrm{d}s \Big]+ \delta^2.
\end{equation}
Note that $\cup_{\a\in \cA} \cP(\a)  \subset \cP_L$ for $L$ large enough, and in particular the above $\dbP$ is also in $\cP_L$. Then, we derive from \eqref{small2} and \eqref{small3} that
\begin{align*}
 -\delta^2 
 \le
 \dbE^{\dbP}\Big[\kappa \t + z\!\cd\! X_\t + {\frac12} {\rm Tr}\big[\g X_\t X^\top_\t\big] 
                            \!+\!\!\int_0^\t\!\!\! f_s(a^*_0, \a^1_s) \mathrm{d}s 
                    \Big]
 =
 \dbE^{\dbP}\bigg[\int_0^\t \!\!\!\big(\kappa + h_s(z,\g, a^*_0, \a^1_s) +  {\rm Tr}\big[\g  b(s, a^*_0,  \a^1_s) X^\top_s\big] \big)\mathrm{d}s \bigg] .
  \end{align*}
 Now by \eqref{small1} we have
 \begin{align*}
 -\delta^2 
 \;\le\;
 \mathbb E^{\mathbb P}\Big[ -{\frac c3}  \tau + \int_0^\tau  {\rm Tr}\big[ \gamma  b(s, a^*_0,  \alpha^1_s) X^\top_s \big]\mathrm{d}s \Big] 
 \;\le\; 
 -{\frac c3} \delta + C\mathbb P\big[\ch_\eps \leq \delta\big] + C \mathbb E^\mathbb P\big[ \|X_{\cdot\wedge\delta}\|_\infty  \big] \delta .
 \end{align*}
Clearly $\dbE^\dbP\big[ \|X_{\wedge\delta}\|_\infty  \big] \le C_L\sqrt{\delta}$. Moroever, for $\delta \le {\frac\eps2}$,
 \begin{equation*}
 \mathbb P\big[\ch_\eps \leq \delta\big]  
 \;\le\; 
 \mathbb P\big[\delta +\|X_{\cdot\wedge\delta}\|_\infty \geq \eps\big]  
 \;\le\;  
 \mathbb P\big[\|X_{\wedge\delta}\|_\infty \geq {\frac\eps2}\big] \leq {\frac C{\eps^4}} \mathbb E^\mathbb P\big[\|X_{\wedge\delta}\|_\infty^4\big] \le C_{\eps,L}\delta^2.
 \end{equation*}
Then $0 \le -{\frac c 3} \delta + CC_{\eps,L} \delta^2+CC_L\delta^{3/2}+\delta^2$, which leads to the desired contradiction for small $\delta>0$. 
\ms

{\it Step 2.} We next verify the viscosity super--solution property. Assume by contradiction that there exists $(\kappa, z, \g) \in  \ol\cJ^L u(0,0)$ with corresponding hitting time $\ch_\eps$ such that $c := \kappa +\ol H_0(z, \g) > 0.$ Then
\begin{equation*}
 \kappa + \sup_{a_1\in \cA_1} h_0(z, \g, a_0, a_1) \ge c
 >
 0,\; 
 \mbox{for all}\; 
 a_0 \in A_0,
 \end{equation*}
and there exists a mapping $\psi: A_0 \longrightarrow A_1$ such that $\kappa + h_0(z, \g, a_0, \psi(a_0)) \ge {\frac c2} > 0,$ for all $a_0 \in A_0$. By choosing $\eps>0$ small enough, it follows from the uniform regularity of $b, \si$ and $f$ that
\begin{equation}
\label{small4}
 \kappa +h_t(\o,  z, \g, a_0, \psi(a_0)) 
 \ge 
 {\frac c3} 
 > 0,
 \mbox{ for all }
 a_0 \in A_0, ~ 0\le t \le \ch_\eps(\o).
\end{equation}
Now fix the above $\eps>0$. For an arbitrary $\delta>0$, denote $\tau := \ch_\eps \wedge \delta \le \ch_\eps$. By \eqref{cJ} we have
\begin{equation}
\label{small5}
-\ol V^{{\rm pc}}_0 
\;\ge\;
\sup_{\dbP\in \cP_L} \dbE^\dbP\bigg[\kappa \t + z\cd X_\t + {\frac12} {\rm Tr}\big[\g X_\t X^\top_\t \big]- \ol V^{{\rm pc}}_\t\bigg].
\end{equation}
On the other hand, by the DPP \eqref{DPP0}, there exists $\a^0\in \cA^0$ such that
 \begin{equation*}
 \overline{V}^{{\rm pc}}_0 
 \geq 
 \sup_{\a^1\in \Ac^1_0}  \sup_{\mathbb P\in \Pc(\a^0,\a^1)} 
 \mathbb E^{\mathbb P}\bigg[ \overline{V}^{{\rm pc}}_\tau +\int_0^\tau f_s(\a^0_s, \a^1_s)\mathrm{d}s   \bigg] -\delta^2.
 \end{equation*}
Since $\cup_{\a^1\in \cA^1_0}  \cP(\a^0,\a^1) \subset \cP_L$,  the above estimate together with \eqref{small5} implies
\begin{eqnarray*}
\delta^2 
&\ge&
\sup_{\a^1\in \cA^1_0}  \sup_{\mathbb P\in \Pc(\a^0,\a^1)} \mathbb E^{\mathbb P}\bigg[ \kappa \tau + z\cd X_\tau + {\frac1 2}{\rm Tr}\big[ \gamma  X_\tau X^\top_\tau\big] +\int_0^\tau f_s(\a^0_s, \a^1_s)\mathrm{d}s   \bigg] 
\\
&=& \sup_{\a^1\in \cA^1_0}  \sup_{\mathbb P\in \Pc(\a^0,\a^1)} \mathbb E^{\mathbb P}\bigg[\int_0^\t \big(\kappa + h_s(z,\gamma, \a^0_s, \a^1_s) + {\rm Tr}\big[ \gamma  b(s,  \a^0_s, \a^1_s) X^\top_s\big] \big)\mathrm{d}s \bigg].
\end{eqnarray*}
Choose $\a^1 := \psi(\a^0)$. By the structure of $\a^0$, we see $\a^1$ is also piecewise constant and thus $\cP(\a^0, \a^1)\neq \emptyset$. Set $\dbP\in \cP(\a^0,\a^1)\subset \cP_L$. Then
$$
\delta^2
\ge
\dbE^{\dbP}\Big[\int_0^\t\!\!\! \big(\kappa + h_s(z,\g, \a^0_s, \psi(\a^0_s)) +  {\rm Tr}\big[\g b(s, \a^0_s, \psi(\a^0_s)) X^\top_s\big] \big)\mathrm{d}s \Big] 
\ge 
\dbE^{\dbP}\Big[\int_0^\t \!\!\!\Big({\frac c3} +  {\rm Tr}\big[\g  b(s, \a^0_s, \psi(\a^0_s)) X^\top_s\big] \Big)\mathrm{d}s \Big],
$$
thanks to \eqref{small4}. Now following the same arguments as in {\it Step 1}, we obtain $0\ge {\frac c3} \delta - C_{\eps,L}\delta^2-C\delta^{3/2}$, which leads to the desired contradiction by choosing $\delta>0$ small enough.
\qed

\vspace{0.5em}

Finally, we can prove our main result of this section.

\vspace{0.5em}
\no{\bf Proof of Theorem \ref{thm-gamevalue}.} Following similar arguments, one may define $\ul V^{{\rm pc}}_t$ by restricting $\a^1$ to piecewise constant processes in the problem $\ul V_t$, and show that $\ul V^{{\rm pc}}\in {\rm UC}_b(\Th, \dbR)$ is a viscosity solution of the {\rm PPDE}
 \begin{equation}
\label{PPDEul}
-\pa_t \ul V^{{\rm pc}} - \ul H_t(\pa_\o \ul V^{{\rm pc}}, \pa^2_{\o\o} \ul V^{{\rm pc}}) =0.
\end{equation} 
Clearly $\ol V^{{\rm pc}}_T = \xi = \ul V^{{\rm pc}}_T$. Then it follows from Isaacs's condition \eqref{Isaacs} and the uniqueness assumption for viscosity solutions of the PPDEs, that $\ol V^{{\rm pc}} = \ul V^{{\rm pc}}$. Moreover, recalling \eqref{VpcV}, we deduce $\ul V^{{\rm pc}} \le \ul V \le \ol V\le \ol V^{{\rm pc}}$, and therefore $\ol V= \ul V$.
\qed

\section{An extension}
\label{sect-extension}
\setcounter{equation}{0}
In this section we shall relax Assumption \ref{assum-bsi}, and replace the expectation $J_0(\a, \dbP)$ in \eqref{J} with the solution to a nonlinear BSDE, as in the seminal paper of Buckdahn and Li \cite{buckdahn2008stochastic}.

\vspace{0.5em}
For $\a\in \cA$ and $\dbP\in \cP(\a)$, consider the solution $(Y^{\a,\dbP}, Z^{\a,\dbP}, N^{\a,\dbP})$ of the BSDE
 \begin{equation}
 \label{BSDE}
 Y^{\a,\dbP}_t 
 =
 \xi(X) 
 + \int_t^T f_s\big(X_\cd, Y^{\a,\dbP}_s, \si_s(X_\cd, \a_s)Z^{\a,\dbP}_s, \a_s\big) \mathrm{d}s 
 - Z_s^{\a,\dbP}\cdot \mathrm{d}M^\a_s + \mathrm{d}N^{\a,\dbP}_s,\; 
 \dbP-\mbox{a.s.}
 \end{equation}
where $M^\a$ is defined in \eqref{cPa}, and $N^{\a,\dbP}$ is an $\dbP-$martingale orthogonal to $X$ (or equivalently to $M^\a$) under $\dbP$, namely 
$ \la N^{\a,\dbP}, X\ra =0$, $\dbP-\mbox{a.s.}$ Recall \eqref{cPa}, so that BSDE \eqref{BSDE} can be rewritten, $\P-$a.s.
 \begin{align}
 \label{BSDEX}
 Y^{\a,\dbP}_t 
 = 
 \xi(X) + \int_t^T \Big[f_s\big(X_\cd, Y^{\a,\dbP}_s, \si_s(X_\cd, \a_s)Z_s^{\a,\dbP}, \a_s\big) 
 + Z^{\a,\dbP}_s b_s(X_\cd, \a_s)\Big]\mathrm{d}s
 - Z_s^{\a,\dbP}\mathrm{d}X_s + \mathrm{d}N^{\a,\dbP}_s.
 \end{align}

 In this section we shall assume
 \begin{assumption}
 \label{assum-f}
 { $(i)$} $b, \si$ and $\xi$ satisfy the conditions in Assumption \ref{assum-coef}. 
 
 \vspace{0.5em}
 {$(ii)$} $f$ is $\dbF-$progressively measurable in all variables, and the function $f_t(\o, 0,0,a)$ is bounded. 
 
 \vspace{0.5em}
 {$(iii)$} $f$ is locally uniformly continuous in $(t,\o, a)$, locally uniformly in $(y,z)$. That is, for any $R>0$, there exists a modulus of continuity function $\rho_R$ such that
 \begin{equation*}
 |f_t(\o, y, z, a)-f_{t'}(\o', y,z,a)| \le \rho_R\big(d_\infty(\th, \th')\big)~\mbox{for all}~ \th, \th'\in \Th, |y|, |z|\le R, a\in A.
 \end{equation*}
 
 {$(iv)$} $f$ is uniformly Lipschitz continuous in $(y,z)$. 
  \end{assumption}
Under Assumption \ref{assum-f}, the BSDE \eqref{BSDE} is well--posed with $\dbF-$progressively measurable solutions. By abusing the notations, we redefine \eqref{J} as
\begin{equation}
\label{BJ}
J_0(\a, \dbP) :=   Y^{\a,\dbP}_0,
 \end{equation}
and  still define $\ol J_0(\a), \ul J_0(\a)$ by \eqref{J},  the upper and lower values $(\ol V_0, \ul V_0)$ by \eqref{ubar}, and a saddle--point $\hat \a$ of the game by \eqref{saddle}, but using the newly defined $J_0(\a, \dbP)$.

\vspace{0.5em}
 It is quite straightforward to extend the results in Subsection \ref{sect-verification}  to this setting. In this section, we shall focus on extending Theorem \ref{thm-gamevalue} to this general case.  We remark that the nonlinear extension to BSDE does not cause significant difficulty, and as we explained, the main difficulty is the regularity of the value functions due to the dependence of $b$ and $\si$ on $\o$.  As mentioned in Subsection \ref{sect:CRpham2014two}, in the Markovian case this difficulty can be circumvented by using the idea of S\^irbu \cite{sirbu2014stochastic, sirbu2015asymptotic}.  In this section we shall provide a sufficient condition, in addition to Assumption \ref{assum-f},  under which we are able to extend Theorem \ref{thm-gamevalue} for path--dependent games. 
 
 \vspace{0.5em}
We first notice that in this case the Hamiltonians become (again abusing notations)
\begin{equation}
\label{BHamiltonian}
\left.\begin{array}{c}
\dis h_t(\o,y, z,\g, a) :=  {\frac 12}{\rm Tr}\big[ \si\si^\top_t(\o, a) \g\big] +  b_t(\o, a) \cd z+ f_t(\o, y, z\si_t(\o, a), a),\\[0.8em]
\dis \ol H_t(\o, y,z,\g) := \inf_{a_0\in A_0}\sup_{a_1\in A_1} h_t(\o,y, z,\g, a),~  \ul H_t(\o, y,z,\g) :=\sup_{a_1\in A_1}  \inf_{a_0\in A_0} h_t(\o,y, z,\g, a).
\end{array}\right.
\end{equation}

\subsection{Drift reduction by Girsanov transformation} 
In this subsection we illustrate that there is flexibility on the drift $b$, through the Girsanov {\color{black}transformation. 
For any} $\l: \Th \times A \longrightarrow \dbR^d$, denote
\begin{eqnarray}
\label{Girsanov}
\left.\begin{array}{c}
\dis \si' := \si, \; b' := b - \si \l,\; f' := f + z\cdot \l,\; \xi' := \xi,\\[0.3em]
\dis h'_t(\o,y, z,\g, a) :=  {\frac12} {\rm Tr}\big[(\si')^2_t(\o, a)  \g\big] +  b'_t(\o, a) \cd z+ f'_t(\o, y, z\si_t(\o, a), a),
\end{array}\right.
\end{eqnarray}
and define $\cP'(\a),$ $Y^{'\a,\dbP},$ $\ol J_0'(\a)$,..., in an obvious manner. It is clear that $h' = h$, and thus the corresponding Isaacs equation will remain the same. We show that the game values are invariant under this transformation.

\begin{proposition}
\label{prop-Girsanov}
Let $b,\; \si,\; f,\; \xi$ satisfy Assumption \ref{assum-f}. Assume $\l$ is bounded, $\dbF-$progressively measurable in all variables, uniformly continuous in $(t, \o)$ under $d_\infty$, uniformly in $a \in A$, and is locally uniformly continuous in $a$, in the sense of Assumption \ref{assum-f}$(iii)$. Then $\ol J_0'(\a) = \ol J_0(\a)$ and $\ul J_0'(\a) = \ul J_0(\a)$ for any $\a\in \cA$.
\end{proposition}

\proof First, it is clear that $b', \si', f', \xi'$ also satisfy Assumption \ref{assum-f}. We proceed in two steps.

\vspace{0.5em}
{\it Step 1.} Let $\ol \O := \O \times \O$ with canonical process $(X, W)$. For any $\a\in \cA$, denote by $\ol \cP(\a)$ the set of probability measures $\ol \dbP$ on $\ol\O$ such that $W$ is a $\ol\dbP-$Brownian motion and \eqref{SDE} holds $\ol\dbP-$a.s. Then clearly $\cP(\a) = \{\dbP := \ol \dbP\circ X^{-1}: \ol\dbP\in \ol\cP(\a)\}$.
Now for each $\ol \dbP\in \ol\cP(\a)$, define 
\begin{equation*}
\mathrm{d}W^\a_t = \mathrm{d}W_t + \l_t(X_\cd, \a_t(X_\cd)) \mathrm{d}t,~ \frac{\mathrm{d} \ol\dbP'}{\mathrm{d}\ol\dbP} := \exp\bigg(-\int_0^T\l_t(X_\cd, \a_t(X_\cd)) \cdot \mathrm{d}W_t - {\frac12}\int_0^T |\l_t(X_\cd, \a_t(X_\cd))|^2 \mathrm{d}t\bigg).
\end{equation*}
Then $W^\a$ is an $\ol\dbP'-$Brownian motion and 
\begin{equation*}
\mathrm{d}X_t = b'_t(X_\cd, \a_t(X_\cd)) \mathrm{d}t + \si'_t(X_\cd, \a_t(X_\cd)) \mathrm{d}W^\a_t,\;\ol\dbP'-\mbox{a.s.},
\end{equation*}
that is to say $\ol \dbP' \in \ol \cP'(\a)$. Similarly one may construct $\ol\dbP\in \ol\cP(\a)$ from $\ol\dbP'\in \ol\cP'(\a)$, which implies that there is a one--to--one correspondence between $\ol\cP(\a)$ and $\ol\cP'(\a)$ through Girsanov transformations. 

\vspace{0.5em}
{\it Step 2.} We now turn to the backward problem. Since the solution of \eqref{BSDE} is $\dbF^X-$measurable, then, by embedding them into the enlarged canonical space $\ol\O$, we have
\begin{equation*}
Y^{\a,\dbP}_t = \xi(X) + \int_t^T f_s(X_\cd, Y^{\a,\dbP}_s, \ol Z^{\a,\dbP}_s, \a_s(X_\cd)) \mathrm{d}s - \int_t^T \ol Z_s^{\a,\dbP}\cdot \mathrm{d}W_s + N^{\a,\dbP}_T - N^{\a,\dbP}_t,\; \ol\dbP-\mbox{a.s.}
\end{equation*}
where $\ol Z^{\a,\dbP}_t =   \si_t(X_\cd, \a_t(X_\cd))Z^{\a,\dbP}_t$. This implies
\begin{equation*}
Y^{\a,\dbP}_t = \xi'(X) + \int_t^T f'_s(X_\cd, Y^{\a,\dbP}_s, \ol Z^{\a,\dbP}_s, \a_s(X_\cd)) \mathrm{d}s - \int_t^T \ol Z_s^{\a,\dbP}\cdot \mathrm{d}W^\a_s + N^{\a,\dbP}_T - N^{\a,\dbP}_t,\; \ol\dbP-\mbox{a.s.}
\end{equation*}
Notice that $\ol\dbP'$ is equivalent to $\ol\dbP$, so that the last decomposition also holds $\ol\dbP'$-a.s. Denote $\dbP' := \ol\dbP'\circ X^{-1}$. 
By the uniqueness of the solution to the BSDE, we see that $Y^{\a,\dbP} = Y^{\a,\dbP'}$.  Since $\cP'(\a) = \{\dbP' := \ol\dbP'\circ X^{-1}:  \ol\dbP' \in \ol\cP'(\a)\}$ and recalling from Step 1 that $\ol\cP(\a)$ and $\ol\cP'(\a)$ are in a one--to--one correspondence, we see that $\ol J'_0(\a) = \ol J_0(\a)$ and $\ul J_0'(\a) = \ul J_0(\a)$. 
\qed

\vspace{0.5em}

\begin{remark}
\label{rem-Girsanov}
{\rm (i)} While very natural, the above result relies heavily on our formulation that $\a$ $($and $\l)$ depends on $X$ only. When $\a$ $($or $\l)$ depends on $W$, the one to one correspondence in Step 1 above fails, and $\a(W) \neq \a(W^\a)$, thus the game values may not be equal under the Girsanov transformation, even though the Hamiltonians remain the same. See Subsection \ref{sect:strong} where the upper and lower values of the corresponding game are not related to the solutions to the corresponding Isaacs equations.

{\rm (ii)} Notice that the transformation in \reff{Girsanov} changes the map $f$ as well. Given Assumption \ref{assump:2bsde} {\rm (i)}, if we apply the above transformation, then the convex sets in Assumption \ref{assump:2bsde} {\rm (ii)} reduces to
\[
\big\{\big(\sigma_t(x,a_0,a), f_t(x, a_0, a)+ z \cd \l_t(x, a_0, a)\big):a\in A_1\big\},\; \big\{\big(\sigma_t(x,a,a_1), f_t(x,a,a_1)+ z \cd \l_t(x, a, a_1)\big):a\in A_0\big\}.
\]
Hence, Proposition \ref{prop-Girsanov} does not help simplifying Assumption \ref{assump:2bsde} {\rm (ii)}.
\end{remark}

\subsection{State independent range of controls}

In this subsection, we relax Assumption \ref{assum-bsi}, and assume the following.

 \begin{assumption}
 \label{assum-stateindependent}
{\rm (i)} $b = b_0(t, \o, a) + \si(t,\o, a) \l(t,\o,a)$, where $\l$ is bounded, $\dbF-$progressively measurable, uniformly continuous in $(t, \o)$ under $d_\infty$, uniformly in $a \in A$, in the sense of Assumption \ref{assum-f} {\rm (iii)}. 
\\
{\rm (ii)} Player $1$ has state independent range of controls with respect to $(b_0, \si)$ in the sense that for any $t\in [0, T], a_0\in A_0$, the range ${\color{black}\mathbf{R}}_t(a_0):=\{(b_0, \si)(t,\o, a_0, a_1): a_1\in A_1\}$ is independent of $\o$. 
\\
{\rm (iii)} Player $0$ has state independent range of controls with respect to $(b_0, \si)$ in a similar sense.
\end{assumption}

Notice that Assumption \ref{assum-bsi} obviously implies Assumption \ref{assum-stateindependent}.  We next provide a non--trivial example  satisfying Assumption \ref{assum-stateindependent}. For simplicity, we shall only focus on $\si$ and verify $(ii)$.
\begin{example}
\label{eg-stateindependent}
%
%
Let $d=1$, $A_0=A_1=\dbR$, $\eta \in {\rm UC}(\Th)$, {\color{black}and $\ul \si, \ol \si: \dbR\to (0, 1)$  satisfy $\ul \si < \ol\si$, $\lim_{x\to -\infty} \ul \si(x) = \lim_{x\to -\infty} \ol \si(x) = 0$, $\lim_{x\to \infty} \ul \si(x) = \lim_{x\to \infty} \ol \si(x) = 1$ $($for instance, we could take $\ul\si$ to be the cdf of the standard normal distribution, and $\ol \si(x) := \ul \si(x+1))$. }
Define
\begin{align}
\label{si}
&\si_t(\o, a) := \big(\ul \si(a_0) + \ul \si(a_1)\big) \vee \big( \eta_t(\o)+ a_0+a_1\big) \wedge  \big(\ol \si(a_0) + \ol \si(a_1)\big).
\end{align}
For any $t, \o, a_0$,  one may check straightforwardly that 
\begin{align*}
\inf_{a_1\in A_1} \si_t(\o, a_0, a_1) = \ul \si(a_0),\; \sup_{a_1\in A_1} \si_t(\o, a_0, a_1) =   \ol \si(a_0)+1.
\end{align*}
That is, $\mathbf{R}_t(a_0) = (\ul \si(a_0),  \ol \si(a_0)+1)$ $($the $\si$ part only$)$ is independent of $\o$.

\vspace{0.5em}
 
 Moreover, we verify that 
 \begin{equation*}
\inf_{a_0\in A_0} \sup_{a_1\in A_1} \si_t(\o, a) = \inf_{a_0\in A_0} [\ol \si(a_0)+1] = 1,\;
\sup_{a_1\in A_1}\inf_{a_0\in A_0}  \si_t(\o, a)= \sup_{a_1\in A_1} \ul \si(a_1)  = 1.
\end{equation*}
Then the Isaacs's condition $\inf_{a_0\in A_0} \sup_{a_1\in A_1} \si^2_t(\o, a) \g = \sup_{a_1\in A_1}\inf_{a_0\in A_0}  \si^2_t(\o, a) \g$ is immediately checked for $\g\ge 0$. One can similarly verify the Isaacs's condition for $\g<0$.
\end{example}

Our main result which generalises Theorem \ref{thm-gamevalue} is given below.
\begin{theorem}
\label{thm-gamevalue2} 
Let Assumptions \ref{assum-coef}, \ref{assum-stateindependent}, and Isaacs condition \eqref{Isaacs} hold. Then

\vspace{0.5em}
{ $(i)$} The following path--dependent Isaacs equation has a viscosity solution $u\in \mbox{\rm UC}_b(\Theta,\dbR)$
 \begin{equation}
 \label{PPDE1}
 -\pa_t u - H_t(\o, u, \pa_\o u, \pa^2_{\o\o} u) =0,\; t<T, \; u_T = \xi.
 \end{equation}
{$(ii)$} Assume further that uniqueness of viscosity solution for the above {\rm PPDE} holds in the class $ \mbox{\rm UC}_b(\Theta,\dbR)$. Then $\ol V_0=\ul V_0=u_0(0)$.
 \end{theorem}

\proof First, by Proposition \ref{prop-Girsanov}, it suffices to prove the theorem in the case $\l = 0$.  Thus in this proof we assume the two players have state independent range of controls with respect to $(b, \si)$. We shall focus on the upper value process $\ol V_t(\o)$ and follow the arguments in Section \ref{sect-Vt}.

\vspace{0.5em}

{\color{black}
Fix $(t,\o)\in \Th$. In this case \reff{sit}  becomes:
\begin{equation}
\label{sito}
b^{t,\o}(s, \tilde \o, a) := b((t+s)\wedge T, \o\otimes _t \tilde \o, a),\; \si^{t, \o}(s, a) := \si((t+s)\wedge T, \o\otimes _t \tilde \o,  a), \; (s,\tilde \o, a)\in \Th \times A.
\end{equation}
For $\a\in \ol\cA$, let $\cP(t,\o,\a)$ denote the set of weak solutions to the SDE
$$
X_s = \int_0^s b^{t,\o}(r, X_\cd, \a_r(X_\cd)) dr + \int_0^s \si^{t,\o}(r, X_\cd, \a_r(X_\cd)) dW_r.
$$
We emphasise that in this case $\cP$ depends on $\o$. Define $\cA_t(\o)$ in an obvious way. Then \reff{Vt} becomes
\begin{equation}\label{Vt2}
 \ol V_t(\o) :=  \inf_{\alpha^0\in\cA_t^{0}(\o)}S_t(\o, \a^0),\q S_t(\o, \a^0):= \sup_{\a^1\in \cA^1_t(\o)} \sup_{\dbP\in \cP(t,\o, \a^0, \a^1)}Y^{t,\o, \a^0,\a^1, \dbP}_0,
 \end{equation}
where
\begin{align*}
 Y^{t,\o, \a,\dbP}_s =&\ \xi^{t,\o} + \int_s^{T-t} \Big(f^{t,\o}_r\big(X_\cd, Y^{t,\o,\a,\dbP}_r, \si^{t,\o}_r(X_\cd, \a_r)Z_r^{t,\o,\a,\dbP}, \a_r\big) 
 + Z^{t,\o,\a,\dbP}_r b^{t,\o}_r(X_\cd, \a_r)\Big)\mathrm{d}r \\
 &-\int_s^{T-t} Z_s^{t,\o,\a,\dbP}\mathrm{d}X_r- \int_s^{T-t} \mathrm{d}N^{t,\o,\a,\dbP}_s,\; \dbP-\mbox{a.s.}
\end{align*}
Notice that, for fixed $(\a, \dbP)$, by BSDE arguments one can easily show that $\o\longmapsto Y^{t,\o, \a,\dbP}_0$ is uniformly continuous. However, the sets $\cP(t,\o, \a)$ and $\cA_t(\o)$ may depend on $\o$, and thus in general we cannot fix $(\a,\dbP)$ for different $\o$. This causes the main difficulty  for obtaining the desired regularity of $\ol V$ and $S$.

\vspace{0.5em}
We shall use Assumption  \ref{assum-stateindependent} (ii) to get around this difficulty.  As in Section \ref{sect-Vt}, we restrict $\a^0$ to $\cA^{0, {\rm pc}}_t$. We emphasise that $\cA^{0, {\rm pc}}_t$ does not depend on $\o$ and $\cA^{0, {\rm pc}}_t \subset \cA^0_t(\o)$ for all $\o$. We then modify \reff{Vpc} as
 \begin{equation}
 \label{Vpc2}
  \ol V^{{\rm pc}}_t(\o) 
 \;:=\;  \inf_{\alpha^0\in\cA_t^{0,{\rm pc}}} S_t(\o, \a^0).
 \end{equation}
 Fix $\a^0\in\Ac^{0, {\rm pc}}_t$, define
\[\Ac^1_t\big(\a^0\big):=\big\{\big(\tilde b,\tilde\sigma\big)\in\mathbb L^0(\Theta)\times\mathbb L^0(\Theta):\big(\tilde b_s(\tilde\omega),\tilde\sigma_s(\tilde\omega)\big)\in\mathbf{R}_{t+s}\big(\a_s^0(\tilde\omega)\big),\; (s, \tilde\o)\in\Theta \big\}.\]
We emphasise that, by Assumption  \ref{assum-stateindependent} (ii), $\Ac^1_t\big(\a^0\big)$ does not depend on $\o$. For each $(\tilde b, \tilde \si)\in  \Ac^1_t(\a^0)$,  let $\Pc(\tilde b,\tilde\sigma)$ denote the set of weak solutions of the SDE
\[\widetilde X_s=\int_0^s\tilde b_r\big(\widetilde X_\cdot\big)\mathrm{d}r+\int_0^s\tilde \sigma_r\big(\widetilde X_\cdot\big)\mathrm{d}W_r.\]
One can check straightforwardly that $\cup_{(\tilde b, \tilde \si)\in  \Ac^1_t(\a^0)} \tilde \cP(\tilde b, \tilde \si) =  \cup_{\a^1\in \cA^1_t(\o)}  \cP(t,\o, \a^0, \a^1)$ for all $\o\in\O$. Moreover, denote
$A_1(s,\tilde b,\tilde\sigma,a_0) := \{a_1\in A_1: (\tilde b, \tilde \si) \in  \mathbf{R}_s(a_0,a_1)\}$, and
\[\tilde f_s^{t,\omega, \tilde b, \tilde \si}(\tilde\o,y,z,a_0):=\underset{a_1\in A_1(t+s,\tilde b_s(\tilde\o),\tilde\sigma_s(\tilde\o),a_0)}{\sup}\; f_s^{t,\omega}(\tilde\o,y,z,a_0,a_1). \]
Then, by the comparison principle for BSDEs, it is immediate to verify that
\begin{equation}
\label{Spc}
S_t(\o, \a^0):= \sup_{(\tilde b, \tilde \si)\in  \Ac^1_t(\a^0)}\sup_{\dbP\in \tilde \cP(\tilde b, \tilde \si)}  \tilde Y^{t,\o, \tilde b, \tilde \si, \dbP}_0,
\end{equation}
where
\begin{align*} 
\tilde Y^{t,\o, \tilde b, \tilde \si, \dbP}_s =&\ \xi^{t,\o} + \int_s^{T-t} \Big(\tilde f^{t,\o}_r\big(X_\cd, Y^{t,\o,\tilde b, \tilde \si,\dbP}_r, \tilde \si_r(X_\cd)Z_r^{t,\o,\tilde b, \tilde\si,\dbP}, \a^0_r(X_\cd)\big) 
 + Z^{t,\o,\tilde b, \tilde\si,\dbP}_r \tilde b_r(X_\cd)\Big)\mathrm{d}r\\
 & - \int_s^{T-t}Z_s^{t,\o,\tilde b, \tilde\si,\dbP}\mathrm{d}X_r- \int_s^{T-t} \mathrm{d}N^{t,\o,\tilde b, \tilde\si,\dbP}_s,\; \dbP-\mbox{a.s.}
\end{align*}
Now for each $(\tilde b, \tilde \si)\in  \Ac^1_t(\a^0)$ and $\dbP\in \tilde \cP(\tilde b, \tilde \si)$, which do not depend on $\o$, by standard BSDE arguments one can show that $\o\longmapsto  \tilde Y^{t,\o, \tilde b, \tilde \si, \dbP}_0$ is uniformly continuous.  Then as in Lemma \ref{lem-Vreg} we see that $S$ and $\ol V^{\rm pc}$ are uniformly continuous in $\o$. The rest of the proof follows then from the arguments in Section \ref{sect-Vt}, combined with standard BSDE arguments. We leave the details to interested readers.
\qed
}

\vspace{0.5em}
As mentioned in Remark \ref{rem-gamevalue} $(iii)$, the approach in Ekren and Zhang \cite{ekren2016pseudo} can be used in this context to identify sufficient conditions for $(ii)$ to hold in Theorem \ref{thm-gamevalue2}. However, we note that the definition of viscosity solution is slightly different in  \cite{ekren2016pseudo}. Rigorously speaking, if we want to apply the results of  \cite{ekren2016pseudo} to conclude the existence of game value in Theorem 5.6, we need to verify that $\ul V^{\rm pc}$ and $\ol V^{\rm pc}$ are viscosity solutions in the sense of  \cite{ekren2016pseudo}. This is done in  \cite{ekren2016pseudo}, but using the formulation of strategy against control as in Subsection \ref{sect:fleming1989existence}, exactly due to the regularity issue. As we saw, Assumption  \ref{assum-stateindependent} enables us to overcome the regularity difficulty, and thus we can apply the arguments in \cite{ekren2016pseudo} to our context. We leave the details to interested readers.

\section{Proof of Theorems \ref{thm:2bsde} and \ref{thm:2bsde2}}
\label{sect:2BSDE}
\setcounter{equation}{0}
Throughout this section Assumptions \ref{assum-coef} and \ref{assump:2bsde} are in force.

\subsection{A relaxed formulation}

To establish the wellposedness of 2BSDEs, we shall apply the results of Possama\"i, Tan and Zhou \cite{possamai2015stochastic}, which relie on the dynamic programming principle in  El Karoui and Tan \cite{karoui2013capacities,karoui2013capacities2} (see also Nutz and van Handel \cite{nutz2013constructing}). However, we shall note that the weak formulation considered in \cite{karoui2013capacities,karoui2013capacities2} is different from the feedback controls in this paper. So our first goal is to establish the equivalence between these two formulations.\footnote{Most of the arguments here are from discussions with Xiaolu Tan, who we thank warmly.}

\vspace{0.5em}
{\color{black}We now fix $\a^0 \in \cA^0_0$, and denote $\ol\f_t(\o, a_1) := \f_t(\o, \a^0_t(\o), a_1)$ for $\f=b, \si, f$.}  The weak formulation in \cite[Section 1.2, pages 7--9]{karoui2013capacities2} consists in working on a fixed canonical space for both the controlled process and the associated controls.  Let $\mathcal C([0,T],\mathbb R^d)$ be the canonical space of continuous functions on $[0,T]$ with values in $\mathbb R^d$, and let $\mathbb A$ be the collection of all finite and positive Borel measures on $[0,T]\times A_1$, whose projection on $[0,T]$ is the Lebesgue measure. In other words, every $q\in\mathbb A$ can be disintegrated as $q(ds,da)=q_s(da)ds$, for an appropriate kernel $q_s$. The weak formulation requires to consider a subset of $\mathbb A$, namely the set $\mathbb A_0$ of all $q\in\mathbb A$ such that the kernel $q_s$ is of the form $\delta_{\phi_s}(da)$ for some Borel function $\phi$. We then define the canonical space $\Omega:= \mathcal C([0,T],\mathbb R^d)\times\mathbb A$, with canonical process $(X,\Lambda)$, where
\[X_t(\omega,q):=\omega(t),\; \Lambda(\omega,q):=q,\; (t,\omega,q)\in[0,T]\times\Omega.\]
The associated canonical filtration is defined by $\mathbb F:=(\mathcal F_t)_{t\in[0,T]}$ where
\[\mathcal F_t:=\sigma\Big((X_s,\Delta_s(\varphi)),\; (s,\varphi)\in[0,t]\times C_b([0,T]\times A_1)\Big),\; t\in [0,T],\]
where $C_b([0,T]\times A_1)$ is the set of bounded continuous functions on $[0,T]\times A_1$, and $\Delta_s(\varphi):=\int_0^s\int_A\varphi(r,a)\Lambda(\mathrm{d}r,\mathrm{d}a)$, for all $(s,\varphi)\in[0,T]\times C_b([0,T]\times A_1)$.
We next define the following set of measures on $(\Omega,\Fc_T)$:
\[\mathcal P:=\Big\{\P:M(\varphi)\text{ is an $(\P,\mathbb F)-$martingale for all $\varphi\in C^2_b(\mathbb R^d)$, and $\mathbb P[X_0=x_0,\Lambda\in\mathbb A_0]=1$}\Big\},\]
where $C^2_b(\mathbb R^d)$ is the set of bounded twice continuously differentiable functions with bounded derivatives, and
\[
M_s(\varphi):=\varphi(X_s)-\int_0^s\int_{A_1}\bigg(\ol b_r(X_\cdot,a)\cdot D\varphi(X_r)+\frac12{\rm Tr}\big[(\ol\sigma\ol\sigma^\top)_r(X_\cdot,a)D^2\varphi(X_r)\big]\bigg)\Lambda(\mathrm{d}r,\mathrm{d}a).\]
{\color{black}The associated weak formulation of the control problem is then defined by
\[
V_{\rm w}:=\sup_{\P\in\Pc} J_w (\dbP)\q \mbox{where}\q J_w(\dbP) := \mathbb E^\mathbb P\bigg[g(X_\cdot) + \int_0^T \int_{A_1} \ol f_t(X_\cd, a) \Lambda(\mathrm{d}t, \mathrm{d}a)\bigg].
\]

\begin{lemma}
\label{lem-weak}
$\cP = \cP^0_0(\a^0)$ and $V_{\rm w} = \sup_{\P\in\Pc_0^1(\a^0)}\E^{\P^\a}\Big[\ol{Y}_0^{\a^0}\Big]$.
\end{lemma}

\begin{proof} By the requirement $\Lambda \in  \mathbb A_0$, $\dbP\in \cP$ amounts to say there exist a process $\phi^\dbP$, possibly in an enlarged space, and a Brownian motion $W^\dbP$, such that 
\[
X_t= x_0+\int_0^t \ol b_s\big(X_\cdot,\phi_s^\P\big)\mathrm{d}s+\int_0^t\ol\sigma_s\big(X_\cdot,\phi_s^\P\big)\mathrm{d}W_s^\P,\; \P-{\rm a.s.},\; \text{and}\; J_w(\dbP) := \mathbb E^\mathbb P\Big[g(X_\cdot) + \int_0^T  \ol f_t(X_\cd, \phi^\dbP_t) \mathrm{d}t\Big].
\]
On the other hand, the set $ \cP^0_0(\a^0)$ corresponds to those $\dbP$ such that $\phi^\dbP$ is $\dbF^X-$progressively measurable. Then clearly $\cP^0_0(\a^0) \subset \cP$.  Now fix $\dbP\in \cP$. Apply the classical results of Wong \cite[Theorem 4.2]{wong1971representation}, we obtain the existence of another $\P-$Brownian motion $\widetilde W^\P$ such that
\[
    X_t= x_0+\int_0^t\widetilde b_s^\P \mathrm{d}s+\int_0^t\widetilde\sigma_s^\P \mathrm{d}\widetilde W_s^\P,\;  \P-{\rm a.s.},\; \text{and}\; J_w(\dbP) := \mathbb E^\mathbb P\Big[g(X_\cdot) + \int_0^T \widetilde f_t^\P  \mathrm{d}t\Big]
\]
where $\widetilde b_s^\P:=\mathbb E^\P\big[\ol b_s\big(X_\cdot,\phi_s^\P\big)\big|\Fc_s^X\big]$, $\widetilde\sigma_s^\P:=\mathbb E^\P\big[\ol\sigma_s\big(X_\cdot,\phi_s^\P\big)\big|\Fc_s^X\big]$, and $\widetilde f_s^\P:=\mathbb E^\P\big[ \ol f_s\big(X_\cdot,\phi_s^\P\big)\big|\Fc_s^X\big]$. Using the fact that range of $(\ol b, \ol \si, \ol f)$ is assumed to be convex, there exists an $\F^X-$progressively measurable process $\widetilde\alpha^\P$ such that $\widetilde b_s^\P=\ol b_s\big(X_\cdot, \tilde \a_s^\P\big)$, $\widetilde \sigma_s^\P=\ol\sigma_s\big(X_\cdot, \tilde \a_s^\P\big)$, and $\widetilde f_s^\P=\ol f_s\big(X_\cdot, \tilde \a_s^\P\big)$. This implies that $\dbP\in \Pc^0_0(\a^0)$ and $J_w(\dbP) = J(\a^0, \tilde \a^\dbP)$, thus inducing the required result. 
\end{proof}
}

\subsection{Proof of Theorem \ref{thm:2bsde}}  
We will only prove the result for the 2BSDE \eqref{eq:2bsde4}, the remaining proof is similar. We first address the well--posedness by verifying the conditions of Possama\"i, Tan and Zhou \cite{possamai2015stochastic}. We introduce the dynamic version $\Pc^0_0(\a^1)(t,\omega)$ of the set $\Pc^0_0(\a^1)$, by considering the same {\rm SDE} on $[t,T]$ starting at time $t$ from the path $\omega\in\Omega$.

\vspace{0.5em}
We first verify that the family $\{\Pc^0_0(\a^1)(t,\omega),\; (t,\omega)\in[0,T]\times\Omega\}$ is saturated, in the terminology of \cite[Definition 5.1]{possamai2015stochastic}, {\it i.e.} for all $\P^1\in\Pc^0_0(\a^1)(t,\omega)$, and $\P^2\sim\P^1$ under which $X$ is an $\P^2-$martingale, we must have $\P^2\in\Pc^0_0(\a^1)(t,\omega)$. To see this, notice that the equivalence between $\P^1$ and $\P^2$ implies that the quadratic variation of $X$ is not changed by passing from $\P^1$ to $\P^2$. Hence the required result.

\vspace{0.5em}
Since $\sigma$ and $b$ are bounded, it follows from the definition of admissible controls that $\ul F$ satisfies the integrability and Lipschitz continuity assumptions required in \cite{possamai2015stochastic}. We also directly check from the fact that $f$ is bounded, together with \cite[Lemma 6.2]{soner2011martingale} that $\sup_{\P\in\Pc^0_0(\a^1)}
 \E^{\P}\Big[{\rm essup}^{\P}_{0\leq t\leq T}
                 \big(\E^{\P}\big[\int_0^T\big|\ul F(0,\widehat\sigma_s^2)\big|^\kappa \mathrm{d}s
                                        \big| \Fc_{t}^+
                                 \big]
                 \big)^{p/\kappa}
         \Big]
 < \infty$, for some $p>\kappa\ge1$.

\vspace{0.5em}
Then, the dynamic programming requirements of \cite[Assumption 2.1]{possamai2015stochastic} follow from the more general results given in El Karoui and Tan \cite{karoui2013capacities,karoui2013capacities2} (see also Nutz and van Handel \cite{nutz2013constructing}), thanks to Lemma \ref{lem-weak}. Finally, since $\xi$ is bounded, the required well--posedness result is a direct consequence of \cite[Lemma 6.2]{soner2011martingale} together with \cite[Theorems 4.1 and 5.1]{possamai2015stochastic}.

\vspace{0.5em}
Now, the representation for $\ul V_0$ is immediate, see for instance the similar proof in \cite[Proposition 4.6]{cvitanic2015dynamic}.
\qed

\begin{remark}
\label{rem-2BSDE3}
Let us investigate the {\rm 2BSDEs} further under additional regularity of the solution. In particular, this will provide a formal justification of the fact that \reff{K0}  implies \eqref{Isaacs}.  For this purpose, we extend and abuse slightly our earlier notations. Omitting $\o$ as usual, we define
\begin{eqnarray}
\label{Aaextend}
&\Sigma_t(a):=\big(\sigma_t\sigma_t^\top\big)(a),\; {\bf \Sigma}^1_t(a_0):=\big\{\Sigma_t(a_0,a_1),\; a_1\in A_1\big\},\; {\bf \Sigma}^0_t(a_1):=\big\{\Sigma_t(a_0,a_1),\; a_0\in A_0\big\},&\nonumber\\[0.3em]
&A_0(t,\Sigma,a_1):=\big\{a_0\in A_0: \big(\sigma_t\sigma_t^\top\big)(a_0,a_1)=\Sigma\big\},\; A_1(t,\Sigma,a_0):=\big\{a_1\in A_1: \big(\sigma_t\sigma_t^\top\big)(a_0,a_1)=\Sigma\big\},&\\[0.3em]
&\ol F_t(z,\Sigma,a_0):=\sup_{a_1\in A_1(t,\Sigma,a_0)} F_t(z, a_0,a_1),\; \ul F_t(z,\Sigma,a_1):=\inf_{a_0\in A_0(t,\Sigma,a_1)}F_t(z, a_0,a_1).&\nonumber
\end{eqnarray}
 Then one can check straightforwardly that
\begin{equation}
\label{eq:hamil}
\ol H_t(z,\gamma)=\inf_{a_0\in A_0}\sup_{\Sigma\in\mathbf{\Sigma}^1_t(a_0)}\Big\{\frac12{\rm Tr}\big[\Sigma\gamma\big]+\ol F_t(z,\Sigma,a_0)\Big\},\;
 \ul H_t(z,\gamma)=\sup_{a_1\in A_1}\inf_{\Sigma\in\mathbf{\Sigma}^0_t(a_1)}\Big\{\frac12{\rm Tr}\big[\Sigma\gamma\big]+\ul F_t(z,\Sigma,a_1)\Big\}.
 \end{equation}
Assume that the processes $K$ in the definition of the {\rm2BSDEs} is absolutely continuous with respect to the Lebesgue measure (see the formal discussion in {\rm\cite[pp. 21-22]{possamai2015stochastic}}, as well as rigorous arguments in the simpler setting of $G-$expectations in {\rm\cite{peng2014complete}}), and can be written as
\begin{align*}
\frac{\mathrm{d}\ul K^{\a^1}_t}{\mathrm{d}t}&= \frac12{\rm Tr}\big[\widehat{\sigma}_t^2\ul\Gamma_t^{\a^1}\big]+\ul F_t\big(\ul Z_t^{ \a^1},\widehat{\sigma}_t^2,\a^1_t\big)-\inf_{\Sigma \in {\mathbf{\Sigma}^0_t}(\a^1)}\bigg\{\frac12{\rm Tr}\big[\Sigma\ul\Gamma_t^{ \a^1}\big]+\ul F_t\big(\ul Z_t^{ \a^1},\Sigma,\a^1_t\big)\bigg\},\\
\frac{\mathrm{d}\ol K^{\a^0}_t}{\mathrm{d}t}&= \sup_{\Sigma \in {\mathbf{\Sigma}^1_t}(\a^0)}\bigg\{\frac12{\rm Tr}\big[\Sigma\ol\Gamma_t^{ \a^0}\big]+\ol F_t\big(\ol Z_t^{\a^0},\Sigma,\a^0_t\big)\bigg\}-\frac12{\rm Tr}\big[\widehat{\sigma}_t^2\ol\Gamma_t^{\a^0}\big]-\ol F_t\big(\ol Z_t^{ \a^0},\widehat{\sigma}_t^2, \a^0_t\big),
\end{align*}
for some predictable processes $\ul\Gamma^{\a^1}$ and $\ol\Gamma^{\a^0}$. Now given \reff{K0}, \reff{eq:2bsde4} and \reff{eq:2bsde5} reduce to the same {\rm BSDE} under $\widehat \dbP$. Then $\ul Y^{\widehat \a^1} = \ol Y^{\widehat \a^0} =: \widehat Y$,  $\ul Z^{\widehat \a^1} = \ol Z^{\widehat \a^0} =: \hat Z$,  $\widehat\dbP-${\rm a.s.} This would imply further that $\ul \Gamma^{\widehat \a^1} = \ol \Gamma^{\widehat \a^0} =: \widehat \Gamma$.  Then by \reff{K0} again we have
\begin{align*}
\ul H_t({\widehat Z}_t, {\widehat \Gamma}_t)  
&\ge \inf_{\Sigma \in {\mathbf{\Sigma}^0_t}(\widehat\a^1)}\bigg\{\frac12{\rm Tr}\big[\Sigma\ul {\widehat \Gamma}_t\big]+\ul F_t\big({\widehat Z}_t,\Sigma,{\widehat \a^1}_t\big)\bigg\}= \frac12{\rm Tr}\big[\widehat{\sigma}_t^2{\widehat \Gamma}_t\big]+\ul F_t\big({\widehat Z}_t,\widehat{\sigma}_t^2,{\widehat \a^1}_t\big)\\
&= \frac12{\rm Tr}\big[\widehat{\sigma}_t^2{\widehat \Gamma}_t\big]+\ol F_t\big({\widehat Z}_t,\widehat{\sigma}_t^2,{\widehat \a^1}_t\big) = \sup_{\Sigma \in {\mathbf{\Sigma}^1_t}(\widehat\a^0)}\bigg\{\frac12{\rm Tr}\big[\Sigma{\widehat \Gamma}_t\big]+\ol F_t\big({\widehat Z}_t,\Sigma,\widehat\a^0_t\big)\bigg\} \ge \ol H_t({\widehat Z}_t, {\widehat \Gamma}_t).
\end{align*}
This implies \reff{Isaacs} at $({\widehat Z}_t, {\widehat \Gamma}_t)$, and we see that $(\widehat \a^0, \widehat \a^1)$ is a saddle point of the Hamiltonian. 
\end{remark}

\begin{remark}
\label{rem-mixed}
In the spirit of relaxed controls, we may reformulate our game problem by using the notion of mixed strategies, exactly as in S\^irbu {\rm\cite{sirbu2014martingale}}. For $i=0, 1$, let $\cP(A_i)$ denote the set of probability measures on $A_i$, and $m_i: \Th \longrightarrow \cP(A_i)$ be $\dbF-$measurable, $i=0,1$. Let $\dbP$ be a weak solution of the {\rm SDE}
\begin{equation}
\label{Xm}
 X_t =\dis \int_0^t [mb]_s(X_\cd)  \mathrm{d}s  +  \int_0^t [m(\si\si^\top)]^{1/2}_s(X_\cd) \mathrm{d}W_s.
\end{equation}
where $[m \f]_t(\o,\l) := \int_A \f_t(\o, \l, a) m_0(t, \o, da_0) m_1(t, \o, \mathrm{d}a_1)$, for any function $\f_t(\o, \l, a)$ with $\l\in\R^d$. Denote
\begin{equation*}
J_0(m, \dbP) := \dbE^\dbP\bigg[ \xi(X_\cd)+ \int_0^T [m f]_t(X_\cd)  \mathrm{d}t\bigg].
\end{equation*}
Then we can introduce the zero--sum game in the setting in an obvious manner. The advantage of this formulation is that Isaacs's condition always holds $\ol H'_t(\o, z, \g) = \ul H'_t(\o, z, \g)$, where
\begin{equation*}
\dis \ol H'_t(\o, z, \g)  := \inf_{m_0\in \cP(A_0)} \sup_{m_1\in \cP(A_1)} [m h]_t(\o, z, \g),~
\dis  \ul H'_t(\o, z, \g)  := \sup_{m_1\in \cP(A_1)}\inf_{m_0\in \cP(A_0)} [m h]_t(\o, z, \g),
\end{equation*}
are the randomised versions of the upper and lower Hamiltonians. It would be interesting to extend our results to this formulation which does not require Isaacs's condition to hold. See also the contribution of Buckdahn, Li and Quincampoix {\rm\cite{buckdahn2014value}}, who considered a setting similar to Buckdahn and Li {\rm\cite{buckdahn2008stochastic}}, but where the players see each other’s actions with a delay relative to a fixed time grid, and both play mixed delayed strategies. Notice that we always have $\ul H \le \ul H' = \ol H' \le \ol H$, so that when the standard Isaacs's condition \eqref{Isaacs} holds, they are all equal.
\end{remark}

\subsection{Proof of Theorem \ref{thm:2bsde2}}
The wellposedness and the representation for $\ol V_0$ are proved as in the previous section. For $\a^1\in\ol\Ac^1$, the 2BSDE
\[ \ul Y^{\a^1}_t
 =
 \xi+\int_t^T\ul g_s\big(\ul Z_s^{\a^1},\widehat{\sigma}^2_s,\a^1_s\big)\mathrm{d}s
 -\int_t^T\ul Z^{\a^1}_s\cdot \mathrm{d}X_s-\int_t^T\mathrm{d} \ul K^{\a^1}_s,
 ~\P-\mbox{a.s. for all}~\P\in\Pc.\]
induces the following representation 
\[\inf_{\P\in\Pc}\E^{\P}\Big[\ul Y_0^{\a^1}\Big]=\inf_{\a^0\in \cA^0_0}\ul J_0(\alpha).\]
Then, the comparison theorem for 2BSDEs, see \cite[Theorem 4.3]{possamai2015stochastic} implies that $\ul Y_t^{\a_1}\leq Y_t$. By an obvious extension of the argument of El Karoui, Peng and Quenez \cite[Corollary 3.1]{el1997backward} to 2BSDEs, we deduce the desired result. Finally, the existence of a value is now immediate when $\ul G=\ol G$.

\section{Appendix: proof of Theorem \ref{thm-qsSDE}}
\label{sect:Appendix}
\setcounter{equation}{0}

We start with the continuous coefficients setting, where the result looks standard. We nevertheless provide a detailed proof for completeness.

\begin{lemma}
\label{lem-contSDE}
Assume $b:\Th \longrightarrow \dbR^d$ and $\si: \Th\longrightarrow \dbS^d$ are bounded, $\dbF-$measurable, and for each $t$, $b(t,\cd)$ and $\si(t,\cd)$ are continuous. Then the {\rm SDE} of Theorem \ref{thm-qsSDE} has a weak solution
\end{lemma}

\proof Assume $W$ is an $\dbP_0-$Brownian motion. For $n\ge 1$, denote $t_i:= {\frac in}T$, $i=0,\cds, n$, and define
\begin{equation*}
X^n_{t_0}:= 0,\; X^n_t := X^n_{t_i} + \int_{t_i}^t b(s, X^n_{\cd\wedge t_i}) \mathrm{d}s +  \int_{t_i}^t \si(s, X^n_{\cd\wedge t_i}) \mathrm{d}W_s, \; t\in [t_i, t_{i+1}],\; i=0,\cds, n-1.
\end{equation*}
Define $\dbP_n:= \dbP_0 \circ (X^n)^{-1}$.  Then $\dbP_n \subset \cP_L$ for $L$ large enough, and $\dbP_n$ is a weak solution to the SDE of Theorem \ref{thm-qsSDE} with the coefficients $b_n(t,\o) :=  b(t, \o_{t^n_i\wedge \cd}),$ and $\si_n(t,\o) :=  \si(t, \o_{t^n_i\wedge \cd})$, $t\in [t_i, t_{i+1}]$, $i=0,\cds, n-1$.
Note that $b_n,\; \si_n$ are uniformly bounded. By Zheng \cite{zheng1985tightness}, 
$\{\dbP_n\}_{n\ge 1}$ has a weakly convergent subsequence, and for notational simplicity we assume $\dbP_n \longrightarrow \dbP$ weakly. Now it suffices to verify that $\dbP$ is a weak solution to the SDE of Theorem \ref{thm-qsSDE}. For this purpose, we first recall that by Zhang \cite[Lemmata 9.2.4 (i) and (9.2.18)]{zhang2017backward}
\begin{equation}
\label{compact}
\left.\begin{array}{c}
\dis \mbox{there exist $\{E_m\}_{m\ge 1} \subset \cF_T$ such that each $E_m$ is compact and $\sup_{\dbP\in \cP_L} \dbE[E_m^c] \le 2^{-m}$,}\\[0.6em]
\dis \mbox{and for each $m$ and each $\o\in E_m$, we have $\o_{t\wedge \cd}\in E_m$ for all $t\in [0, T]$.}
\end{array}\right.
\end{equation}
 Denote 
\begin{equation}
\label{MN}
M_t := X_t - \int_0^t b(s, X_\cd) \mathrm{d}s,\; N_t := M_t M_t^\top - \int_0^t \si^2(s, X_\cd) \mathrm{d}s,
\end{equation}
and define $M^n, N^n$ by replacing $(b, \si)$ above with $(b_n, \si_n)$. Then it is equivalent to prove that $\P$ is a weak solution to the SDE of Theorem \ref{thm-qsSDE}, and that $M$ and $N$ are $\dbP-$martingales.  First, for any $s<t$ and any $\eta\in C^0_b(\cF_s)$, by the definition of $\dbP_n$, we have $\dbE^{\dbP_n}[ (M^n_t-M^n_s)\eta_s] = 0$. Note that
$M_t-M_s = X_t - X_s - \int_s^t b(r, X_\cd) \mathrm{d}r$. Since $b(r, \cd)$ is continuous for each $r$, by the weak convergence of $\dbP_n$, together with the bounded convergence theorem (under the Lebesgue measure), we have
\begin{equation*}
\lim_{n\to \infty} \dbE^{\dbP_n}\bigg[\eta_s \int_s^t b(r, X_\cd) \mathrm{d}r\bigg] =  \dbE^{\dbP}\bigg[\eta_s\int_s^t b(r, X_\cd)\mathrm{d}r\bigg].
\end{equation*}
Moreover, for any $R >0$, denote by $I_R(x)$ the truncation of $x$ by $R$ and $X_{s,t}:=X_t-X_s$, we have
\begin{align*}
&\lim_{n\to \infty} \dbE^{\dbP_n}\Big[ I_R(X_{s,t}) \eta_s\Big] =  \dbE^{\dbP}\Big[ I_R(X_{s,t}) \eta_s\Big],
\sup_{\dbP'\in \cP_L} \dbE^{\dbP'} \Big[|(X_{s,t}) - I_R(X_{s,t})|^2\Big] 
\le 
{\frac1R} \sup_{\dbP'\in \cP_L} \dbE^{\dbP'} \Big[|X_{s,t}|^3\Big] \le {\frac CR}. 
\end{align*}
Denoting similarly $M_{s,t}:=M_t-M_s$ and $M^n_{s,t}:=M^n_t-M^n_s$, wee see that
\begin{align}
\label{EPM}
 \dbE^{\dbP}\big[ M_{s,t}\eta_s\big] 
   &
   = 
   \lim_{n\to \infty} \dbE^{\dbP_n}\big[ M_{s,t}\eta_s\big] 
   = 
   \lim_{n\to \infty} \dbE^{\dbP_n}\big[ \big(M_{s,t}-M^n_{s,t}\big)\eta_s\big]
   \nonumber\\
 & =\lim_{n\to \infty}  \dbE^{\dbP_n}\Big[ \eta_s \int_s^t [b(r, X_\cd) - b_n(r, X_\cd)]\mathrm{d}r\Big] =  \lim_{n\to \infty} \int_s^t \dbE^{\dbP_n}\big[ \eta_s  \big(b(r, X_\cd) - b_n(r, X_\cd)\big)\big]\mathrm{d}r .
\end{align}
Now for each $r\in [s, t]$ and $m\ge 1$, since $b(r,\cd)$ is continuous and $E_m$ is compact, $b(r,\cd)$ is uniformly continuous on $E_m$ with a certain modulus of continuity function $\rho_{r,m}$. Then, by the definition of $b_n$ and \eqref{compact}
\begin{align*}
&\Big|\dbE^{\dbP_n}\Big[ \eta_s \big(b(r, X_\cd) - b_n(r, X_\cd)\big)\Big]\Big| \le C\dbE^{\dbP_n}\Big[ |b(r, X_\cd) - b_n(r, X_\cd)|\1_{E_m}\Big] + C\dbP_n[E_m^c]\\
&\le  C\dbE^{\dbP_n}\Big[\rho_{r,m}\big(\mbox{OSC}_{2^{-n}}(X)\big)\Big] + C2^{-m} \le  C\dbE^{\dbP_n}\Big[\rho_{r,m}\big(\mbox{OSC}_{2^{-k}}(X)\big)\Big] + C2^{-m},
\end{align*}
for any $k \le n$, where $\mbox{OSC}_\delta (X) := \sup_{t_1, t_2: |t_1-t_2|\le \delta} |X_{t_1}-X_{t_2}|$. Fix $m, k$ and send $n\longrightarrow \infty$, we deduce
\begin{equation*}
\limsup_{n\to\infty}\Big|\dbE^{\dbP_n}\Big[ \eta_s \big(b(r, X_\cd) - b_n(r, X_\cd)\big)\Big]\Big|\le C\dbE^{\dbP}\Big[\rho_{r,m}\big(\mbox{OSC}_{2^{-k}}(X)\big)\Big] + C2^{-m}.
\end{equation*}
By first sending $k\longrightarrow\infty$ and then $m\longrightarrow \infty$, we have $\lim_{n\to\infty}\dbE^{\dbP_n}\big[ \eta_s \big(b(r, X_\cd) - b_n(r, X_\cd)\big)\big] =0$, and by the bounded convergence theorem, it follows from \eqref{EPM} that $\dbE^{\dbP}[ M_{s,t}\eta_s]=0$, i.e. $M$ is an $\dbP-$martingale. Similarly one can show that $N$ is an $\dbP-$martingale. Therefore, $\dbP$ is a weak solution to the SDE of Theorem \ref{thm-qsSDE}.
\qed

\vspace{0.5em}

\no {\bf Proof of Theorem \ref{thm-qsSDE}}
For $\eps>0$, let $\O^\eps_t$ be a common set for $(b, \si)$ as in Definition \ref{qscont}.  By Zhang \cite[Problem 10.5.3]{zhang2017backward}, there exists $(b_\eps, \si_\eps)$ such that $(b_\eps, \si_\eps)(t,\cd)$ agree with $(b, \si)(t, \cd)$ on $\O^\eps_t$ and are continuous for each $t$. Moreover, by the construction in \cite[Problem 10.5.3]{zhang2017backward}, it follows from the progressive measurability in Definition \ref{qscont} $(i)$, that $b_\eps,$ $\si_\eps$ are $\dbF-$progressively measurable. By Lemma \ref{lem-contSDE}, let $\dbP_\eps$ be a weak solution to the SDE of Theorem \ref{thm-qsSDE} with coefficients $(b_\eps, \si_\eps)$. Similarly to Lemma \ref{lem-contSDE}, there exists $\eps_n \longrightarrow 0$ such that $\dbP_{\eps_n}$ converges to some $\dbP\in \cP$ weakly.  Recall \eqref{MN} and define $M^{\eps}$ and $N^{\eps}$ by replacing $(b, \si)$ above with $(b_\eps, \si_\eps)$. Then, for any $s<t$ and $\eta_s \in C^0_b(\cF_s)$, following similar arguments as in Lemma \ref{lem-contSDE} we have, for any $m\ge 1$,
$$
\dbE^{\dbP}\big[M^{\eps_m}_{s,t}\eta_s\big] 
= 
\lim_{n\to \infty} \dbE^{\dbP_{\eps_n}}\big[M^{\eps_m}_{s,t}\eta_s\big] 
= \lim_{n\to \infty} \dbE^{\dbP_{\eps_n}}\big[\big(M^{\eps_m}_{s,t} -M^{\eps_n}_{s,t}\big) \eta_s\big]
=  \lim_{n\to \infty} \dbE^{\dbP_{\eps_n}}\Big[\eta_s \int_s^t [b_{\eps_m} - b_{\eps_n}](r, X_\cd)dr\Big].
$$
Thus
\begin{align*}
\big| \dbE^{\dbP}\big[M_{s,t}\eta_s\big]\big| 
&\le 
\big| \dbE^{\dbP}\big[M^{\eps_m}_{s,t}\eta_s\big]\big| 
+ \big| \dbE^{\dbP}\big[\eta_s\big(M_{s,t}- M^{\eps_m}_{s,t}\big)\big]\big|\\
&\le  C\liminf_{n\to\infty} \dbE^{\dbP_{\eps_n}}\bigg[\int_s^t |[b_{\eps_m} - b_{\eps_n}](r, X_\cd)|\mathrm{d}r\bigg] + C\dbE^{\dbP}\bigg[\int_s^t |[b_{\eps_m} - b](r, X_\cd)|\mathrm{d}r\bigg] \\
&\le C\liminf_{n\to\infty} \dbE^{\dbP_{\eps_n}}\bigg[\int_s^t [\1_{(\O^{\eps_n}_r)^c}+\1_{(\O^{\eps_m}_r)^c}]\mathrm{d}r\bigg]  + C\dbE^{\dbP}\bigg[\int_s^t \1_{(\O^{\eps_m}_r)^c}\mathrm{d}r\bigg]
\;\le\; C\eps_m.
\end{align*}
Since $m$ is arbitrary, we have $\dbE^{\dbP}\big[M_{s,t}\eta_s\big]=0$ for all $\eta_s\in C^0_b(\cF_s)$. That is, $M$ is an $\dbP-$martingale. Similarly, $N$ is an $\dbP-$martingale, so that $\dbP$ is a weak solution to the SDE of Theorem \ref{thm-qsSDE}.
\qed

  \bibliographystyle{plain}
  \small
\bibliography{bibliographyDylan}

\begin{thebibliography}{10}

\bibitem{barlow1982one}
M.~Barlow.
\newblock One dimensional stochastic differential equations with no strong
  solution.
\newblock {\em Journal of the London Mathematical Society}, 26(2):335--347,
  1982.

\bibitem{barron1984viscosity}
E.N. Barron, L.C. Evans, and R.~Jensen.
\newblock Viscosity solutions of {I}saacs' equations and differential games
  with {L}ipschitz controls.
\newblock {\em Journal of Differential Equations}, 53(2):213--233, 1984.

\bibitem{bayraktar2012stochastic}
E.~Bayraktar and M.~S{\^\i}rbu.
\newblock Stochastic {P}erron's method and verification without smoothness
  using viscosity comparison: the linear case.
\newblock {\em Proceedings of the American Mathematical Society},
  140(10):3645--3654, 2012.

\bibitem{bayraktar2013stochastic}
E.~Bayraktar and M.~S{\^\i}rbu.
\newblock Stochastic {P}erron's method for {H}amilton--{J}acobi--{B}ellman
  equations.
\newblock {\em SIAM Journal on Control and Optimization}, 51(6):4274--4294,
  2013.

\bibitem{bichteler1981stochastic}
K.~Bichteler.
\newblock Stochastic integration and ${L}^p-$theory of semimartingales.
\newblock {\em The Annals of Probability}, 9(1):49--89, 1981.

\bibitem{bouchard2014stochastic}
B.~Bouchard, L.~Moreau, and M.~Nutz.
\newblock Stochastic target games with controlled loss.
\newblock {\em The Annals of Applied Probability}, 24(3):899--934, 2014.

\bibitem{bouchard2015stochastic}
B.~Bouchard and M.~Nutz.
\newblock Stochastic target games and dynamic programming via regularized
  viscosity solutions.
\newblock {\em Mathematics of Operations Research}, 41(1):109--124, 2016.

\bibitem{buckdahn2011some}
R.~Buckdahn, P.~Cardaliaguet, and M.~Quincampoix.
\newblock Some recent aspects of differential game theory.
\newblock {\em Dynamic Games and Applications}, 1(1):74--114, 2011.

\bibitem{buckdahn2008stochastic}
R.~Buckdahn and J.~Li.
\newblock Stochastic differential games and viscosity solutions of
  {H}amilton--{J}acobi--{B}ellman--{I}saacs equations.
\newblock {\em SIAM Journal on Control and Optimization}, 47(1):444--475, 2008.

\bibitem{buckdahn2014value}
R.~Buckdahn, J.~Li, and M.~Quincampoix.
\newblock Value in mixed strategies for zero--sum stochastic differential games
  without {I}saacs condition.
\newblock {\em The Annals of Probability}, 42(4):1724--1768, 2014.

\bibitem{cardaliaguet2009stochastic}
P.~Cardaliaguet and C.~Rainer.
\newblock Stochastic differential games with asymmetric information.
\newblock {\em Applied Mathematics \& Optimization}, 59(1):1--36, 2009.

\bibitem{cardaliaguet2013pathwise}
P.~Cardaliaguet and C.~Rainer.
\newblock Pathwise strategies for stochastic differential games with an erratum
  to "{S}tochastic differential games with asymmetric information''.
\newblock {\em Applied Mathematics \& Optimization}, 68(1):75--84, 2013.

\bibitem{cheridito2007second}
P.~Cheridito, H.M. Soner, N.~Touzi, and N.~Victoir.
\newblock Second--order backward stochastic differential equations and fully
  nonlinear parabolic {PDE}s.
\newblock {\em Communications on Pure and Applied Mathematics},
  60(7):1081--1110, 2007.

\bibitem{crandall1983viscosity}
M.G. Crandall and P.-L. Lions.
\newblock Viscosity solutions of {H}amilton--{J}acobi equations.
\newblock {\em Transactions of the American Mathematical Society},
  277(1):1--42, 1983.

\bibitem{cvitanic2015dynamic}
J.~Cvitani{\'c}, D.~Possama{\"\i}, and N.~Touzi.
\newblock Dynamic programming approach to principal--agent problems.
\newblock {\em Finance and Stochastics}, 22(1):1--37, 2018.

\bibitem{davis1973dynamic}
M.H.A. Davis and P.~Varaiya.
\newblock Dynamic programming conditions for partially observable stochastic
  systems.
\newblock {\em SIAM Journal on Control and Optimization}, 11(2):226--261, 1973.

\bibitem{dupire2009functional}
B.~Dupire.
\newblock Functional {I}t{$\rm \bar{o}$} calculus.
\newblock Technical Report 2009--04--FRONTIERS, Bloomberg portfolio research
  paper, 2009.

\bibitem{ekren2016viscosity}
I.~Ekren, N.~Touzi, and J.~Zhang.
\newblock Viscosity solutions of fully nonlinear parabolic path dependent
  {PDE}s: part {I}.
\newblock {\em The Annals of Probability}, 44(2):1212--1253, 2016.

\bibitem{ekren2012viscosity}
I.~Ekren, N.~Touzi, and J.~Zhang.
\newblock Viscosity solutions of fully nonlinear parabolic path dependent
  {PDE}s: part {II}.
\newblock {\em The Annals of Probability}, 44(4):2507--2553, 2016.

\bibitem{ekren2016pseudo}
I.~Ekren and J.~Zhang.
\newblock Pseudo {M}arkovian viscosity solutions of fully nonlinear degenerate
  {PPDE}s.
\newblock {\em Probability, Uncertainty and Quantitative Risk}, 1(6), 2016.

\bibitem{el1997backward}
N.~El~Karoui, S.~Peng, and M.-C. Quenez.
\newblock Backward stochastic differential equations in finance.
\newblock {\em Mathematical Finance}, 7(1):1--71, 1997.

\bibitem{karoui2013capacities}
N.~El~Karoui and X.~Tan.
\newblock Capacities, measurable selection and dynamic programming part {I}:
  abstract framework.
\newblock {\em arXiv preprint arXiv:1310.3363}, 2013.

\bibitem{karoui2013capacities2}
N.~El~Karoui and X.~Tan.
\newblock Capacities, measurable selection and dynamic programming part {II}:
  application in stochastic control problems.
\newblock {\em arXiv preprint arXiv:1310.3364}, 2013.

\bibitem{elliot1972existence}
R.J. Elliot and N.J. Kalton.
\newblock {\em The existence of value in differential games}.
\newblock Number 126 in Memoirs of the American Mathematical Society. American
  Mathematical Society, Providence, Rhode Island, 1972.

\bibitem{elliott1976existence}
R.J. Elliott.
\newblock The existence of value in stochastic differential games.
\newblock {\em SIAM Journal on Control and Optimization}, 14(1):85--94, 1976.

\bibitem{elliott1977existence}
R.J. Elliott.
\newblock The existence of optimal strategies and saddle points in stochastic
  differential games.
\newblock In P.~Hagedorn, H.W. Knobloch, and G.J. Olsder, editors, {\em
  Differential games and applications. Proceedings of a workshop, Enschede
  1977}, volume~3 of {\em Lecture notes in control and information sciences},
  pages 123--135. Springer--Verlag Berlin Heidelberg New York, 1977.

\bibitem{elliott1981optimal}
R.J. Elliott and M.H.A. Davis.
\newblock Optimal play in a stochastic differential game.
\newblock {\em SIAM Journal on Control and Optimization}, 19(4):543--554, 1981.

\bibitem{evans1984differential}
L.C. Evans and H.~Ishii.
\newblock Differential games and nonlinear first order {PDE} on bounded
  domains.
\newblock {\em Manuscripta Mathematica}, 49(2):109--139, 1984.

\bibitem{evans1984differential2}
L.C. Evans and P.E. Souganidis.
\newblock Differential games and representation formulas for solutions of
  {H}amilton--{J}acobi--{I}saacs equations.
\newblock {\em Indiana University Mathematics Journal}, 33:773--797, 1984.

\bibitem{fleming1957note}
W.H. Fleming.
\newblock A note on differential games of prescribed duration.
\newblock In M.~Dresher, A.W. Tucker, and P.~Wolfe, editors, {\em Contributions
  to the theory of games}, volume III of {\em Annals of mathematics studies},
  pages 407--412. Princeton University Press, 1957.

\bibitem{fleming1961convergence}
W.H. Fleming.
\newblock The convergence problem for differential games.
\newblock {\em Journal of Mathematical Analysis and Applications}, 8:102--116,
  1961.

\bibitem{fleming1964convergence}
W.H. Fleming.
\newblock The convergence problem for differential games {II}.
\newblock In M.~Dresher, L.S. Shapley, and A.W. Tucker, editors, {\em Advances
  in game theory}, Annals of mathematics studies, pages 195--210. Princeton
  University Press, 1964.

\bibitem{fleming2011value}
W.H. Fleming and D.~Hern{\'a}ndez-Hern{\'a}ndez.
\newblock On the value of stochastic differential games.
\newblock {\em Communications on Stochastic Analysis}, 5(2):341--351, 2011.

\bibitem{fleming1989existence}
W.H. Fleming and P.E. Souganidis.
\newblock On the existence of value--functions of two--player, zero--sum
  stochastic differential--games.
\newblock {\em Indiana University Mathematics Journal}, 38(2):293--314, 1989.

\bibitem{friedman1970definition}
A.~Friedman.
\newblock On the definition of differential games and the existence of value
  and of saddle points.
\newblock {\em Journal of Differential Equations}, 7(1):69--91, 1970.

\bibitem{friedman1972stochastic}
A.~Friedman.
\newblock Stochastic differential games.
\newblock {\em Journal of Differential Equations}, 11(1):79--108, 1972.

\bibitem{friedman1971differential}
A.~Friedman.
\newblock {\em Differential games}, volume XXV of {\em Pure and applied
  mathematics}.
\newblock Wiley--Interscience, 2013.

\bibitem{hamadene1995backward}
S.~Hamad\`ene and J.-P. Lepeltier.
\newblock Backward equations, stochastic control and zero--sum stochastic
  differential games.
\newblock {\em Stochastics: An International Journal of Probability and
  Stochastic Processes}, 54(3-4):221--231, 1995.

\bibitem{hamadene1995zero}
S.~Hamad\`ene and J.-P. Lepeltier.
\newblock Zero--sum stochastic differential games and backward equations.
\newblock {\em Systems \& Control Letters}, 24(4):259--263, 1995.

\bibitem{hamadene1997bsdes}
S.~Hamad\`ene, J.-P. Lepeltier, and S.~Peng.
\newblock {BSDE}s with continuous coefficients and stochastic differential
  games.
\newblock In N.~El~Karoui and L.~Mazliak, editors, {\em Backward stochastic
  differential equations}, volume 364 of {\em Chapman \& Hall/CRC Research
  Notes in Mathematics Series}, pages 115--128. Longman, 1997.

\bibitem{hernandez2018moral}
N.~Hern{\'a}ndez~Santib{\'a}{\~n}ez and T.~Mastrolia.
\newblock Contract theory in a {VUCA} world.
\newblock {\em arXiv preprint arXiv:1803.08951}, 2018.

\bibitem{isaacs1954differential}
R.~Isaacs.
\newblock Differential games {III}: the basic principles of the solution
  process.
\newblock Technical Report RM--1411--PR, RAND Corporation, 1954.

\bibitem{isaacs1965differential}
R.~Isaacs.
\newblock {\em Differential games: a mathematical theory with applications to
  warfare and pursuit, control and optimization}.
\newblock John Wiley \& Sons Inc., 1965.

\bibitem{kovats2009value}
J.~Kovats.
\newblock Value functions and the {D}irichlet problem for {I}saacs equation in
  a smooth domain.
\newblock {\em Transactions of the American Mathematical Society},
  361(8):4045--4076, 2009.

\bibitem{krylov1980controlled}
N.V. Krylov.
\newblock {\em Controlled diffusion processes}, volume~14 of {\em Stochastic
  modelling and applied probability}.
\newblock Springer--Verlag New York, 1980.

\bibitem{krylov2013dynamic}
N.V. Krylov.
\newblock On the dynamic programming principle for uniformly nondegenerate
  stochastic differential games in domains.
\newblock {\em Stochastic Processes and their Applications}, 123(8):3273--3298,
  2013.

\bibitem{krylov2014dynamic}
N.V. Krylov.
\newblock On the dynamic programming principle for uniformly nondegenerate
  stochastic differential games in domains and the {I}saacs equations.
\newblock {\em Probability Theory and Related Fields}, 158(3-4):751--783, 2014.

\bibitem{mastrolia2015moral}
T.~Mastrolia and D.~Possama{\"\i}.
\newblock Moral hazard under ambiguity.
\newblock {\em Journal of Optimization Theory and Applications}, to appear.

\bibitem{nisio1988stochastic}
M.~Nisio.
\newblock Stochastic differential games and viscosity solutions of {I}saacs
  equations.
\newblock {\em Nagoya Mathematical Journal}, 110:163--184, 1988.

\bibitem{nutz2013constructing}
M.~Nutz and R.~van Handel.
\newblock Constructing sublinear expectations on path space.
\newblock {\em Stochastic Processes and their Applications}, 123(8):3100--3121,
  2013.

\bibitem{peng2015g}
S.~Peng and Y.~Song.
\newblock ${G}-$expectation weighted {S}obolev spaces, backward {SDE} and path
  dependent {PDE}.
\newblock {\em Journal of the Mathematical Society of Japan}, 67(4):1725--1757,
  2015.

\bibitem{peng2014complete}
S.~Peng, Y.~Song, and J.~Zhang.
\newblock A complete representation theorem for ${G}-$martingales.
\newblock {\em Stochastics: An International Journal of Probability and
  Stochastic Processes}, 86(4):609--631, 2014.

\bibitem{pham2014two}
T.~Pham and J.~Zhang.
\newblock Two person zero--sum game in weak formulation and path dependent
  {B}ellman--{I}saacs equation.
\newblock {\em SIAM Journal on Control and Optimization}, 52(4):2090--2121,
  2014.

\bibitem{possamai2015stochastic}
D.~Possama{\"\i}, X.~Tan, and C.~Zhou.
\newblock Stochastic control for a class of nonlinear kernels and applications.
\newblock {\em The Annals of Probability}, 46(1):551--603, 2018.

\bibitem{ren2015comparison}
Z.~Ren, N.~Touzi, and J.~Zhang.
\newblock Comparison of viscosity solutions of fully nonlinear degenerate
  parabolic path--dependent {PDE}s.
\newblock {\em SIAM Journal on Mathematical Analysis}, 49(5):4093--4116, 2017.

\bibitem{roxin1969axiomatic}
E.~Roxin.
\newblock Axiomatic approach in differential games.
\newblock {\em Journal of Optimization Theory and Applications}, 3(3):153--163,
  1969.

\bibitem{sirbu2014martingale}
M.~S{\^\i}rbu.
\newblock On martingale problems with continuous--time mixing and values of
  zero--sum games without the {I}saacs condition.
\newblock {\em SIAM Journal on Control and Optimization}, 52(5):2877--2890,
  2014.

\bibitem{sirbu2014stochastic}
M.~S{\^\i}rbu.
\newblock Stochastic {P}erron's method and elementary strategies for zero--sum
  differential games.
\newblock {\em SIAM Journal on Control and Optimization}, 52(3):1693--1711,
  2014.

\bibitem{sirbu2015asymptotic}
M.~S{\^\i}rbu.
\newblock Asymptotic {P}erron's method and simple {M}arkov strategies in
  stochastic games and control.
\newblock {\em SIAM Journal on Control and Optimization}, 53(4):1713--1733,
  2015.

\bibitem{soner2011martingale}
H.M. Soner, N.~Touzi, and J.~Zhang.
\newblock Martingale representation theorem for the ${G}-$expectation.
\newblock {\em Stochastic Processes and their Applications}, 121(2):265--287,
  2011.

\bibitem{soner2012wellposedness}
H.M. Soner, N.~Touzi, and J.~Zhang.
\newblock Wellposedness of second order backward {SDE}s.
\newblock {\em Probability Theory and Related Fields}, 153(1-2):149--190, 2012.

\bibitem{song2012uniqueness}
Y.~Song.
\newblock Uniqueness of the representation for ${G}-$martingales with finite
  variation.
\newblock {\em Electronic Journal of Probability}, 17(24):1--15, 2012.

\bibitem{souganidis1983approximation}
P.E. Souganidis.
\newblock {\em Approximation schemes for the viscosity solutions of
  {H}amilton--{J}acobi equations}.
\newblock PhD thesis, University of Wisconsin--Madison, 1983.

\bibitem{sung2015optimal}
J.~Sung.
\newblock Optimal contracting under mean--volatility ambiguity uncertainties.
\newblock {\em {SSRN} preprint 2601174}, 2015.

\bibitem{swikech1996another}
A.~{\'S}wi{\k{e}}ch.
\newblock Another approach to the existence of value functions of stochastic
  differential games.
\newblock {\em Journal of Mathematical Analysis and Applications},
  204(3):884--897, 1996.

\bibitem{varaiya1967existence}
P.~Varaiya.
\newblock On the existence of solutions to a differential game.
\newblock {\em SIAM Journal on Control}, 5(1):153--162, 1967.

\bibitem{varaiya1969existence}
P.~Varaiya and J.~Lin.
\newblock Existence of saddle points in differential games.
\newblock {\em SIAM Journal on Control}, 7(1):141--157, 1969.

\bibitem{wong1971representation}
E.~Wong.
\newblock Representation of martingales, quadratic variation and applications.
\newblock {\em SIAM Journal on Control}, 9(4):621--633, 1971.

\bibitem{zhang2017existence}
F.~Zhang.
\newblock Existence of game value and approximating {N}ash equilibrium for
  path--dependent stochastic differential game.
\newblock In T.~Liu and Q.~Zhao, editors, {\em 36th Chinese control conference,
  Dalian, China, 26--28 July 2017}, pages 295--300. IEEE, 2017.

\bibitem{zhang2017existence2}
F.~Zhang.
\newblock The existence of game value for path--dependent stochastic
  differential game.
\newblock {\em SIAM Journal on Control and Optimization}, 55(4):2519--2542,
  2017.

\bibitem{zhang2017backward}
J.~Zhang.
\newblock {\em Backward stochastic differential equations -- from linear to
  fully nonlinear theory}, volume~86 of {\em Probability theory and stochastic
  modelling}.
\newblock Springer--Verlag New York, 2017.

\bibitem{zheng1985tightness}
W.A. Zheng.
\newblock Tightness results for laws of diffusion processes application to
  stochastic mechanics.
\newblock {\em Annales de l'institut Henri Poincar{\'e}, Probabilit{\'e}s et
  Statistiques $(${\rm B}$)$}, 21(2):103--124, 1985.

\end{thebibliography}

\end{document}